\documentclass[11pt]{amsart}

\usepackage{amssymb,amsfonts,amsmath,amscd} 
\usepackage{amsthm}
\usepackage{stmaryrd}
\usepackage{graphicx,graphics}
\usepackage{latexsym}
\usepackage{color} %for the pictures in xfig with Latex
\usepackage{pstricks}
\usepackage{pst-node}
\usepackage{setspace} % after contents page: \setstretch{1.5} or \doublespacing
\usepackage{verbatim}
\usepackage{multirow}
\usepackage[dvips,pdfstartview={FitH -32768},bookmarks=false,colorlinks=false]{hyperref}

\newtheorem{thm}{Theorem}[section]
\newtheorem{prop}[thm]{Proposition}
\newtheorem{lem}[thm]{Lemma}

\newtheorem{col}[thm]{Corollary}
\newtheorem{prop-defn}[thm]{Proposition-Definition}
\newtheorem{conj}[thm]{Conjeture}

\theoremstyle{definition}
\newtheorem{defn}[thm]{Definition}
\newtheorem{exa}[thm]{Example}
\newtheorem{rem}[thm]{Remark}

\numberwithin{equation}{subsection}

\newcommand{\Span}[1]{\left<#1\right>}
\newcommand{\Hilb}[2]{#1\text{-Hilb}(#2)}
\newcommand{\Mtheta}{\mathcal{M}_{\theta}(Q,R)}
\newcommand{\C}{\mathbb C}
\newcommand{\R}{\mathbb R}

\newcommand{\Z}{\mathbb Z}
\newcommand{\Q}{\mathbb Q}
\newcommand{\PP}{\mathbb P}

\newcommand{\F}{\mathcal F}
\newcommand{\G}{\mathcal G}

\newcommand{\RR}{\mathcal R}
\newcommand{\kk}{\ensuremath{\Bbbk}}

\newcommand{\OO}{\mathcal O}

\DeclareMathOperator{\GL}{GL}
\DeclareMathOperator{\SL}{SL}
\DeclareMathOperator{\rep}{rep}
\DeclareMathOperator{\Irr}{Irr}

\DeclareMathOperator{\Hom}{Hom}

\DeclareMathOperator{\Spec}{Spec}
\DeclareMathOperator{\diag}{diag}

\DeclareMathOperator{\Aut}{Aut}
\DeclareMathOperator{\Coh}{Coh}

\DeclareMathOperator{\D}{D}

\DeclareMathOperator{\SO}{SO}

\DeclareMathOperator{\Orb}{Orb}
\DeclareMathOperator{\Id}{Id}

\DeclareMathOperator{\Supp}{Supp}
\DeclareMathOperator{\Ker}{Ker}
\DeclareMathOperator{\Mat}{Mat}

\title{On $G/N$-Hilb of $N$-Hilb}
\author{Akira Ishii}
\author{Yukari Ito}
\author{\'Alvaro Nolla de Celis}
\address{Department of Mathematics, Graduate School of Science, Hiroshima
University, 1-3-1
Kagamiyama, Higashi-Hiroshima, 739-8526, Japan}
\email{akira@math.sci.hiroshima-u.ac.jp}
\address{Graduate School of Mathematics\\ 
 Nagoya University, Chikusa-ku Nagoya 464-8602 Japan}
\email{y-ito@math.nagoya-u.ac.jp}
\address{Graduate School of Mathematics\\ 
 Nagoya University, Chikusa-ku Nagoya 464-8602 Japan}
\email{alnolla@gmail.com}
\date{}
\thanks{The first author is supported by JSPS grant No.21540039 and No. 23540045.
The second and third authors are supported by JSPS grant No.23540045, No.22224001 and No.09F09768.
The third author is supported by FY2009 JSPS Fellowship for Foreign Researchers and JSPS grant No.09F09768.}


\begin{document}

\maketitle
\begin{abstract}
In this paper we consider the iterated $G$-equivariant Hilbert scheme $G/N$-Hilb($N$-Hilb) and prove that $\Hilb{G/N}{\Hilb{N}{\C^3}}$ is a crepant resolution of $\C^3/G$ isomorphic to the moduli space $\mathcal{M}_\theta(Q)$ of $\theta$-stable representations of the McKay quiver $Q$ for certain stability condition $\theta$. We provide several explicit examples to illustrate this construction. We also consider the problem of when G/N-Hilb(N-Hilb) is isomorphic to G-Hilb showing the fact that these spaces are most of the times different.
\end{abstract}

\setcounter{tocdepth}{1}
\tableofcontents

\section{Introduction}

Let $X$ be a nonsingular quasiprojective complex 3-fold and let $G\subset\Aut X$ be a finite subgroup such that the stabilizer subgroup of any point $x\in X$ acts on the tangent space $T_xX$ as a subgroup of $\SL(T_xX)$. Let $\Hilb{G}{X}$ be the fine moduli space of $G$-clusters and $\mathcal{Z}$ be the universal subscheme. We have the following celebrated theorem of Bridgeland, King and Reid:

\begin{thm}[\cite{BKR}]\label{BKR} $Y=\Hilb{G}{X}$ is irreducible and $f:Y\to X/G$ is a crepant resolution. Furthermore, $\Phi:\D^b(\Coh Y)\longrightarrow\D^b(\Coh^GX)$ is an equivalence of derived categories where $\Phi$ is the Fourier-Mukai transform with kernel $\OO_{\mathcal{Z}}$.
\end{thm}

Our framework is the following: let $G\subset\SL(3,\C)$ be finite and $N\unlhd G$ a normal subgroup. First consider the action of $N$ on $\C^3$ and take the crepant resolution $Y:=\Hilb{N}{\C^3}$. Next act with $G/N$ on $Y$ to obtain $G/N$-Hilb($Y$).

\begin{center}
\begin{pspicture}(0,-1.5)(6,1.8)
\psset{nodesep=4pt}
\rput(0,1.3){\rnode{HoH}{$\Hilb{G/N}{Y}$}}
\rput(2,0){\rnode{YN}{$Y/(G/N)$}}
\rput(4,1.3){\rnode{NN}{$Y=\Hilb{N}{\C^3}$}}
\rput(6,0){\rnode{C3N}{$\C^3/N$}}
\rput(8,1.3){\rnode{C3}{$\C^3$}}
\rput(4,-1.3){\rnode{C3G}{$\C^3/G$}}
\ncline{->}{C3}{C3N}\Aput[0.05]{$\pi_1$}
\ncline{->}{NN}{C3N}\Bput[0.05]{$\tau_1$}
\ncline{->}{NN}{YN}\Aput[0.05]{$\pi_2$}
\ncline{->}{HoH}{YN}\Bput[0.05]{$\tau_2$}
\ncline{->}{C3N}{C3G}\Aput[0.05]{$\phi_1$}
\ncline{->}{YN}{C3G}\Bput[0.05]{$\phi_2$}
\end{pspicture}
\end{center}

As an immediate consequence of Theorem \ref{BKR} we have the following corollary.
\begin{col}\label{cor:crep}
$\Hilb{G/N}{\Hilb{N}{\C^3}}$ is a crepant resolution of $\C^3/G$.
\end{col}

 A similar construction was considered before by the second author in  \cite{Ito94}, \cite{Ito95b} in the case of trihedrals subgroups in $\SL(3,\C)$. 
The trihedral group is a non-Abelian finite subgroup generated by diagonal matrices and the matrix $T:=\left(\begin{smallmatrix}0&1&0\\0&0&1\\1&0&0\end{smallmatrix}\right)$. In this
case, $G= N \rtimes T$ is a semidirect product. Then in Ito's construction, we require that $T=G/N$ acts on the crepant resolution $Y$ of $\C^3/N$ symmetrically on the exceptional locus. Therefore,  this construction gives the $\Hilb{G/N}{\Hilb{N}{\C^3}}$ when  $Y=\Hilb{N}{\C^3}$.

This construction can be extended in a natural way to obtain crepant resolutions of $\C^3/G$ for any finite non-simple group $G\subset\SL(3,\C)$ (see \cite{YY} for a classification of such groups). In general, if we consider the sequence of normal subgroups $N_i$ of the form $N_0:=G$, $N_1:=N \lhd G$, $N_2 \lhd N_0/N_1, \ldots, N_i \lhd (\ldots((G/N_1)/N_2)\ldots)/N_{i-1}$ for $i\geq1$ and $N_i$ is normal in $G$, the iterated equivariant Hilbert scheme 
\[
\text{$N_i$-Hilb($N_{i-1}$-Hilb($\ldots(\Hilb{N_1}{\C^3})\ldots)$}
\] 
described in this paper is crepant. In particular it is always possible to find such a crepant resolution with $N_i$ Abelian for all $i$.

Denote by $\Irr(G)$ the set of irreducible representations of $G$ and let $(Q,R)$ be the McKay quiver of $G$ with relations $R$. For $\textbf{d}=(\dim(\rho))_{\rho\in\Irr(G)}$ and any generic $\theta$ in the space of stability conditions $\Theta$ we can define $\mathcal{M}_{\theta,\textbf{d}}(Q,R)$ to be the moduli space of $\theta$-stable representations of $Q$ satisfying the relations $R$. Moreover, there exists a chamber decomposition of $\Theta$ such that the GIT quotient $\mathcal{M}_{\theta, \mathbf{d}}(Q,R)$ is constant for all $\theta$ in any given (open) chamber $C$.
Thus,we also denote this moduli space simply by $\mathcal{M}_C$. Thus we may also denote this moduli space simply by $\mathcal{M}_C$. The methods in \cite{BKR} can be applied to prove that $\tau:\mathcal{M}_C\longrightarrow\C^3/G$ is a crepant resolution and $\Phi_C:D(\mathcal{M}_C)\longrightarrow\D^G(\C^3)$ is an equivalence of categories (see \cite{CI}, \S 2). 

For these moduli spaces, the problem of whether every (projective) crepant resolution of $\C^3/G$ is a moduli of representations of the McKay quiver was treated by Craw and the first author in the case of Abelian group actions.

\begin{thm}[\cite{CI}]\label{CIAb} For a finite Abelian subgroup $A\subset\SL(3,\C)$ let $Y\rightarrow\C^3/A$ be a projective crepant resolution. Then $Y\cong\mathcal{M}_C$ for some chamber $C\subset\Theta$.
\end{thm}

Then Craw and Ishii proposed the following conjecture.

\begin{conj}\label{CI-conj}For a finite subgroup $G\subset\SL(3,\C)$ let $Y\rightarrow\C^3/G$ be a projective crepant resolution. Then $Y\cong\mathcal{M}_C$ for some chamber $C\subset\Theta$.
\end{conj}

In this paper we show that the projective crepant resolution $\Hilb{G/N}{\Hilb{N}{\C^3}}$ is a fine moduli space for a particular chamber $C\subset\Theta$ as follows:

\begin{thm}\label{thm:main}(= \ref{Stab})  Let $G\subset\SL(3,\C)$ be finite and let $N\unlhd G$ be a normal subgroup. The crepant resolution $\Hilb{G/N}{\Hilb{N}{\C^3}}$ is isomorphic to a moduli space of $G$-constellations $\mathcal{M}_C$ for some chamber $C\subset\Theta$.
\end{thm}

Thus our main result shows that the conjecture holds for the family $\Hilb{G/N}{\Hilb{N}{\C^3}}$ of crepant resolutions for general $G\subset\SL(3,\C)$. 

As we see in this paper, the varieties underlying the fine moduli spaces $\Hilb{G}{\C^3}$ and $\Hilb{G/N}{\Hilb{N}{\C^3}}$ are in general non-isomorphic quasiprojective varieties. Even when they coincide, as moduli spaces of representations of the McKay quiver almost always they belong to different chambers in the space of stability conditions $\Theta$, or in other words, the corresponding tautological vector bundles are not the same. 

It is therefore natural to ask when the iterated Hilbert scheme $\Hilb{G/N}{\Hilb{N}{\C^3}}$ is isomorphic to $\Hilb{G}{\C^3}$. In this paper we give a complete answer for this problem when they are considered as moduli spaces, i.e.\ both the underlying variety and the tautological vector bundle coincide. For the problem of when they are isomorphic as algebraic varieties, we present the list of such cases when the group $G$ is Abelian. For non-Abelian cases, we prove that $\Hilb{G/N}{\Hilb{N}{\C^3}}$ and $\Hilb{G}{\C^3}$ are non-isomorphic varieties when $G$ is a non-Abelian small subgroup of $\GL(2, \C)$ embedded in $\SL(3,\C)$ and $N=G \cap \SL(3,\C)$, and for some polyhedral groups $G$ in $\SO(3)$. These results suggest that the moduli spaces $\mathcal{M}_C$ are actually varying. We summarize the results in this direction in the following theorems.

\begin{thm}\label{thm:isoDim2}(= \ref{IsoModulidim3},\ref{IsoModulidim2},\ref{prop:non-min})
Let $G\subset\GL(2,\C)$ be a finite small subgroup and let $N\neq G,\{1\}$ be a normal subgroup. Let $Y:=\Hilb{G/N}{\Hilb{N}{\C^2}}$. Then,
\begin{itemize}
\item[(i)] If $G\subset\GL(2,\C)$ then $Y\cong\Hilb{G}{\C^2}$ as moduli spaces if and only if $G\cong\frac{1}{rs}(1,1)$ and $N\cong\frac{1}{s}(1,1)$ for some $r,s\geq2$.
\item[(ii)] If $G\subset\SL(2,\C)$ then $Y\cong\Hilb{G}{\C^2}$ as algebraic varieties.
\item[(iii)] If $G\not\subset\SL(2,\C)$ is non-Abelian and $N=G\cap\SL(2,\C)$ then $Y$ and $\Hilb{G}{\C^2}$ are non-isomorphic as algebraic varieties.
\end{itemize}
\end{thm}

\begin{thm}\label{thm:isoDim3}(= \ref{IsoModulidim3},\ref{IsoVarAb},\ref{col:non-iso})
Let $G\subset\SL(3,\C)$ be a finite small subgroup and let $N\neq G,\{1\}$ be a normal subgroup. Let $Y:=\Hilb{G/N}{\Hilb{N}{\C^3}}$. Then,
\begin{itemize}
\item[(i)] $Y\cong\Hilb{G}{\C^3}$ as moduli spaces if and only if $G\cong\frac{1}{2r}(1,1,2r-2)$ and $N\cong\frac{1}{2}(1,1,0)$.
\item[(ii)] If $G$ is Abelian, then $Y\cong\Hilb{G}{\C^3}$ as algebraic varieties if and only if we are in one of the following situations:
	\begin{enumerate}
	\item $G/N\cong\Z/m\Z\times\Z/m\Z$ for some $m>1$.
	\item $G\cong\frac{1}{r}(1,1,r-2)$ or $G\cong\frac{1}{r}(1,r-1,0)$, i.e.\ $\C^3/G$ has a 	unique crepant resolution.
	\item $G\cong\frac{1}{2r}(1,a,-a-1)$ with $(2r,a)=1$, $a^2\equiv1$ (mod $4r$) and 	$N\cong\frac{1}{2}(1,1,0)$.
	\item There is a subgroup $G'\subset G$ containing $N$ such that $(G',N)$ fits into 	either $(2)$ or $(3)$ and $G/G'\cong\Z/m\Z\times\Z/m\Z$ for some $m>1$.
	\end{enumerate}
\item[(iii)] If $G\subset\SO(3)$ is of type $D_{2n}$ or $G_{12}$ (defined in sections \ref{exaDn} and \ref{exaE6} respectively) and $N$ is the maximal Abelian subgroup, or if $G$ is isomorphic to a non-Abelian finite small subgroup of $\GL(2,\C)$ and $N=G\cap\SL(2,\C)$, then $Y$ and $\Hilb{G}{\C^3}$ are non-isomorphic as algebraic varieties.
\end{itemize}
\end{thm}

This paper is organized as follows. 
In section \ref{Sect:Moduli}, we introduce moduli spaces of $G$-constellations and find the stability for $\Hilb{G/N}{\Hilb{N}{\C^3}}$ to prove Theorem \ref{thm:main}.
In section \ref{semidirect}, we recall representations of semidirect products
and fix some notation used in examples.
In section \ref{AbCase}, we show some examples of $\Hilb{G/N}{\Hilb{N}{\C^3}}$ when $G$ is Abelian. The construction of crepant resolutions for Abelian quotient singularities are well known as toric resolutions where one of them is $\Hilb{G}{\C^3}$. However, if we have several choices for normal subgroups $N$ of $G$, then we have several $\Hilb{G/N}{\Hilb{N}{\C^3}}$ which can be obtained by a finite sequence of flops from $\Hilb{G}{\C^3}$. Moreover, the actions of $G/N$ on $\Hilb{N}{\C^3}$ are also interesting.
In section \ref{LocalCoord}, we introduce the notion of {\it skeletons} to compute local coordinates in examples.
In section \ref{Examples}, we give several examples of $\Hilb{G/N}{\Hilb{N}{\C^3}}$ when $G$ is non-Abelian. We also describe the $G$-constellations and McKay quiver with relations.  
In section \ref{When}, we investigate when $\Hilb{G/N}{\Hilb{N}{\C^3}}$ is isomorphic to  $\Hilb{G}{\C^3}$ as moduli spaces and as algebraic varieties.

The authors thank Alastair King and Laurent Demonet for their valuable comments and Alastair Craw and Michael Wemyss for useful discussions.

\section{The moduli space of $G$-constellations and $G/N$-Hilb($N$-Hilb)}\label{Sect:Moduli}

Recall that $G$-Hilb($\C^n$) is the moduli space of $G$-clusters, where a {\em $G$-cluster} $\mathcal{Z}$ is a $G$-invariant subscheme $\mathcal{Z}\subset\C^n$ such that $H^0(\OO_Z)\cong R_G$ the regular representation of $G$ as $\C[G]$-modules. Thus a point $y\in\Hilb{G/N}{\Hilb{N}{\C^n}}$ is a $G/N$-cluster of $N$-clusters. 

\subsection{A family of $G$-constellations}
The first observation that appears is that $y$ may not be a $G$-cluster. Therefore, in order to construct the moduli space of such objects we need a generalized notion of $G$-cluster called $G$-{\em constellation} (see \cite{CI} and \cite{Craw01}, Chapter 5).

\begin{defn} A {\em $G$-constellation} $\mathcal{F}$ on $X$ is a $G$-equivariant coherent sheaf on $X$ such that $H^0(\mathcal{F})\cong R_G$ as $\C[G]$-modules. 
\end{defn}

Notice that a $G$-cluster is a $G$-equivariant $\C[x_1,\ldots,x_n]$-module $\OO_\mathcal{Z}$ generated from $1$ mod $I_\mathcal{Z}$, which is precisely a $G$-constellation generated from the trivial representation $\rho_0$. 

Equivalently, when $X=\C^n$ a $G$-equivariant coherent sheaf $\mathcal{F}$ on $X$ is a representation of the McKay quiver $Q$ satisfying the relations $R$. This identification was first stated in \cite[\S 3]{IN00} and rewritten in the language of $G$-constellations in \cite[\S 2.1]{CI}. For an explicit description of the relations $R$ see \cite{BSW}. Recall that the McKay quiver is the quiver with $\Irr G$ as vertex set, and $\dim_\C \Hom(\rho,V\otimes\rho')$ arrows from $\rho$ to $\rho'$.

We fix the following notation in this paper.
Denote by $Y_1:=\Hilb{N}{\C^n}$ and $Y_2:=\Hilb{G/N}{Y_1}$. Then we have the diagram

\begin{center}
\begin{pspicture}(0,-0.25)(6,2.25)
\psset{nodesep=4pt}
\rput(1.5,2){\rnode{Z2}{$\mathcal{Z}_2$}}
\rput(4.5,2){\rnode{Z1}{$\mathcal{Z}_1$}}
\rput(0,1){\rnode{HoH}{$Y_2$}}
\rput(1.5,0){\rnode{YN}{$Y_1/(G/N)$}}
\rput(3,1){\rnode{NN}{$Y_1$}}
\rput(4.5,0){\rnode{C3N}{$\C^n/N$}}
\rput(6,1){\rnode{C3}{$\C^n$}}
\ncline{->}{C3}{C3N}\Aput[0.05]{$\pi_1$}
\ncline{->}{NN}{C3N}\Bput[0.05]{$\tau_1$}
\ncline{->}{NN}{YN}\Aput[0.05]{$\pi_2$}
\ncline{->}{HoH}{YN}\Bput[0.05]{$\tau_2$}
\ncline{->}{Z2}{HoH}\Bput[0.05]{$p_2$}
\ncline{->}{Z2}{NN}\Aput[0.05]{$q_2$}
\ncline{->}{Z1}{NN}\Bput[0.05]{$p_1$}
\ncline{->}{Z1}{C3}\Aput[0.05]{$q_1$}
\end{pspicture}
\end{center}
where $\mathcal{Z}_1$ and $\mathcal{Z}_2$ are the universal families for $Y_1$ and $Y_2$ respectively.

\begin{lem}\label{Gconst} Every point in the connected component of $\Hilb{G/N}{\Hilb{N}{\C^n}}$ dominating $\C^3/G$ is a $G$-constellation on $\C^n$.
More precisely, there is a canonically defined flat family of $G$-constellations parametrised by this connected component of $\Hilb{G/N}{\Hilb{N}{\C^n}}$.
\end{lem}

\begin{proof} 
 Consider the fibre product $\mathcal{Z}_2\times_{Y_1}\mathcal{Z}_1\subset Y_2\times Y_1\times\C^n$ and the projection $p_{20}:Y_2\times Y_1\times\C^n\longrightarrow Y_2\times\C^n$ onto the first and third factor. Then $p_{20*}(\OO_{\mathcal{Z}_2\times_{Y_1}\mathcal{Z}_1})$ is a $G$-equivariant coherent sheaf on $Y_2 \times \C^n$, flat over
$Y_2$.
For a closed point $y \in Y_2$, let $\OO_W=q_{2*}p_2^* \OO_y \subset Y_1$ be the corresponding $G/N$-cluster.
Then the fibre of $p_{20*}(\OO_{\mathcal{Z}_2\times_{Y_1}\mathcal{Z}_1})$ over $y$ is
$q_{1*}p_{1}^*\OO_W$.
Especially, if $y$ lies over a free orbit in $\C^n/G$, then $q_{1*}p_{1}^*\OO_W$ is the $G$-cluster
supported by the free orbit and it is the regular representation as a $G$-module.
Since $p_{20*}(\OO_{\mathcal{Z}_2\times_{Y_1}\mathcal{Z}_1})$ is flat over $Y_2$, it is a
flat family of $G$-constellations in the connected component containing free orbits.
%
%Let $Y=\Hilb{N}{\C^3}$ and $W\in\Hilb{G/N}{Y}$ be a $G/N$-cluster in $Y$. Then $\OO_W$ can be viewed as a $G/N$-constellation with $H^0(\OO_W)\cong\C[G/N]$ generated from the trivial representation as a $\C[G/N]$-module. Consider the following diagram:
%\begin{center}
%\begin{pspicture}(0,-0.25)(3,2.25)
%\psset{nodesep=4pt}
%\rput(1,2){$Y\times\C^n$}
%\rput(1,1.5){$\bigcup$}
%\rput(1,1){\rnode{Z}{$\mathcal{Z}$}}
%\rput(-0.5,0){\rnode{Y}{$Y$}}
%\rput(2.5,0){\rnode{C3}{$\C^n$}}
%\ncline{->}{Z}{C3}\Aput[0.05]{\footnotesize $q$}
%\ncline{->}{Z}{Y}\Bput[0.05]{\footnotesize $p$}
%\end{pspicture}
%\end{center}
%where $\mathcal{Z}$ be the universal closed subscheme $\mathcal{Z}\subset Y\times\C^3$ with $\OO_\mathcal{Z}$ the the universal $G$-constellation. Note that $G$ acts on every scheme in the above diagram so $p$ and $q$ are $G$-equivariant maps. Then $p^{-1}(W)\subset\mathcal{Z}$ is a $G$-equivariant subscheme and $\OO_{p^{-1}(W)}=p^*\OO_W$ is $G$-equivariant. We define 
%\[
%\Phi(\OO_W):=q_*(p^*\OO_W)
%\]
%Because the map $p$ is finite and flat, and $p^*\OO_W$ is supported on a $0$-dimensional subscheme, the functor $q_*p^*$ restricted to $G/N$-clusters is exact, thus $\Phi(\OO_W)$ is a $G$-equivariant coherent sheaf.
%
%To conclude that $\Phi(\OO_W)$ is a $G$-constellation it remains to prove that $\Phi(\OO_W)\cong\C[G]$ as $\C[G]$-modules.
\end{proof}
\subsection{Stability for $\Hilb{G/N}{\Hilb{N}{\C^3}}$}

Let $\Theta=\Theta^G:=\{\theta\in\Hom_\Z(R[G],\Q)|\theta(R_G)=0\}$ with $R[G]$ the representation ring. The notion of stability by \cite{King} translates into the language of $G$-constellations as follows. For $\theta\in\Theta$, a $G$-constellation $\mathcal{F}$ is $\theta$-stable (or $\theta$-semistalbe) if $\theta(\mathcal{E})=\theta(H^0(\mathcal{E}))>0=\theta(\mathcal{F})$ (or $\theta(\mathcal{E}) \ge 0$) for $0\subsetneq\mathcal{E}\subsetneq\mathcal{F}$. 
With quiver-theoretic point of view, if $M$ is a representation of $Q$ of dimension vector $\textbf{d}=(d_i)_{i\in Q_0}$ then the notion of stability for $M$ is given as follows: let $\theta\in\Q^{Q_0}$ and define $\theta(M):=\sum\theta_id_i$. Then $M$ is $\theta$-stable (or $\theta$-semistable) if $\theta(M')>0=\theta(M)$ (or $\theta(M') \ge 0$) for $0\subsetneq M'\subsetneq M$.
More generally, $\theta$-stability and semistability are defined for a $G$-equivariant coherent sheaf
$\mathcal{F}$ on $\C^3$ with finite support such that $\theta(H^0(\mathcal{F}))=0$
in the same way.

It is known from \cite{IN00} that $\Hilb{G}{\C^3}$ can be considered as a moduli $\mathcal{M}_\theta$ of $\theta$-stable representations of the McKay quiver of $G$
satisfying the relations, for any stability condition $\theta\in\Theta$ satisfying $\theta(\rho)>0$ for every
non-trivial irreducible representation $\rho$ of $G$
(and hence $\theta(\rho_0)<0$ for the trivial representation $\rho_0$).
We call such $\theta$ a {\em 0-generated} stability. 

\begin{rem}
In general, the chamber of stability parameters defining $G$-Hilb is larger than the cone defined by the inequalities above.
\end{rem}

Let $\theta^N \in \Theta^N$ and $\theta^{G/N}\in \Theta^{G/N}$ be 
{\em \textbf{0}-generated} stabilities for $N$ and $G/N$ respectively.
(In the following argument, $\theta^N$ and $\theta^{G/N}$ can be arbitrary parameters in the chambers of $N$-Hilb and $G/N$-Hilb respectively.
However, we assume they are \textbf{0}-generated to simplify the proof
of Lemma \ref{genericity} below.)

\begin{defn}\label{theta} Let $\rho\in\Irr(G)$ and $\theta\in\Theta$. We define
\[
\theta(\rho):=\left\{
{\renewcommand{\arraystretch}{1.25}
\begin{array}{ll}
\theta^N(\rho|_N)+\varepsilon\cdot\theta^{G/N}(\rho) & \text{ if $\rho\in\Irr(G/N)$} \\
\theta^N(\rho|_N)				 & \text{ if $\rho\notin\Irr(G/N)$} 
\end{array}}
\right.
\]
where $0<\varepsilon<<1$.
\end{defn}

Note that $\theta^N\in \Theta^N$ can be regarded as an element of $\Theta$ by
composing the map $\theta^N: R[N] \to \Q$ with the restriction map $R[G] \to R[N]$.
The condition $\rho\in\Irr(G/N)$ as a representation of $G$ means that $\rho$ is trivial for every element in $N$. It is straightforward to check that $\theta(R_G)=0$ for the regular representation $R_G$ as required.

\begin{lem}\label{genericity}
The parameter $\theta$ defined in Lemma \ref{theta} is generic.
\end{lem}
\begin{proof}
For any non-zero subrepresentation $S \subsetneq R_G$,
$\theta(S) \ne 0$ by the choice of $\theta$.
This implies that $\theta$ is generic.
\end{proof}

From now on we restrict ourselves to the case $G\subset\SL(3,\C)$. 
Consider the functor
$$
\Phi: D^b(\Coh^{G/N} Y_1) \to D^b(\Coh^G \C^3)
$$
defined by $\Phi(-)=\R q_{1*}p_1^*(-)$.
Then it is an equivalence of triangulated categories by \cite{BKR}(see also \cite[Theorem 3.1]{IU}).
Let $\Coh^{G/N}_0(Y_1)$ denote the Abelian category of $G/N$-equivariant coherent sheaves
on $Y_1$ with zero-dimensional supports.
Then $\Phi$ sends objects of $\Coh^G_0(Y_1)$ to $G$-equivariant sheaves with zero-dimensional supports
and $\Phi$ is exact on $\Coh^G_0(Y_1)$.
\begin{lem}\label{semistable}
Let $G \subset \SL(3, \C)$ be a finite subgroup and $N$ be a normal subgroup of $G$.
Let $\theta^N$ be a 0-generated stability parameter for $N$.
Then for an object $E \in \Coh^{G/N}_0{Y_1}$, $\Phi(E) \in \Coh^G\C^3$ is $\theta^N$-semistable.
Moreover, if $\F \subseteq \Phi(E)$ is a $G$-equivariant subsheaf of $\Phi(E)$
with $\theta^N(\F)=0$,
then there is a $G/N$-equivariant subsheaf $F$ of $E$ such that $\F=\Phi(F)$.
\end{lem}
\begin{proof}
Let $F_0:D^b(\Coh^G \C^3) \to D^b(\Coh^N \C^3)$ and $F_1: D^b(\Coh^{G/N} Y_1) \to D^b(\Coh Y_1)$
be the forgetful functors.
We have a commutative diagram
$$
\begin{CD}
D^b(\Coh^{G/N} Y_1) @>\Phi>\sim> D^b(\Coh^G \C^3) \\
@VF_1VV @VVF_0V \\
D^b(\Coh Y_1) @>\sim>\Phi^N> D^b(\Coh^N \C^3)
\end{CD}
$$
where $\Phi^N$ is the functor which is defined in the same way as $\Phi$
and is an equivalence by \cite{BKR}. 
Since $F_1(E)$ has a filtration in $\Coh Y_1$ whose factors are skyscraper sheaves,
$\Phi^N(F_1(E))$ has a filtration in $\Coh^N \C^3$ whose factors are $N$-clusters.
Since $N$-clusters are $\theta^N$-stable, $F_0(\Phi(E))\cong \Phi^N(F_1(E))$ is $\theta^N$-semistable.
Now for any $G$-invariant subsheaf $\F$ of $\Phi(E)$, we have
$\theta^N(\F)=\theta^N(F_0(\F))\ge 0$ by the semistability of $F_0(\Phi(E))$,
which shows that $\Phi(E)$ is $\theta^N$-semistable.

Suppose $\F \subseteq \Phi(E)$ is a $G$-invariant subsheaf with $\theta^N(\F)=0$.
Then $\F$ is also $\theta^N$-semistable by the definition of semistability for
a $G$-equivariant coherent sheaf.
Moreover, $F_0(\Phi(E))$ is also $\theta^N$-semistable as an $N$-equivariant coherent sheaf
as in the previous paragraph and hence so is $F_0(\F)$.
Consider the Jordan-H{\"o}lder filtraions on the $\theta^N$-semistable
$N$-equivariant coherent sheaves $F_0(\Phi(E))$, $F_0(\F)$ and
$F_0(\Phi(E)/F)$ respectively, whose factors are $\theta^N$-stable.
Since the Jordan-H{\"o}lder factors of $F_0(\Phi(E))$ are $N$-clusters,
those of $F_0(\F)$ (and $F_0(\Phi(E)/F)$) are also $N$-clusters by the Jordan-H{\"o}lder theorem
for semistable sheaves.
Note that $N$-clusters are of the form $\Phi^N(\OO_y)$ for $y \in Y_1$.
Therefore, there is a filtration
$$
0=\G_0 \subset \G_1 \subset \dots \subset \G_l=F_0(\F)
$$
such that $\G_i/\G_{i-1} \cong \Phi(\OO_{y_i})$ for some $y_i \in Y_1$.
Since the equivalence $\Phi^N$ induces an isomorphism
$$\Hom(\OO_y, F_1(E)/G) \cong \Hom(\Phi^N(\OO_y), F_0(\Phi(E))/\Phi^N(G))$$
for any closed point $y \in Y_1$ and any subsheaf $G\subset F_1(E)$,
induction on $i$ shows that there is a unique subsheaf $G_i \subset F_1(E)$ such that $\Phi^N(G_i)=\G_i$ for each $i$.
Especially, $G:=G_l$ is a unique subsheaf of $F_1(E)$
such that $F_0(\F)=\Phi^N(G)$.
$G$ must be preserved by the action of $G/N$ by its uniqueness,
which shows that $G$ is of the form $F_1(F)$ for a $G/N$-invariant subsheaf $F \subseteq E$.
\end{proof}

%% and $\Coh^G_{\theta^N}(\C^3)$ the Abelian category
%%of $\theta^N$-semistable $G$-equivariant coherent sheaves with zero-dimensional support.
%\begin{lem}
%If $E$ is a simple object of  $\Coh^{G/N}_0(Y_1)$, then $\Phi(E)$ is $\theta^N$-stable.
%\end{lem}
%\begin{proof}
%Since $E$ is simple, the support of $E$ is a single $G/N$-orbit.
%Let $L\subset G$ be the stabiliser subgroup of a point $P$ in $\Supp E$.
%Then $V:=E|_P$ is a representation of $L$.
%If we identify $\Supp E$ with $G/L$, we can regard $E$ as
%$G \times_L V:=(G\times V)/L$
%where the action of $L$ on $G\times V$ is diagonal.
%The simplicity of $E$ implies that $V$ is irrecucible.
%
%Now we show that $\Phi(E)$ is $\theta^N$-stable.
%If we regard $E$ as a coherent sheaf on $Y_1$, the it has a filtration whose associated
%graded module is the direct sum of skyscraper sheaves.
%Hence, as an $N$-equivariant coherent sheaf on $\C^3$, $\Phi(E)$ has a filtration
%whose associated graded module is the direct sum of $N$-clusters.
%This implies that $\Phi(E)$ is $\theta^N$-semistable.
%Suppose there is a subobject $F \ne 0$ of $\Phi(E)$ which is $\theta^N$-semistable.
%Then, as an $N$-equivariant coherent sheaf, $F$ contains a $N$-cluster corresponding
%to a point in $\Supp(E)$, which we may assume is $P$.

%\end{proof}

\begin{thm}\label{Stab} Let $\theta\in\Theta$ be as in Definition \ref{theta}. Let $G\subset\SL(3,\C)$ be a finite subgroup and $N$ be a normal subgroup of $G$. Then 
\[
\Hilb{G/N}{\Hilb{N}{\C^3}}\cong\mathcal{M}_C
\] 
for the chamber $C\subset\Theta$ which contains $\theta$.
\end{thm}

\begin{proof}
$\Hilb{G/N}{\Hilb{N}{\C^3}}$ parametrises a family of $G$-constellations of the form $\Phi(\OO_W)$ where $W \subset Y_1$ is a $G/N$-cluster.
Take a $G$-invariant subsheaf $\F$ of $\Phi(\OO_W)$.
If $\theta^N(\F)>0$, then we have $\theta(\F)>0$ by the assumption $\varepsilon<<1$
and we may assume $\theta^N(\F)=0$. In this case, there is a $G/N$-invariant subsheaf $F\subseteq \OO_W$ such that $\F=\Phi(F)$ by Lemma \ref{semistable}.
Note that we have $\tau_{1*} F \cong (\pi_{1*} \F)^N$ by the definition of $\Phi$,
which implies that the number of copies of an irreducible representation of $G/N$ appearing in $H^0(F)$ is the same as
that in $H^0(\F)$, proving $\theta(\F)=\varepsilon \theta^{G/N}(F)>0$.
Thus we obtain the $\theta$-stability of $\Phi(\OO_W)$ and hence $\Phi$ induces a morphism $f:\Hilb{G/N}{\Hilb{N}{\C^3}}\to\mathcal{M}_C$. Now since $G\subset\SL(3,\C)$ both are crepant resolutions of $\C^3/G$ so $f$ is an isomorphism.
\end{proof}

\begin{rem}
In the above proof, we can replace $\theta^N$ by an arbitrary $G/N$-invariant generic stability parameter
for $N$-constellations.
Especially, we have similar results for iterated constuctions such as ``Hilb of Hilb of Hilb".

Note that to let $\theta^{G/N}$ be general, we have to construct the moduli space of $\theta$-stable $G$-constellations
on a quasi-projective variety with $G$-action, which could be done by patching local constructions.
\end{rem}

\section{Representations of semidirect products}\label{semidirect}

%We study how the irreducible representations of $G$ can be obtained from irreducible representations of $N$. 

We recall representations of semidirect products from \cite[8.2]{Serre}
to compute several examples in \S \ref{Examples}.
We also use the same notation for an Abelian group $G$ with a subgroup $N$
in \S \ref{AbCase}.
Let $G$ be a finite group obtained as the semidirect product $N \rtimes H$ of subgroups $N$ and $H$.
We assume the normal subgroup $N$ is Abelian.
All the examples in \S \ref{Examples} are of this form.
 Let $\Irr(N)=\{\sigma_0,\ldots,\sigma_{p-1}\}$ be the set of irreducible representations of $N$ where $\dim(\sigma_i)=1$ and denote by $I=\{0,\ldots,p-1\}$ the set of subindices. Let us also denote by $\Irr(G/N)=\{\tau_0,\ldots,\tau_{h_0}\}$ the set of irreducible representations of $G/N=H$ with $d^{G/N}_j:=\dim(\tau_j)$. 

The group $H=G/N$ acts on $\Irr(N)$ as follows:
$H$ acts on $N$ by conjugation and thus on $\Irr(N)$ by 
$h \cdot \sigma (n)=\sigma(h^{-1}nh)$, for $h \in H$, $\sigma \in \Irr(N)$ and $n \in N$.
% since $N$ is normal, $G$ acts on $N$ by conjugation $g\cdot h:=ghg^{-1}$, for $g\in G, h\in N$. Let $\sigma\in\Irr(N)$ and consider its character $\chi_\sigma:N\to\C$. Then $G$ acts on the characters $\chi_\sigma$ by $g\cdot\chi_\sigma(h):=\chi_\sigma(g^{-1}hg)$. Since the character is a function on the conjugacy classes in $N$ this action is constant in the cosets $gN$, i.e.\ $gh\cdot\chi_\sigma(x)=gh'\cdot\chi_\sigma(x)=\chi_\sigma(g^{-1}xg)$, for $g\in G$, $x,h,h'\in N$. Therefore $G/N$ acts on the characters of $N$, which induces an action on $\Irr(N)$ by 
%\[ 
%g\cdot\rho = \sigma \iff g\cdot\chi_\rho=\chi_\sigma 
%\]
%for $\rho,\sigma\in\Irr(N)$, $g\in G/N$. In particular $g\cdot\chi_\rho(1)=\chi_\sigma(1)=\dim(\sigma)$, thus representations in the same orbit have the same dimension. 
%
Choose a set of representatives of the classes in $\Irr N$ under the action of $G/N$ and denote by $\tilde{I}=\{0,\ldots,k\}\subseteq I$ the corresponding subset of subindices. For any $i\in\tilde{I}$ consider the orbit $\Orb(\sigma_i)$ of $\sigma_i$ under $G/N$ of length $n_i$. Let $G_i$ be the stabilizer and let $\Irr(G_i)=\{\tau_i^0,\ldots,\tau_i^{h_i}\}$ the set of irreducible representations of $G_i$. Recall that if $\sigma_i$ and $\sigma_j$ are in the same orbit then $G_i$ and $G_j$ are conjugate, in particular isomorphic. The trivial representation $\sigma_0\in\Irr(N)$ is always fixed so that $G_0=G/N$ and $\tau_0^j=\tau_j$ for all $j$.

The irreducible representations of $G$ are obtained as follows: for every $i\in\tilde{I}$ the representations in the orbit $\Orb(\sigma_i)$ combine to give $h_i+1$ irreducible representations $\rho_i^j$ for $j=0,\ldots,h_i$ with $\dim(\rho_i^j)=n_i\dim(\tau_i^j)$ (see Table \ref{IrrG}).
In other words, they are induced by the representations of $N \rtimes G_i$ obtained as the tensor product of the extensions of $\sigma_i$ and $\tau_i^j$ to $N \rtimes G$.
In particular, if $\sigma_i$ is fixed by $G/N$ then it give rise to $h_0+1$ irreducible representations, each corresponding to an irreducible representation of $G/N$. Note that $\rho_0^0$ is the trivial representation of $G$. 
Then $\rho_i^j$ are all the irreducible representations of $G$ by \cite[8.2]{Serre}.

{\renewcommand{\arraystretch}{1.15}
\begin{table}[htdp]
%\begin{small}
\begin{center}
\begin{tabular}{|c|c|c|c|}
\multicolumn{1}{r}{$\sigma_0$}
 &  \multicolumn{1}{c}{$\Orb(\sigma_1)$}
 & \multicolumn{1}{c}{$\ldots$} 
 & \multicolumn{1}{c}{$\Orb(\sigma_k)$} \\
\cline{1-4}
$\rho_0^0$ & \multirow{3}{*}{$\rho_1^0$} & & \multirow{2}{*}{$\rho_k^0$} \\
\cline{1-1}
$\rho_0^1$ &  &  &  \\
\cline{1-1}\cline{4-4}
 &  &  &  \\
\cline{2-2}
$\vdots$ &  $\vdots$ & $\ldots$ & $\vdots$ \\
\cline{2-2}
 & \multirow{3}{*}{$\rho_1^{h_1}$} &  &  \\
\cline{4-4}
 &  &  & \multirow{2}{*}{$\rho_k^{h_k}$} \\
\cline{1-1}
$\rho_0^{h_0}$ &  &  &  \\
\cline{1-4}
\end{tabular}
\end{center}
%\end{small}
\caption{Irreducible representations of $G$ from the action of $G/N$ into $\Irr(N)$.}
\label{IrrG}
\end{table}}

%Recall that for any finite group $G$ we have that $|G|=\sum_{\rho\in\Irr(G)}\dim(\rho)^2$ and if $G$ acts on a set $X$ then $|\Orb(x)|=|G|/|G_x|$ where $G_x$ is the stabilizer subgroup of $x\in X$. Then
%\begin{align*}
%\sum_{i,j}(\dim\rho_i^j)^2 &
%%= \sum_{j=0}^{h_0}(\dim\rho_0^j)^2+\ldots+\sum_{j=0}^{h_k}(\dim\rho_k^j)^2  
%= \sum_{i=0}^k\sum_{j=0}^{h_i}(n_id_i\dim\tau_i^j)^2
%%	&= \frac{|G/H|^2}{|G_{0}|^2}d_0^2\sum_{j=0}^{h_0}(\dim\tau_0^j)^2+\ldots
%%		+\frac{|G/H|^2}{|G_{k}|^2}d_k^2\sum_{j=0}^{h_k}(\dim\tau_k^j)^2 
%= |G/H|\sum_{i\in\tilde{I}}\frac{|G/H|}{|G_{i}|}d_i^2  %\\
%%	&= |G/H|\sum_{i\in I}d_i^2 
%= |G/H||H| = |G|.
%\end{align*}

\begin{rem} The action of $G/N$ on $\Irr(N)$ to produce $\Irr(G)$ can be translated into the McKay quiver $N$, where every vertex corresponds to a irreducible representation of $N$. Then $G/N$ acts on the set of vertices and on the set of arrows of $Q$, as well as on the path algebra $\kk Q$ permuting the set of primitive idempotents $\{e_i|i\in I\}$. We thus can construct the McKay quiver of $G$ as the $G/N$-orbifold quiver of the McKay quiver of $N$. See \cite{Dem} for the general formulation and \cite{NdC2} for the case of binary dihedral groups in $\GL(2,\C)$.
\end{rem}

Let us describe the stability parameter defined in Definition \ref{theta} which is shown in Table \ref{StabTable}. We are going to use the \textbf{0}-generated stabilities for the groups $N$ and $G/N$ separately so let us denote them as follows:   
\[
\begin{array}{l}
\text{$\theta^N\in\Q^{p}$ such that $\theta^N_i:=\theta^N(\sigma_i)>0$ for $i\neq0$ and $\sigma_i\in\Irr(N)$.} \\
\text{$\theta^{G/N}\in\Q^{h_0+1}$ such that $\theta^{G/N}_j:=\theta^{G/N}(\tau_j)>0$ for $j\neq0$ and $\tau_j\in\Irr(G/N)$.} 
\end{array}
\]

In particular we have $\sum_{i=0}^{p-1}d^N_i\theta^N_i=0$ and $\sum_{i=0}^{h_0}d^{G/N}_j\theta^{G/N}_j=0$, so that 
%\[
%\text{
$\theta^N_0=-\sum_{i=1}^{p-1}d^N_i\theta^N_i$ and $\theta^{G/N}_0=-\sum_{j=1}^{h_0}d^{G/N}_j\theta^{G/N}_j$.
%}
%\] 

{\renewcommand{\arraystretch}{1.3}
\begin{table}[htdp]
\begin{small}
\begin{center}
\begin{tabular}{|c|c|c|c|}
\multicolumn{1}{c}{$\sigma_0$}
 &  \multicolumn{1}{c}{$\Orb(\sigma_1)$}
 & \multicolumn{1}{c}{$\ldots$} 
 & \multicolumn{1}{c}{$\Orb(\sigma_k)$} \\
\cline{1-4}
$\theta_0^N+\varepsilon\theta^{G/N}_0$ & \multirow{3}{*}{$\displaystyle\dim(\tau_1^0)\cdot\!\!\!\!\!\!\!\!\sum_{\sigma_i\in\Orb(\sigma_1)}\!\!\!\!\!\!\!\theta^N_i$} & & \multirow{2}{*}{$\displaystyle\dim(\tau_k^0)\cdot\!\!\!\!\!\!\!\!\sum_{\sigma_i\in\Orb(\sigma_k)}\!\!\!\!\!\!\!\theta^N_i$} \\
\cline{1-1}
$d^{G/N}_1\theta_0^N+\varepsilon\theta^{G/N}_1$ &  &  &  \\
\cline{1-1}\cline{4-4}
 &  &  &  \\
\cline{2-2}
$\vdots$ &  $\vdots$ & $\ldots$ & $\vdots$ \\
\cline{2-2}
 & \multirow{3}{*}{$\displaystyle\dim(\tau_1^{h_1})\cdot\!\!\!\!\!\!\!\!\sum_{\sigma_i\in\Orb(\sigma_1)}\!\!\!\!\!\!\!\theta^N_i$} &  &  \\
\cline{4-4}
 &  &  & \multirow{2}{*}{$\displaystyle\dim(\tau_k^{h_k})\cdot\!\!\!\!\!\!\!\!\sum_{\sigma_i\in\Orb(\sigma_k)}\!\!\!\!\!\!\!\theta^N_i$} \\
\cline{1-1}
$d^{G/N}_{h_0}\theta_0^N+\varepsilon\theta^{G/N}_{h_0}$ &  &  &  \\
\cline{1-4}
\end{tabular}
\end{center}
\end{small}
\caption{Stability condition $\theta$ in terms of $\theta^N$ and $\theta^{G/N}$.}
\label{StabTable}
\end{table}}

\section{The case $G$ Abelian}\label{AbCase}

Let $G\subset\SL(3,\C)$ be a finite Abelian subgroup and let $A\lhd G$ be a normal subgroup of $G$ with $|A|=p$ and $|G/A|=q$. After introducing the toric notation that is needed, we describe how to calculate the triangulation of the junior simplex $\Delta$ corresponding to $Y:=\Hilb{G/A}{\Hilb{A}{\C^3}}$ and we construct explicitly every $G$-constellation in $Y$ from the $A$-clusters. Then we describe a method to calculate the local coordinates of a moduli space of $G$-constellations using the McKay quiver and finish the section describing the stability condition in the Abelian case.

\subsection{How to calculate $\Hilb{A/N}{\Hilb{N}{\C^3}}$}

Every element of $G$ can be written of the form $g=\diag(\varepsilon^{a_1},\varepsilon^{a_2},\varepsilon^{a_3})$ where $\varepsilon$ is an $r$th primitive root of unity and $0\leq a_i<r$. Let $L\supset\Z^3$ the lattice generated by the elements of $G$ written in the form $\frac{1}{r}(a_1,a_2,a_3)$ and let $M:=L^\vee$ the dual lattice of Laurent monomials. The {\em junior simplex} is the triangle $\Delta\subset L_\R:=L\otimes_\Z\R$ with vertices the standard basis $e_1$,$e_2$,$e_3$. We denote by $\R^2_\Delta$ the affine plane spanned by $\Delta$ and $\Z^2_\Delta:=L\cap\R^2_\Delta$. Recall that $\Delta$ contains all lattice points with $a_1+a_2+a_3=r, a_i\geq0$, and triangulations of $\Delta$ are in 1-to-1 correspondence with crepant resolutions of $\C^3/G$.

First consider the action of $A$ on $\C^3$. In \cite{CR02} Craw and Reid give a method to triangulate $\Delta$ into $p$ regular triangles $\Delta_i$ which produces the crepant resolution $\Hilb{A}{\C^3}$. This triangulation shows that $\Hilb{A}{\C^3}\cong\bigcup_{i=1}^p Y_i$ where $Y_i:=\sigma(\Delta_i)\cong\C^3_{\varepsilon_i, \eta_i, \zeta_i}$ is the affine toric variety associated to the triangle $\Delta_i$, and $\varepsilon_i$, $\eta_i$, $\zeta_i$ are Laurent monomials in $x$, $y$ and $z$. 

The action of $G/A$ on $\Hilb{A}{\C^3}$ is again Abelian so it is given by diagonal matrices, thus it acts on every $Y_i$ separately. For every triangle $\Delta_i$ with $i=1,\ldots,p$, we form the the toric singular quotient $Y_i/(G/A)$ and take $G/A$-Hilb($Y_i$) as crepant resolution. Therefore, 

\[ \Hilb{G/A}{\Hilb{A}{\C^3}}=\bigcup_{i=1}^p\Hilb{G/A}{Y_i} \]

In other words, the triangulation of $\Delta$ which gives $\Hilb{G/A}{\Hilb{A}{\C^3}}$ is produced in two steps: Firstly calculate $\Hilb{A}{\C^3}$ according to \cite{CR02} to obtain $\Delta=\bigcup_{i=1}^p\Delta_i$. Secondly, triangulate every $\Delta_i$ into $q$ regular triangles with the same method according to the $\Z/q$-action of $G/A$ into $Y_i$ to produce $\Delta=\bigcup\Delta_{ij}$ for $i=1,\ldots,p$, $j=1,\ldots,q$. Obviously, the same process of successive triangulations can be done as many times as nontrivial normal subgroups we have in a filtration of $G$. See Figure $\ref{Z30}$.

\begin{figure}[ht]
\begin{center}
\begin{pspicture}(0,-0.25)(9,2.5)
	\psset{arcangle=15}
\scalebox{0.65}{
%%%%%%%%%%%%%%%%%%%%%%%%%%%%%%% ROW 1
% Z/2-Hilb
\rput(-1,0){
\scalebox{0.4}{
	% Points of the triangulation
	\rput(5,10){\rnode{1}{\Large $\bullet$}}\rput(7.5,5){\rnode{7}{\Large $\bullet$}}
	\rput(0,0){\rnode{18}{\Large $\bullet$}}\rput(10,0){\rnode{19}{\Large $\bullet$}}	
	% lines
	\ncline{-}{1}{18}\ncline{-}{18}{19}\ncline{-}{19}{1}  % traingle
	\ncline{-}{18}{7}
	}	
\rput(2,-0.5){$\Z/2$-Hilb}
	}
	
% Z/3-Hilb(Z/2-Hilb)
\rput(5,0){
\scalebox{0.4}{
	% Points of the triangulation
	\rput(5,10){\rnode{1}{\Large $\bullet$}}\rput(7.5,5){\rnode{7}{\Large $\bullet$}}
	\rput(0,0){\rnode{18}{\Large $\bullet$}}\rput(10,0){\rnode{19}{\Large $\bullet$}}
	\rput(4.17,5){\rnode{5}{\Large \red{$\bullet$}}}
	\rput(3.33,0){\rnode{12}{\Large \red{$\bullet$}}}
	\rput(6.67,0){\rnode{16}{\Large \red{$\bullet$}}}
	% lines
	\ncline{-}{1}{18}\ncline{-}{18}{19}\ncline{-}{19}{1}  % traingle
	\ncline{-}{18}{7}
	\ncline[linecolor=red]{-}{5}{1}\ncline[linecolor=red]{-}{5}{7}
	\ncline[linecolor=red]{-}{5}{18}\ncline[linecolor=red]{-}{12}{7}
	\ncline[linecolor=red]{-}{16}{7}
	}	
\rput(2,-0.5){$\Z/3$-Hilb($\Z/2$-Hilb)}
	}
	
% Z/5-Hilb(Z/3-Hilb(Z/2-Hilb))
\rput(11,0){
\scalebox{0.4}{
	% Points of the triangulation
	\rput(5,10){\rnode{1}{\Large $\bullet$}}\rput(3.83,7){\rnode{2}{\Large \blue{$\bullet$}}}
	\rput(5.33,8){\rnode{3}{\Large \blue{$\bullet$}}}\rput(2.67,4){\rnode{4}{\Large \blue{$\bullet$}}}
	\rput(4.17,5){\rnode{5}{\Large \red{$\bullet$}}}\rput(5.67,6){\rnode{6}{\Large \blue{$\bullet$}}}
	\rput(7.5,5){\rnode{7}{\Large $\bullet$}}\rput(1.5,1){\rnode{8}{\Large \blue{$\bullet$}}}
	\rput(3,2){\rnode{9}{\Large \blue{$\bullet$}}}\rput(4.5,3){\rnode{10}{\Large \blue{$\bullet$}}}
	\rput(6,4){\rnode{11}{\Large \blue{$\bullet$}}}\rput(3.33,0){\rnode{12}{\Large \red{$\bullet$}}}
	\rput(4.83,1){\rnode{13}{\Large \blue{$\bullet$}}}\rput(6.33,2){\rnode{14}{\Large \blue{$\bullet$}}}
	\rput(7.83,3){\rnode{15}{\Large \blue{$\bullet$}}}\rput(6.67,0){\rnode{16}{\Large \red{$\bullet$}}}
	\rput(8.17,1){\rnode{17}{\Large \blue{$\bullet$}}}\rput(0,0){\rnode{18}{\Large $\bullet$}}
	\rput(10,0){\rnode{19}{\Large $\bullet$}}
	% lines
	\ncline{-}{1}{18}\ncline{-}{18}{19}\ncline{-}{19}{1}  % traingle
	\ncline{-}{18}{7}
	\ncline[linecolor=red]{-}{5}{1}\ncline[linecolor=red]{-}{5}{7}
	\ncline[linecolor=red]{-}{5}{18}\ncline[linecolor=red]{-}{12}{7}
	\ncline[linecolor=red]{-}{16}{7}
	\ncline[linecolor=blue]{-}{1}{6}\ncline[linecolor=blue]{-}{1}{4}
	\ncline[linecolor=blue]{-}{2}{10}\ncline[linecolor=blue]{-}{2}{18}
	\ncline[linecolor=blue]{-}{3}{9}\ncline[linecolor=blue]{-}{3}{7}
	\ncline[linecolor=blue]{-}{4}{18}\ncline[linecolor=blue]{-}{4}{5}
	\ncline[linecolor=blue]{-}{6}{5}\ncline[linecolor=blue]{-}{6}{7}
	\ncline[linecolor=blue]{-}{8}{5}\ncline[linecolor=blue]{-}{8}{12}
	\ncline[linecolor=blue]{-}{9}{12}\ncline[linecolor=blue]{-}{10}{12}
	\ncline[linecolor=blue]{-}{11}{5}\ncline[linecolor=blue]{-}{11}{12}
	\ncline[linecolor=blue]{-}{12}{14}\ncline[linecolor=blue]{-}{14}{7}
	\ncline[linecolor=blue]{-}{14}{16}\ncline[linecolor=blue]{-}{13}{7}
	\ncline[linecolor=blue]{-}{13}{16}\ncline[linecolor=blue]{-}{7}{17}
	\ncline[linecolor=blue]{-}{15}{16}\ncline[linecolor=blue]{-}{15}{19}
	\ncline[linecolor=blue]{-}{17}{16}\ncline[linecolor=blue]{-}{17}{19}
	}	
\rput(2,-0.5){$\Z/5$-Hilb($\Z/3$-Hilb($Z/2$-Hilb))}
	}}
\psline{->}(2.25,1.5)(3.25,1.5)\psline{->}(6,1.5)(7,1.5)
\end{pspicture}
\caption{Successive triangulations of $\Delta$ for a group $G$ of order $r=30=2\!\cdot\!3\!\cdot\!5$.}
\label{Z30}
\end{center}
\end{figure}
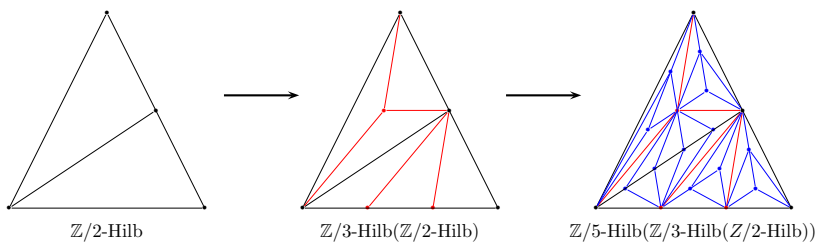

We now look at how $\Hilb{G/A}{\Hilb{A}{\C^3}}$ can be constructed explicitly as a moduli of $G$-constellations. As an $A$-constellation, a point $\mathcal{F}\in\Hilb{A}{\C^3}$ is an $A$-equivariant coherent sheaf on $\C^3$ such that $H^0(\mathcal{F})\cong\C[A]\cong\bigoplus_{\sigma\in\Irr(A)}\sigma$, and $\mathcal{F}\cong\OO_Z$ for some $A$-cluster $Z$. Therefore, locally at $U\subset\Hilb{A}{\C^3}$ we can take a basis $\Gamma:=\{\lambda_\sigma|\sigma\in\Irr(A),\lambda_\sigma\text{ is $\sigma$-semi-invariant}\}$ of $H^0(\mathcal{F})$ to be an $A$-graph. That is, $\lambda_\sigma$ is a monomial in $\C[x,y,z]$ and if $x^iy^jz^k\in\Gamma$ then $x^{i'}y^{j'}z^{k'}\in\Gamma$ for any $i'\leq i$, $j'\leq j$ and $k'\leq k$ (see \cite{Nak01}). We call $\Gamma$ the {\em building block} for $U$. 

It is also known from \cite{Nak01} that $U=\sigma(\Delta_i)\cong\C^3_{a,b,c}$ where $a=\frac{f_\sigma}{\lambda_\sigma}$, $b=\frac{f_{\sigma'}}{\lambda_{\sigma'}}$ and $c=\frac{f_{\sigma''}}{\lambda_{\sigma''}}$ are Laurent monomials in $x$,$y$ and $z$, where $f_\sigma$, $f_{\sigma'}$ and $f_{\sigma''}$ are $\sigma,\sigma'$ and $\sigma''$-semi-invariants respectively. Then, an open set $V=\sigma(\Delta_{ij})\subset\Hilb{G/N}{U}$ is determined by a $G/N$-graph $\Omega:=\{\omega_\tau|\tau\in\Irr(G/N)\}\cong\C[G/N]$ where $\omega_\tau$ are now monomials in $\C[a,b,c]$. Thus, a point $\mathcal{Z}\in V$ as a $G$-equivariant module can be written in the form
\[
\mathcal{Z}=\{ \omega_\tau\Gamma | \tau\in\Irr(G/N) \} = \{ \omega_\tau\lambda_\sigma | \tau\in\Irr(G/N), \sigma\in\Irr(N) \}\cong\C[G]
\]
In other words, the resulting $G$-constellations arising from the open set $U$ are obtained by multiplying the building block $\Gamma$ by the $q$ different $G/N$-graphs $\Omega$. 

\begin{exa}\label{ex!Z6} Let $G=\frac{1}{6}(1,2,3)=\frac{1}{2}(1,0,1)\times\frac{1}{3}(1,2,0)\cong\Z/6\Z$. Take the normal subgroup in $G$ to be $A=\frac{1}{2}(1,0,1)$. The triangulation of the junior simplex $\Delta=\Delta_1\cup\Delta_2$ corresponding to $\Hilb{A}{\C^3}=U_1\cup U_2$ and the toric coordinates are given in Figure \ref{1over2}. 

\begin{figure}[h]
\begin{center}
\begin{pspicture}(0,0)(11,2.75)
	\psset{arcangle=15}
\scalebox{0.8}{
% Z/2-Hilb
\rput(0,0){
\scalebox{0.35}{
	% Points of the triangulation 
	\rput(0,0){\rnode{z}{\Huge $\bullet$}}\rput(-0.75,0){\Huge $e_3$}
	\rput(5,10){\rnode{x}{\Huge $\bullet$}}\rput(5,10.75){\Huge $e_1$}
%	\rput(4.17, 1.67){\rnode{1}{\Huge $\bullet$}}
%	\rput(8.33, 3.33){\rnode{2}{\Huge $\bullet$}}
	\rput(2.50, 5){\rnode{3}{\Huge $\bullet$}}
%	\rput(6.67, 6.67){\rnode{4}{\Huge $\bullet$}}
	\rput(10,0){\rnode{y}{\Huge $\bullet$}}\rput(10.75,0){\Huge $e_2$}	
	% lines
	}}
	\ncline{-}{x}{y}\aput[0.05]{:U}{$z^2$}\ncline{-}{z}{y}\bput[0.05]{:U}{$x^2$}
	\ncline{-}{z}{3}\aput[0.05]{:U}{$y$}\ncline{-}{3}{x}\aput[0.05]{:U}{$y$}
	\ncline{-}{3}{y}\aput[0.05]{:U}{$x:z$}
\rput(1.9,1.9){$U_1$}
\rput(1.25,0.75){$U_2$}
\rput(9,1.5){
${\renewcommand{\arraystretch}{2}\begin{array}{|c|c|c|c|}
\hline
U_i & \varepsilon_i, \eta_i, \zeta_i & \text{$A$-graph} & \text{$G/A$-action type on $U_i$} \\
\hline
U_1 & \frac{x}{z}, z^2, y & \framebox{1~$z$} & \frac{1}{3}(1,0,2)  \\
U_2 & \frac{z}{x}, x^2, y & \framebox{1~$x$} & \frac{1}{3}(1,1,1)  \\
\hline
\end{array}}$
}}
\end{pspicture}
\end{center}
\caption{Toric fan and coordinates for $\Hilb{\frac{1}{2}(1,0,1)}{\C^3}$.}
\label{1over2}
\end{figure}
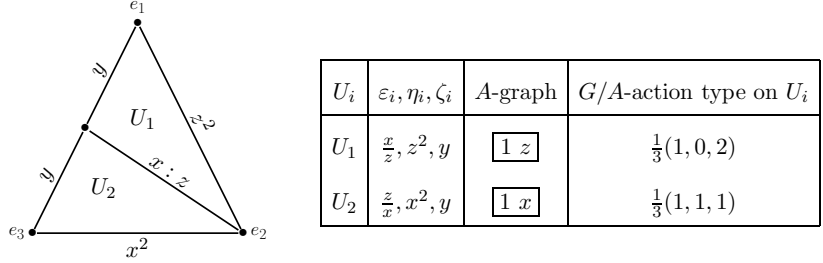

The action of $G/A\cong\frac{1}{3}(1,2,0)$ on $\Hilb{A}{\C^3}$ leaves invariant the open sets $U_1$ and $U_2$, sending $(\varepsilon_1,\eta_1,\zeta_1)\mapsto(\omega\varepsilon_1,\eta_1,\omega^2\zeta_1)$ and $(\varepsilon_2,\eta_2,\zeta_2)\mapsto(\omega\varepsilon_2,\omega\eta_2,\omega\zeta_2)$ respectively, where $\omega$ is a primitive cubic root of unity. Therefore, each of the quotient open sets $U_i/(G/A)$ contains the singularities $\frac{1}{3}(1,0,2)$ and $\frac{1}{3}(1,1,1)$ respectively, which we resolve with the crepant resolutions $(G/A)$-Hilb($U_i$) for $i=1,2$. The triangulations of $\Delta_1$ and $\Delta_2$ are shown in Figure \ref{Zmod3}. 

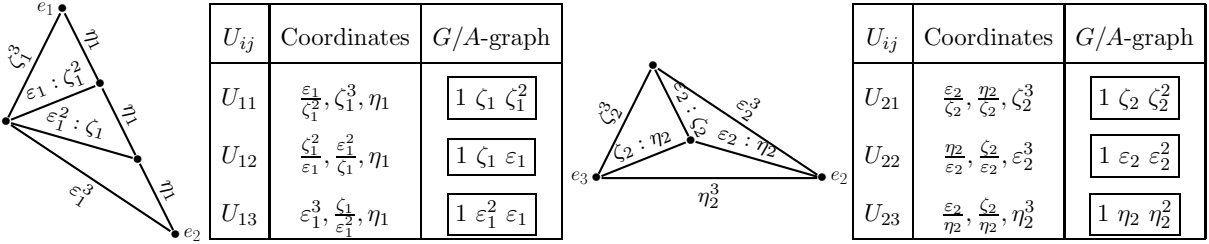
\begin{figure}[h]
\begin{center}
\begin{pspicture}(0,0)(20,2.75)
	\psset{arcangle=15}
% Z/3(1,0,2)-Hilb
\rput(-0.6,0){
\scalebox{0.3}{
	% Points of the triangulation 
%	\rput(0,0){\rnode{z}{\Huge $\bullet$}}\rput(-0.75,0){\Huge $e_3$}
	\rput(5,10){\rnode{x}{\Huge $\bullet$}}\rput(4.25,10){\Huge $e_1$}
%	\rput(4.17, 1.67){\rnode{1}{\Huge $\bullet$}}
	\rput(8.33, 3.33){\rnode{2}{\Huge $\bullet$}}
	\rput(2.50, 5){\rnode{3}{\Huge $\bullet$}}
	\rput(6.67, 6.67){\rnode{4}{\Huge $\bullet$}}
	\rput(10,0){\rnode{y}{\Huge $\bullet$}}\rput(10.75,0){\Huge $e_2$}	
	% lines
	}}
	\ncline{-}{x}{4}\aput[0.05]{:U}{\scriptsize $\eta_1$}
	\ncline{-}{4}{2}\aput[0.05]{:U}{\scriptsize $\eta_1$}
	\ncline{-}{2}{y}\aput[0.05]{:U}{\scriptsize $\eta_1$}
	\ncline{-}{3}{y}\bput[0.05]{:U}{\scriptsize $\varepsilon_1^3$}
	\ncline{-}{3}{x}\aput[0.05]{:U}{\scriptsize $\zeta_1^3$}
	\ncline{-}{3}{2}\aput[0.05]{:U}{\scriptsize $\varepsilon_1^2:\zeta_1$}
	\ncline{-}{3}{4}\aput[0.05]{:U}(0.6){\scriptsize $\varepsilon_1:\zeta_1^2$}
%\rput(0.65,1.65){$U_1$}
%\rput(0,0.5){$U_2$}
\rput(5.2,1.5){
\scalebox{0.85}{
${\renewcommand{\arraystretch}{2}
\begin{array}{|c|c|c|}
\hline
U_{ij} & \text{Coordinates} & \text{$G/A$-graph} \\
\hline
U_{11} & \frac{\varepsilon_1}{\zeta_1^2},\zeta_1^3,\eta_1 & \framebox{1~$\zeta_1$~$\zeta_1^2$} \\
U_{12} & \frac{\zeta_1^2}{\varepsilon_1},\frac{\varepsilon_1^2}{\zeta_1},\eta_1 & \framebox{1~$\zeta_1$~$\varepsilon_1$}  \\
U_{13} & \varepsilon_1^3,\frac{\zeta_1}{\varepsilon_1^2},\eta_1  &  \framebox{1~$\varepsilon_1^2$~$\varepsilon_1$}  \\
\hline
\end{array}}$
}}
% Z/3(1,1,1)-Hilb
\rput(8,0.75){
\scalebox{0.3}{
	% Points of the triangulation 
	\rput(0,0){\rnode{z}{\Huge $\bullet$}}\rput(-0.75,0){\Huge $e_3$}
%	\rput(5,10){\rnode{x}{\Huge $\bullet$}}\rput(5,10.75){\Huge $e_1$}
	\rput(4.17, 1.67){\rnode{1}{\Huge $\bullet$}}
%	\rput(8.33, 3.33){\rnode{2}{\Huge $\bullet$}}
	\rput(2.50, 5){\rnode{3}{\Huge $\bullet$}}
%	\rput(6.67, 6.67){\rnode{4}{\Huge $\bullet$}}
	\rput(10,0){\rnode{y}{\Huge $\bullet$}}\rput(10.75,0){\Huge $e_2$}	
	% lines
	}}
	\ncline{-}{z}{y}\bput[0.05]{:U}{\scriptsize $\eta_2^3$}
	\ncline{-}{z}{3}\aput[0.05]{:U}{\scriptsize $\zeta_2^3$}
	\ncline{-}{3}{y}\aput[0.05]{:U}{\scriptsize $\varepsilon_2^3$}
	\ncline{-}{3}{1}\aput[0.05]{:U}(0.7){\scriptsize $\varepsilon_2:\zeta_2$}
	\ncline{-}{z}{1}\aput[0.05]{:U}(0.55){\scriptsize $\zeta_2:\eta_2$}
	\ncline{-}{1}{y}\aput[0.05]{:U}(0.4){\scriptsize $\varepsilon_2:\eta_2$}
%\rput(0.65,1.65){$U_1$}
%\rput(0,0.5){$U_2$}
\rput(13.75,1.5){
\scalebox{0.85}{
${\renewcommand{\arraystretch}{2}\begin{array}{|c|c|c|}
\hline
U_{ij} & \text{Coordinates} & \text{$G/A$-graph} \\
\hline
U_{21} & \frac{\varepsilon_2}{\zeta_2},\frac{\eta_2}{\zeta_2},\zeta_2^3  & \framebox{1~$\zeta_2$~$\zeta_2^2$} \\
U_{22} & \frac{\eta_2}{\varepsilon_2},\frac{\zeta_2}{\varepsilon_2},\varepsilon_2^3 & \framebox{1~$\varepsilon_2$~$\varepsilon_2^2$} \\
U_{23} & \frac{\varepsilon_2}{\eta_2},\frac{\zeta_2}{\eta_2},\eta_2^3 & \framebox{1~$\eta_2$~$\eta_2^2$} \\
\hline
\end{array}}$
}}
\end{pspicture}
\end{center}
\caption{$\frac{1}{3}(1,0,2)$-Hilb($U_1$) and $\frac{1}{3}(1,1,1)$-Hilb($U_2$).}
\label{Zmod3}
\end{figure}

In constructing the resolution $\Hilb{\Z/2\Z}{\C^3}$ we added only one new lattice point to $\Delta$, namely $\frac{1}{2}(1,0,1)$, producing the subdivision $\Delta=\Delta_1\cup\Delta_2$. To construct $\Hilb{\Z/3\Z}{\Hilb{\Z/2\Z}{\C^3}}$ we now introduce the remaining lattice points, namely $\frac{1}{6}(1,2,3)$, $\frac{1}{6}(2,4,0)$ and $\frac{1}{6}(4,2,0)$, and triangulate $\Delta_1$ and $\Delta_2$ according to the \cite{CR02} algorithm. Changing back to the coordinates $x$, $y$ and $z$ we obtain the fan shown in Figure \ref{z3z2}. 

\begin{figure}[h]
\begin{center}
\begin{pspicture}(0,0)(3.5,3.5)
	\psset{arcangle=15}
\scalebox{0.75}{
% Z/3-Hilb(Z/2-Hilb)
\rput(0,-0.25){
\scalebox{0.45}{
	% Points of the triangulation 
	\rput(0,0){\rnode{z}{\Huge $\bullet$}}\rput(-0.75,0){\Huge $e_3$}
	\rput(5,10){\rnode{x}{\Huge $\bullet$}}\rput(5,10.75){\Huge $e_1$}
	\rput(4.17, 1.67){\rnode{1}{\Huge\red $\bullet$}}
	\rput(8.33, 3.33){\rnode{2}{\Huge\red $\bullet$}}
	\rput(2.50, 5){\rnode{3}{\Huge $\bullet$}}
	\rput(6.67, 6.67){\rnode{4}{\Huge\red $\bullet$}}
	\rput(10,0){\rnode{y}{\Huge $\bullet$}}\rput(10.75,0){\Huge $e_2$}	
	% lines
	}
	\ncline{-}{x}{4}\aput[0.05]{:U}{\scriptsize $z^2$}
	\ncline{-}{4}{2}\aput[0.05]{:U}{\scriptsize $z^2$}
	\ncline{-}{2}{y}\aput[0.05]{:U}{\scriptsize $z^2$}
	\ncline{-}{3}{x}\aput[0.05]{:U}{\scriptsize $y^3$}
	\ncline{-}{3}{2}\aput*[-0.1]{:U}{\tiny $x^2:yz^2$}
	\ncline{-}{3}{4}\aput*[-0.1]{:U}(0.6){\tiny $x:y^2z$}	
	\ncline{-}{z}{y}\bput[0.05]{:U}{\scriptsize $z^6$}
	\ncline{-}{z}{3}\aput[0.05]{:U}{\scriptsize $y^3$}
	\ncline{-}{3}{y}\aput*[-0.1]{:U}{\tiny $x^3:z^3$}
	\ncline{-}{3}{1}\aput*[-0.1]{:U}(0.6){\tiny $z:xy$}
	\ncline{-}{z}{1}\aput*[-0.1]{:U}(0.55){\tiny $x^2:y$}
	\ncline{-}{1}{y}\aput*[-0.1]{:U}(0.4){\tiny $x^3:z$}
	% Names of opens
	\rput(2.25,3.35){\small $U_{11}$}
	\rput(2.35,1){\small $U_{22}$}
	\rput(1.1,1.1){\small $U_{21}$}
	\rput(2.85,2.25){\small $U_{12}$}
	\rput(3.5,1.15){\small $U_{13}$}
	\rput(2,0.35){\small $U_{23}$}
}}
\end{pspicture}
\end{center}
\caption{Toric fan for $\Z/3\Z$-Hilb($\Hilb{\Z/2\Z}{\C^3})$.}
\label{z3z2}
\end{figure}
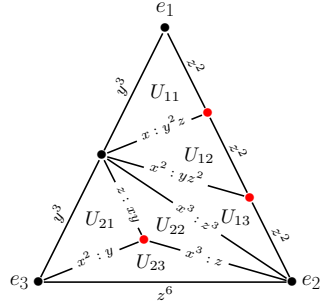

For $j=1,2,3$ the basis for the $G$-constellations of the open set $U_{1j}$ are given by multiplying every basis element in the $(G/A)$-graphs $\Omega_1=\{1,\zeta_1,\zeta_1^2\}$, $\Omega_2=\{1,\zeta_1,\varepsilon_1\}$ and $\Omega_3=\{1,\varepsilon_1,\varepsilon_1^2\}$ by the building block $\Gamma=\{1,z\}$ coming from the open set $U_1$. Similarly for the open sets $U_{2j}\subset\Hilb{\Z/3\Z}{\Hilb{\Z/2\Z}{\C^3}}$. More precisely, the $G$-constellations are 
\[
\small 
{\renewcommand{\arraystretch}{1.5}
\begin{array}{l}
M_{11} = \{ 1\cdot\framebox{1~$z$}, \zeta_1\cdot\framebox{1~$z$}, \zeta_1^2\cdot\framebox{1~$z$}\}	= \{1,z,y,yz,y^2,y^2z\} \\
M_{12} = \{ 1\cdot\framebox{1~$z$}, \zeta_1\cdot\framebox{1~$z$}, \varepsilon_1\cdot\framebox{1~$z$}\} 
		= \{1,z,y,yz,\text{\Large $\frac{x}{z}$},x\} \\
M_{13} = \{ 1\cdot\framebox{1~$z$}, \varepsilon_1^2\cdot\framebox{1~$z$}, \varepsilon_1\cdot\framebox{1~$z$}\} 
		= \{1,z,\text{\Large $\frac{x^2}{z^2}$},\text{\Large $\frac{x^2}{z}$},\text{\Large $\frac{x}{z}$},x\} \\
M_{21} = \{ 1\cdot\framebox{1~$x$}, \zeta_2\cdot\framebox{1~$x$}, \zeta_2^2\cdot\framebox{1~$x$}\} 
		= \{1,x,y,xy,y^2,xy^2\} \\
M_{22} = \{ 1\cdot\framebox{1~$x$}, \varepsilon_2\cdot\framebox{1~$x$}, \varepsilon_2^2\cdot\framebox{1~$x$}\} 
		= \{1,x,\text{\Large $\frac{z}{x}$},z,\text{\Large $\frac{z^2}{x^2}$},\text{\Large $\frac{z^2}{x}$}\} \\
M_{23} = \{ 1\cdot\framebox{1~$x$}, \eta_2\cdot\framebox{1~$x$}, \eta_2^2\cdot\framebox{1~$x$}\} 
		= \{1,x,x^2,x^3,x^4,x^5\} 
\end{array}}
\]

Let now $A=\frac{1}{3}(1,2,0)$ be the normal subgroup. Then $\Hilb{A}{\C^3}\cong V_1\cup V_2\cup V_3$ where $V_i\cong\C^3$. The quotient group $G/A\cong\frac{1}{2}(1,0,1)$ produces a $\Z/2\Z$-action on every $V_i$ for $i=1,2,3$. The resolution of these singularities is translated into the junior simplex $\Delta=\Delta_1\cup\Delta_2\cup\Delta_3$ by adding the points $\frac{1}{6}(3,0,3)$ and $\frac{1}{6}(1,2,3)$, triangulating in the only possible way as in Figure \ref{z2z3}. All crepant resolutions of $\C^3/\frac{1}{6}(1,2,3)$ are shown in Figure \ref{crepZ6}.

\begin{figure}[h]
\begin{center}
\begin{pspicture}(0,0)(8,3.5)
	\psset{arcangle=15}
% Z/2-Hilb(Z/3-Hilb)
\scalebox{0.8}{
\rput(-0.5,-0.25){
\scalebox{0.45}{
	% Points of the triangulation 
	\rput(0,0){\rnode{z}{\Huge $\bullet$}}\rput(-0.75,0){\Huge $e_3$}
	\rput(5,10){\rnode{x}{\Huge $\bullet$}}\rput(5,10.75){\Huge $e_1$}
	\rput(4.17, 1.67){\rnode{1}{\Huge\red $\bullet$}}
	\rput(8.33, 3.33){\rnode{2}{\Huge $\bullet$}}
	\rput(2.50, 5){\rnode{3}{\Huge\red $\bullet$}}
	\rput(6.67, 6.67){\rnode{4}{\Huge $\bullet$}}
	\rput(10,0){\rnode{y}{\Huge $\bullet$}}\rput(10.75,0){\Huge $e_2$}	
	% lines
	}
	\ncline{-}{x}{4}\aput[0.05]{:U}{\scriptsize $z^2$}
	\ncline{-}{4}{2}\aput[0.05]{:U}{\scriptsize $z^2$}
	\ncline{-}{2}{y}\aput[0.05]{:U}(0.6){\scriptsize $z^2$}
	\ncline{-}{3}{4}\aput*[-0.15]{:U}{\tiny $x:y^2z$}
	\ncline{-}{z}{3}\aput[0.05]{:U}{\scriptsize $y^3$}
	\ncline{-}{3}{x}\aput[0.05]{:U}{\scriptsize $y^3$}	
	\ncline{-}{z}{1}\aput*[-0.15]{:U}(0.5){\tiny $x^2:y$}
	\ncline{-}{1}{2}\aput*[-0.15]{:U}(0.5){\tiny $x^2:y$}
	\ncline{-}{1}{4}\aput*[-0.15]{:U}(0.5){\tiny $xz:y^2$}
	\ncline{-}{1}{y}\aput*[-0.1]{:U}{\tiny $x^3:z$}
	\ncline{-}{z}{4}\aput*[-0.15]{:U}(0.45){\tiny $x^2:y^4$}
	\ncline{-}{z}{y}\bput[0.05]{:U}{\scriptsize $x^6$}
	% Names of opens
	\rput(2.25,3.35){\small $U'_{11}$}
	\rput(1.5,2){\small $U'_{12}$}
	\rput(1.65,1.075){\small $U'_{21}$}
	\rput(3.15,1.8){\small $U'_{22}$}
	\rput(3.6,0.8){\small $U'_{32}$}
	\rput(2,0.35){\small $U'_{31}$}
}
\rput(8.25,2)
{
${\renewcommand{\arraystretch}{1.5}\begin{array}{|c|c|}
\hline
	& \text{$G$-constellations} \\
\hline
M'_{11} & \{1,y,y^2,z,yz,y^2z\} \\
M'_{12} & \{1,y,y^2,\frac{x}{y^2},\frac{x}{y},x\} \\
M'_{21} & \{1,x,y,\frac{y^2}{x},y^2,\frac{y^3}{x}\} \\
M'_{22} & \{1,x,y,z,xz,yz\} \\
M'_{31} & \{1,x,x^2,x^3,x^4,x^5\} \\
M'_{32} & \{1,x,x^2,z,xz,x^2z\} \\
\hline
\end{array}}$
}}
\end{pspicture}
\end{center}
\caption{Toric fan $\Hilb{\Z/2\Z}{\Hilb{\Z/3\Z}{\C^3}}$ and the corresponding $G$-constellations.}
\label{z2z3}
\end{figure}
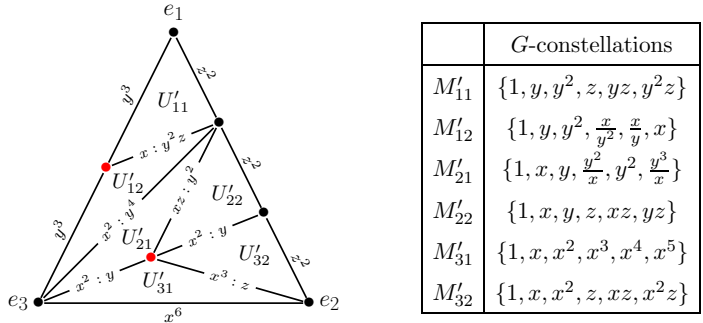

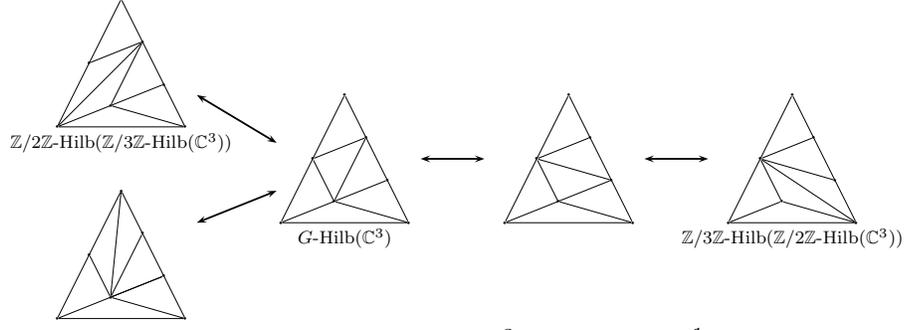
\begin{figure}[h]
\begin{center}
\begin{pspicture}(0,0)(10,3.5)
	\psset{arcangle=15}
\rput(0,-0.5){
\scalebox{0.85}{
% Z3-Hilb( Z/2-Hilb)
\rput(0,3){
\scalebox{0.2}{
	% Points of the triangulation 
	\rput(0,0){\rnode{z}{\Large $\bullet$}}
	\rput(5,10){\rnode{x}{\Large $\bullet$}}
	\rput(4.17, 1.67){\rnode{1}{\Large $\bullet$}}
	\rput(8.33, 3.33){\rnode{2}{\Large $\bullet$}}
	\rput(2.50, 5){\rnode{3}{\Large $\bullet$}}
	\rput(6.67, 6.67){\rnode{4}{\Large $\bullet$}}
	\rput(10,0){\rnode{y}{\Large $\bullet$}}	
	% lines
	\ncline{-}{x}{y}\ncline{-}{y}{z}\ncline{-}{z}{x}  % junior simplex
	\ncline{-}{z}{2}\ncline{-}{z}{4}
	\ncline{-}{1}{y}\ncline{-}{1}{4}
	\ncline{-}{3}{4}
	}
\rput(1,-0.25){\scriptsize $\Hilb{\Z/2\Z}{\Hilb{\Z/3\Z}{\C^3}}$}
	}

% botton-left
\rput(0,0){
\scalebox{0.2}{
	% Points of the triangulation 
	\rput(0,0){\rnode{z}{\Large $\bullet$}}
	\rput(5,10){\rnode{x}{\Large $\bullet$}}
	\rput(4.17, 1.67){\rnode{1}{\Large $\bullet$}}
	\rput(8.33, 3.33){\rnode{2}{\Large $\bullet$}}
	\rput(2.50, 5){\rnode{3}{\Large $\bullet$}}
	\rput(6.67, 6.67){\rnode{4}{\Large $\bullet$}}
	\rput(10,0){\rnode{y}{\Large $\bullet$}}	
	% lines
	\ncline{-}{x}{y}\ncline{-}{y}{z}\ncline{-}{z}{x}  % junior simplex
	\ncline{-}{z}{2}
	\ncline{-}{1}{y}\ncline{-}{1}{2}\ncline{-}{1}{4}
	\ncline{-}{1}{3}\ncline{-}{1}{x}
	}}

% G-Hilb
\rput(3.5,1.5){
\scalebox{0.2}{
	% Points of the triangulation 
	\rput(0,0){\rnode{z}{\Large $\bullet$}}
	\rput(5,10){\rnode{x}{\Large $\bullet$}}
	\rput(4.17, 1.67){\rnode{1}{\Large $\bullet$}}
	\rput(8.33, 3.33){\rnode{2}{\Large $\bullet$}}
	\rput(2.50, 5){\rnode{3}{\Large $\bullet$}}
	\rput(6.67, 6.67){\rnode{4}{\Large $\bullet$}}
	\rput(10,0){\rnode{y}{\Large $\bullet$}}	
	% lines
	\ncline{-}{x}{y}\ncline{-}{y}{z}\ncline{-}{z}{x}  % junior simplex
	\ncline{-}{z}{2}\ncline{-}{1}{y}
	\ncline{-}{1}{3}\ncline{-}{1}{4}
	\ncline{-}{3}{4}
	}
\rput(1,-0.25){\scriptsize $G$-Hilb($\C^3$)}
	}

% middle one
\rput(7,1.5){
\scalebox{0.2}{
	% Points of the triangulation 
	\rput(0,0){\rnode{z}{\Large $\bullet$}}
	\rput(5,10){\rnode{x}{\Large $\bullet$}}
	\rput(4.17, 1.67){\rnode{1}{\Large $\bullet$}}
	\rput(8.33, 3.33){\rnode{2}{\Large $\bullet$}}
	\rput(2.50, 5){\rnode{3}{\Large $\bullet$}}
	\rput(6.67, 6.67){\rnode{4}{\Large $\bullet$}}
	\rput(10,0){\rnode{y}{\Large $\bullet$}}	
	% lines
	\ncline{-}{x}{y}\ncline{-}{y}{z}\ncline{-}{z}{x}  % junior simplex
	\ncline{-}{z}{2}\ncline{-}{1}{y}
	\ncline{-}{1}{3}\ncline{-}{3}{2}
	\ncline{-}{3}{4}
	}}
% Z3-Hilb( Z/2-Hilb)
\rput(10.5,1.5){
\scalebox{0.2}{
	% Points of the triangulation 
	\rput(0,0){\rnode{z}{\Large $\bullet$}}
	\rput(5,10){\rnode{x}{\Large $\bullet$}}
	\rput(4.17, 1.67){\rnode{1}{\Large $\bullet$}}
	\rput(8.33, 3.33){\rnode{2}{\Large $\bullet$}}
	\rput(2.50, 5){\rnode{3}{\Large $\bullet$}}
	\rput(6.67, 6.67){\rnode{4}{\Large $\bullet$}}
	\rput(10,0){\rnode{y}{\Large $\bullet$}}	
	% lines
	\ncline{-}{x}{y}\ncline{-}{y}{z}\ncline{-}{z}{x}  % junior simplex
	\ncline{-}{1}{z}\ncline{-}{1}{y}
	\ncline{-}{1}{3}\ncline{-}{3}{y}
	\ncline{-}{3}{4}\ncline{-}{3}{2}
	}
\rput(1,-0.25){\scriptsize $\Hilb{\Z/3\Z}{\Hilb{\Z/2\Z}{\C^3}}$}
	}
\psline{<->}(2.25,1.5)(3.5,2)\psline{<->}(5.75,2.5)(6.75,2.5)
\psline{<->}(2.25,3.5)(3.5,2.75)\psline{<->}(9.25,2.5)(10.25,2.5)
}}
\end{pspicture}
\end{center}
\caption{Crepant resolutions of $\C^3/G$ with $G=\frac{1}{6}(1,2,3)$.}
\label{crepZ6}
\end{figure}

\end{exa}

\subsection{$\theta$-stability in the Abelian case.}

Every irreducible representation $\sigma_i\in\Irr(N)$ is fixed so all stabilizers $G_i$ are isomorphic to $G/N$. The irreducible representations of $G$ are therefore distributed as in Table \ref{tAbStab}. Since $n_i=d^N_i=d^{G/N}_j=1$ for all $i=0,\ldots,p-1$, $j=0,\ldots,q-1$, we have that $\dim(\rho)=1$ for $\rho\in\Irr(G)$. The stability condition $\theta\in\Theta$ for which the crepant resolution $\Hilb{G/A}{\Hilb{A}{\C^3}}\cong\mathcal{M}_\theta$ is given also in Table \ref{tAbStab}. Notice that in this case $\theta^N_0=-\sum_{i=1}^{p-1}\theta_i^N$ and $\theta^{G/N}_0=-\sum_{j=1}^{q-1}\theta_i^{G/N}$, so the fact that $0<\varepsilon<<1$ implies that $\theta(\rho^k_0)<0$ for all $k$.

{\renewcommand{\arraystretch}{1.4}
\begin{table}[htdp]
\begin{small}
\begin{center}
$\Irr(G)=$
\begin{tabular}{|c|c|c|c|}
 \multicolumn{1}{c}{$\sigma_0$}
 &  \multicolumn{1}{c}{$\sigma_1$}
 & \multicolumn{1}{c}{$\ldots$} 
 & \multicolumn{1}{c}{$\sigma_{p-1}$} \\
\cline{1-4}
$\rho_0^0$	& $\rho^0_1$	&  $\ldots$	& $\rho^0_{p-1}$	 \\
\cline{1-4}
$\rho_0^1$	& $\rho^1_1$	&  $\ldots$	& $\rho^1_{p-1}$	 \\
\cline{1-4}
& $\vdots$ &	& $\vdots$ \\
\cline{1-4}
$\rho_0^{q-1}$	& $\rho^{q-1}_1$	&  $\ldots$	& $\rho^{q-1}_{p-1}$	 \\
\cline{1-4}
\end{tabular} $~~\text{  }~~\theta=$
\begin{tabular}{|c|c|c|c|}
 \multicolumn{1}{c}{$\sigma_0$}
 &  \multicolumn{1}{c}{$\sigma_1$}
 & \multicolumn{1}{c}{$\ldots$} 
 & \multicolumn{1}{c}{$\sigma_{p-1}$} \\
\cline{1-4}
$-\sum_{i=1}^{p-1}\theta^N_i-\varepsilon\sum_{j=1}^{q-1}\theta^{G/N}_j$	& $\theta^N_1$	&  $\ldots$	& $\theta^N_{p-1}$ \\
\cline{1-4}
$-\sum_{i=1}^{p-1}\theta^N_i+\varepsilon\theta^{G/N}_1$ & $\theta^N_1$ & $\ldots$ & $\theta^N_{p-1}$ \\
\cline{1-4}
$\vdots$ & $\vdots$ &	& $\vdots$  \\
\cline{1-4}
$-\sum_{i=1}^{p-1}\theta^N_i+\varepsilon\theta^{G/N}_{q-1}$ & $\theta^N_1$ & $\ldots$ & $\theta^N_{p-1}$ \\
\cline{1-4}
\end{tabular}
\end{center}
\end{small}
\caption{$\Irr(G)$ and the stability condition $\theta$ in terms of $\theta_N$ and $\theta_{G/N}$.}
\label{tAbStab}
\end{table}}

\begin{exa} Let us take the group $G=\frac{1}{6}(1,2,3)$ and consider $\Hilb{\Z/3\Z}{\Hilb{\Z/2\Z}{\C^3}}$. The distribution of $\Irr(G)$ and the stability condition $\theta$ are shown in Figure \ref{McKZ6} where $a,b_1,b_2\in\Q$ are positive numbers and $0<\varepsilon<<1$.  

The stability condition is in this case given clockwise around the McKay quiver. By checking the subrepresentations in every affine piece we can see that chamber $C\subset\Theta$ is given by 
\[
\begin{array}{rrr}
\theta_2,\theta_4 < 0, & \theta_2+\theta_5 > 0, &  \theta_1+\theta_4 > 0, \\
\theta_3 > 0, & \theta_4+\theta_5 > 0, & \theta_0+\theta_1+\theta_3 > 0. 
\end{array}
\]
Thus if we take $0<\varepsilon<a/$max$\{b_1,b_2\}$ every inequality is satisfied by $\theta$.

{\renewcommand{\arraystretch}{1.25}
\begin{figure}[h]
\begin{center}
\begin{pspicture}(0,0)(3,1.25)
	\psset{arcangle=20,nodesep=2pt}
\rput(-2.5,0.5){
$\Irr(G)=
\begin{tabular}{|c|c|}
\cline{1-2}
$\rho_0^0$	& $\rho_1^0$	 \\
\cline{1-2}
$\rho_0^1$	& $\rho_1^1$	 \\
\cline{1-2}
$\rho_0^2$	& $\rho_1^2$	 \\
\cline{1-2}
\end{tabular}$}

\rput(3,0.5){
$\theta := 
\begin{tabular}{|c|c|}
\cline{1-2}
$\theta_0$	& $\theta_1$	 \\
\cline{1-2}
$\theta_2$	& $\theta_3$	 \\
\cline{1-2}
$\theta_4$	& $\theta_5$	 \\
\cline{1-2}
\end{tabular}
=
\begin{tabular}{|c|c|}
\cline{1-2}
$-a-\varepsilon(b_1+b_2)$	& $a$	 \\
\cline{1-2}
$-a+\varepsilon b_1$	& $a$	 \\
\cline{1-2}
$-a+\varepsilon b_2$	& $a$	 \\
\cline{1-2}
\end{tabular}$}
\end{pspicture}
\end{center}
\caption{$\Irr(G)$ and stability condition for $\Hilb{\Z/3}{\Hilb{\Z/2}{\C^3}}$.}
\label{McKZ6}
\end{figure}}

\end{exa}

\section{Local coordinates}\label{LocalCoord} 

In this section, we introduce some notation and terminology to illustrate non-Abelian examples in the next section.
Let $(Q, R)$ be a quiver with relations, where $R$ is a two-sided
ideal in the path algebra $\C Q$ of $Q$. Let $Q'$ be a connected quiver and let $\phi:Q' \to Q$ be a morphism of quivers, that is, a pair of morphisms $\phi_0:Q'_0\to Q_0$ and $\phi_1:Q'_1\to Q_1$ between the respective vertex and arrows sets. For a vertex $v$ of $Q$, let $d_v$ be the number of vertices in the preimage $H_v:=\{\phi^{-1}(v)\}\subseteq Q'_0$ and let $\C^{H_v} \cong \C^{d_v}$ be the vector space with a distinguished basis $\{e_w \mid w \in H_v\}$.
Notice that for an arrow $a \in Q_1$ with the head $h(a)$ and the tail $t(a)$,
a linear map $\C^{H_{t(a)}} \to \C^{H_{h(a)}}$ is given by a matrix in $\Mat_{d_{t(a)} \times d_{h(a)}}$.
%Notice that using the map $\phi$ we can assign to any arrow $p:a\to b$ in $Q'$ the corresponding linear map $\C^{d_{\phi(a)}}\to\C^{d_{\phi(b)}}$ in the representation space of $Q$ of dimension vector $(d_i)_{i\in Q_0}$ given by a matrix ${\bf m}^p\in\Mat_{d_{\phi(a)}\times d_{\phi(b)}}$.

To the pair $(Q', \phi)$ we construct a representation $S_{Q'}$ of dimension vector $(d_v)$
such that for an arrow $a\in Q_1$, the associated matrix in $\Mat_{d_{t(a)} \times d_{h(a)}}$
is given by writing $k_\alpha \in \C$, $k_\alpha \neq 0$ at the $(t(\alpha), h(\alpha))$-entry 
for every $\alpha \in \phi^{-1}(a)\subseteq Q'_1$ and $0$ everywhere else.
Note that $S_{Q'}$ can be regarded as the direct image of a representation of $Q'$ with dimension vector $(1, \dots, 1)$
whose linear maps are nonzero by the morphism $\phi$.

%Now assign to every vertex in $H_v$ a distinct vector $e_i$ from the standard basis $\{e_j\}_{j=1}^{d_v}$ of $\C^{d_v}$. Then, to the pair $(Q',\{e_v\}_{v\in Q'_0})$ formed by the quiver $Q'$ together with the above labeling of basis elements at every vertex, we construct a representation $S_{Q'}\in\rep(Q)$ by writing $k_a\in\C$, $k_a\neq0$ at the entry ${\bf m}^a_{ij}$ for every arrow $a:(v,e_i)\to(v',e_j)$ in $Q'$, and 0 everywhere else.

\begin{defn}\label{defn:skeleton} We say that $(Q',\phi)$ is a {\em skeleton} if the representation $S_{Q'}$ verifies the relations $R$ for a suitable choice of
$(k_{\alpha})_{\alpha \in Q_1'} \in (\C^*)^{Q_1'}$ and the isomorphism class of $S_{Q'}$ does not depend on such a choice. By abuse of notation we also call $S_{Q'}$ a skeleton.
\end{defn}

When the quiver $Q$ is the McKay quiver, it happens often (for instance in every family of examples treated in this paper) that a suitable subset $\mathcal{U}$ of skeletons determines an open cover of the moduli space $\mathcal{M}_C$ for a given $C\subset\Theta$. In other words, the conditions $k_a\neq0$ for all $a\in Q'$ determines an open set $U_{Q'}\subset\mathcal{M}_C$, and the union of such open sets for skeletons in $\mathcal{U}$ form an open cover.  
Here, the skeleton $S_{Q'}$ is the representation corresponding to the origin in the affine open set $U_{Q'}\subseteq\C^m$ for some $m$ (compare with \cite{NdCS},$\S7$).
%Conversely, given a finite open cover $\{V_i\}_{i=1}^N$ of the moduli space $\mathcal{M}_C$, the set of skeletons $\mathcal{U}$ consist of the representations of $Q$ corresponding to the origin of coordinates of each open set $V_i$ (compare with \cite{NdCS},$\S7$).

For Abelian groups in $\SL(3,\C)$, the set $\mathcal{U}$ is determined by the torus fixed points (see \cite{Ish04} $\S3$). As it will be shown in the examples of the following sections, in the case of the iterated Hilbert scheme $\Hilb{G/N}{\Hilb{N}{\C^3}}$ the subset $\mathcal{U}$ is induced by the skeletons defining the open cover of $\Hilb{N}{\C^3}$.

The local coordinates of an open set $U_{Q'}$ associated with a skeleton $(Q', \phi)$ can be obtained explicitly as follows.
Fix $(k_{\alpha})\in (\C^*)^{Q'_1}$ in the definition of skeletons.
Consider representations of $Q$ which associates to each arrow $a \in Q_1$
a matrix in $\Mat_{d_{t(a)} \times d_{h(a)}}$ whose $(h(\alpha), t(\alpha))$-entry
is $k_{\alpha}$ for $\alpha \in \phi^{-1}(a)$.
These representations form an affine space whose coordinates are the remaining entries
of the matrices.
If we consider only representations which satisfy the relations $R$,
we obtain an affine scheme $U_{Q'}\subseteq\C^m$ for some $m$ which contains $S_{Q'}$ as the representation corresponding to the origin.
In good cases including all the examples in this paper, $U_{Q'}$ becomes an affine
open neighbourhood of $S_{Q'}$ in the moduli space of representations of $(Q, R)$
and we can specify the entries of the matrices which form the local coordinates around $S_{Q'}$.

Let $S:=\C[x,y,z]$ and for every $\rho\in\Irr(G)$ consider the Cohen-Macaulay $S^G$-module $S_\rho:=(S\otimes\rho^*)^G$.
We have the tautological bundle $\RR_\rho$ on the moduli space of $G$-constellations
whose global sections form the module $S_\rho$.
On the open set $U_{Q'}$ corresponding to a skeleton $Q'$, the vertices of $Q'$ correspond
to sections of $\RR_{\rho}$'s over $U_{Q'}$, where we always assume the vertex over
the trivial representation $\rho_0$ corresponds to $1$.
These sections can be regarded as rational sections of $S_{\rho}$ over $\C^3/G$.
If we take a basis $u_1, \dots, u_d$ of the representation space $\rho$,
a rational section of $S_{\rho}$ over $\C^3/G$ is of the form $\sum_{i=1}^d f_i u_i^*$,
where $f_i$ are rational functions in the $\rho$-part of $\C(x, y, z)$ and $u_1^*, \dots, u_d^*$ form the dual basis of $\rho^*$. 
Then such a rational section is given by a $d$-tuple $(f_1, \dots, f_d)$ of rational functions which spans the representation $\rho$ in $\C(x, y, z)$.

% Taking $T=\bigoplus_{\rho\in\Irr(G)}S_\rho$, the McKay quiver of $G$ is the quiver associated to $\End_{S^G}(T)$. 
%
%Now let $M$ be a $G$-constellation in $\C^3$. Taking a basis of $H^0(M)$, it can be decomposed into the Cohen-Macaulay modules $S_v$, thus having $d_v:=\dim(\rho_v)$ basis elements associated to each vertex $v\in Q_0$. Therefore we can construct the skeleton $(Q',\{e_v\}_{v\in Q'_0})$ where the labels $\{e_i\}_{i=1}^{d_v}$ are given by the basis of $S_v$, and the set arrows $Q'_1$ correspond to the set of irreducible maps between then modules $S_v$, for $v\in Q$. Thus we can obtain the corresponding open set $U_{Q'}$, from which we can read off the local coordinates.

\begin{exa} Let $M_{22}=\{1,x,\frac{z}{x},z,\frac{z^2}{x^2},\frac{z^2}{x}\}$ be the $G$-constellation defining the open $U_{22}\in\Hilb{\Z/3\Z}{Y}$ where $Y:=\Hilb{\Z/2\Z}{\C^3}$. Let $Q$ be the McKay quiver of $G$ with the usual commutativity relations deriving from $xy=yx$, $xz=zx$ and $yz=zy$ (see Figure \ref{OpensMcKZ6}), and consider $M_{22}$ as a representation of $Q$. Then by choosing the basis element at every vector space $\C_\rho$ to be given by the unique element $\lambda_\rho\in M$, we have that $x\cdot x=a\cdot\frac{z}{x}$, $y\cdot1=b\cdot\frac{z}{x}$ and $z\cdot\frac{z^2}{x^2}=c\cdot x$ for some $a,b,c\in\C$. This implies that $a=\frac{x^3}{z}$, $b=\frac{xy}{z}$ and $c=\frac{z^3}{x^3}$, which are precisely the local coordinates of $\sigma(\Delta_{22})$. Since after change of basis any nonzero map can be chosen to be 1, the skeleton for $U_{22}$ in this case is formed by the linear maps equal to 1. 

{\renewcommand{\arraystretch}{1.25}
\begin{figure}[hd]
\begin{center}
\begin{pspicture}(0,0.5)(10,3.5)
	\psset{arcangle=20,nodesep=2pt}
% U_11
\scalebox{0.9}{
\rput(0,0.25){
% Vertices of the quiver 
	\rput(0,1.5){\rnode{0}{$S_{\rho_0^0}$}}
	\rput(1,3){\rnode{1}{$S_{\rho_1^0}$}}
	\rput(2.75,3){\rnode{2}{$S_{\rho_0^1}$}}
	\rput(3.75,1.5){\rnode{3}{$S_{\rho_1^1}$}}
	\rput(2.75,0){\rnode{4}{$S_{\rho_0^2}$}}
	\rput(1,0){\rnode{5}{$S_{\rho_1^2}$}}
% x-arrows
	\ncline{->}{0}{1}\Aput[0.05]{\footnotesize $x$}
	\ncline{->}{1}{2}\Aput[0.05]{\footnotesize $x$}
	\ncline{->}{2}{3}\Aput[0.05]{\footnotesize $x$}
	\ncline{->}{3}{4}\Aput[0.05]{\footnotesize $x$}
	\ncline{->}{4}{5}\Aput[0.05]{\footnotesize $x$}
	\ncline{->}{5}{0}\Aput[0.05]{\footnotesize $x$}
% y-arrows
	\ncarc{->}{0}{2}\Bput[0.025]{\footnotesize $y$}
	\ncarc{->}{2}{4}\Bput[0.025]{\footnotesize $y$}
	\ncarc{->}{4}{0}\Bput[0.025]{\footnotesize $y$}
	\ncarc{->}{1}{3}\Bput[0.025]{\footnotesize $y$}
	\ncarc{->}{3}{5}\Bput[0.025]{\footnotesize $y$}
	\ncarc{->}{5}{1}\Bput[0.025]{\footnotesize $y$}
% z-arrows
	\ncarc[arcangle=30]{->}{0}{3}\Bput[0.025]{\footnotesize $z$}
	\ncarc[arcangle=30]{->}{3}{0}\Bput[0.025]{\footnotesize $z$}
	\ncarc[arcangle=30]{->}{1}{4}\Bput[0.025]{\footnotesize $z$}
	\ncarc[arcangle=30]{->}{4}{1}\Bput[0.025]{\footnotesize $z$}
	\ncarc[arcangle=30]{->}{2}{5}\Bput[0.025]{\footnotesize $z$}
	\ncarc[arcangle=30]{->}{5}{2}\Bput[0.025]{\footnotesize $z$}
	}
\rput(7,0.25){
% Vertices of the quiver 
	\rput(0,1.5){\rnode{0}{$\C_1$}}
	\rput(1,3){\rnode{1}{$\C_x$}}
	\rput(2.75,3){\rnode{2}{$\C_\frac{z}{x}$}}
	\rput(3.75,1.5){\rnode{3}{$\C_z$}}
	\rput(2.75,0){\rnode{4}{$\C_\frac{z^2}{x^2}$}}
	\rput(1,0){\rnode{5}{$\C_\frac{z^2}{x}$}}
% x-arrows
	\ncline[linecolor=red]{->}{0}{1}\Aput[0.05]{\scriptsize ${\red{1}}$}
	\ncline{->}{1}{2}\Aput[0.05]{\scriptsize $a$}
	\ncline[linecolor=red]{->}{2}{3}\Aput[0.05]{\scriptsize ${\red{1}}$}
	\ncline{->}{3}{4}\Aput[0.05]{\scriptsize $a$}
	\ncline[linecolor=red]{->}{4}{5}\Aput[0.05]{\scriptsize ${\red{1}}$}
	\ncline{->}{5}{0}\Aput[0.05]{\footnotesize $ac$}
% y-arrows
	\ncarc{->}{0}{2}\Bput[0.025]{\footnotesize $b$}
	\ncarc{->}{2}{4}\Bput[0.025]{\footnotesize $b$}
	\ncarc{->}{4}{0}\Bput[0.025]{\footnotesize $bc$}
	\ncarc{->}{1}{3}\Bput[0.025]{\footnotesize $b$}
	\ncarc{->}{3}{5}\Bput[0.025]{\footnotesize $b$}
	\ncarc{->}{5}{1}\Bput[0.025]{\footnotesize $bc$}
% z-arrows
	\ncarc[arcangle=30,linecolor=red]{->}{0}{3}\Bput[0.025]{\footnotesize ${\red{1}}$}
	\ncarc[arcangle=30]{->}{3}{0}\Bput[0.025]{\footnotesize $ac$}
	\ncarc[arcangle=30]{->}{1}{4}\Bput[0.025]{\footnotesize $a$}
	\ncarc[arcangle=30]{->}{4}{1}\Bput[0.025]{\footnotesize $c$}
	\ncarc[arcangle=30,linecolor=red]{->}{2}{5}\Bput[0.025]{\footnotesize ${\red{1}}$}
	\ncarc[arcangle=30]{->}{5}{2}\Bput[0.025]{\footnotesize $ac$}
}}
\end{pspicture}
\end{center}
\caption{McKay quiver for $G=\frac{1}{6}(1,2,3)$ and the open set $U_{22}\cong\C^3_{a,b,c}$.}
\label{OpensMcKZ6}
\end{figure}
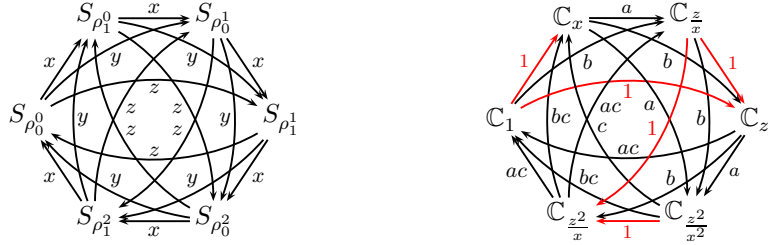}

The skeletons in every open set in $\Hilb{\Z/3\Z}{Y}$ are shown in Figure \ref{GconstZ3Z2}. Notice that in the skeleton for $U_{ij}$ it is possible to find the skeletons $U_1,U_2\subset Y$, repeated $|\Z/3\Z|=3$ times.

\begin{figure}[h]
\begin{center}
\begin{pspicture}(0,0)(9,5.5)
	\psset{arcangle=20,nodesep=2pt}
\scalebox{0.9}{
% U_11
\rput(0,4){
\scalebox{0.8}{
% Vertices of the quiver 
	\rput(0,1.25){\rnode{0}{1}}
	\rput(0.75,2.5){\rnode{1}{$y^2z$}}
	\rput(2.25,2.5){\rnode{2}{$y$}}
	\rput(3,1.25){\rnode{3}{$z$}}
	\rput(2.25,0){\rnode{4}{$y^2$}}
	\rput(0.75,0){\rnode{5}{$yz$}}
% x-arrows
% y-arrows
	\ncarc{->}{0}{2}
	\ncarc{->}{2}{4}
	\ncarc{->}{3}{5}
	\ncarc{->}{5}{1}
% z-arrows
	\ncline[linecolor=blue]{->}{0}{3}
	\ncline[linecolor=blue]{->}{4}{1}
	\ncline[linecolor=blue]{->}{2}{5}
	}}
% U_12
\rput(3.5,4){
\scalebox{0.85}{
% Vertices of the quiver 
	\rput(0,1.25){\rnode{0}{1}}
	\rput(0.75,2.5){\rnode{1}{$x$}}
	\rput(2.25,2.5){\rnode{2}{$y$}}
	\rput(3,1.25){\rnode{3}{$z$}}
	\rput(2.25,0){\rnode{4}{\Large $\frac{x}{z}$}}
	\rput(0.75,0){\rnode{5}{$yz$}}
% x-arrows
	\ncline{->}{0}{1}
% y-arrows
	\ncarc{->}{0}{2}
	\ncarc{->}{3}{5}
% z-arrows
	\ncline[linecolor=blue]{->}{0}{3}
	\ncline[linecolor=blue]{->}{4}{1}
	\ncline[linecolor=blue]{->}{2}{5}
	}}
% U_13
\rput(7,4){
\scalebox{0.85}{
% Vertices of the quiver 
	\rput(0,1.25){\rnode{0}{1}}
	\rput(0.75,2.5){\rnode{1}{$x$}}
	\rput(2.25,2.5){\rnode{2}{\Large $\frac{x^2}{z^2}$}}
	\rput(3,1.25){\rnode{3}{$z$}}
	\rput(2.25,0){\rnode{4}{\Large $\frac{x}{z}$}}
	\rput(0.75,0){\rnode{5}{\Large $\frac{x^2}{z}$}}
% x-arrows
	\ncline{->}{0}{1}
	\ncline{->}{4}{5}
% y-arrows
% z-arrows
	\ncline[linecolor=blue]{->}{0}{3}
	\ncline[linecolor=blue]{->}{4}{1}
	\ncline[linecolor=blue]{->}{2}{5}
	}}
% U_21
\rput(0,0.5){
\scalebox{0.85}{
% Vertices of the quiver 
	\rput(0,1.25){\rnode{0}{1}}
	\rput(0.75,2.5){\rnode{1}{$x$}}
	\rput(2.25,2.5){\rnode{2}{\Large $\frac{z}{x}$}}
	\rput(3,1.25){\rnode{3}{$z$}}
	\rput(2.25,0){\rnode{4}{\Large $\frac{z^2}{x^2}$}}
	\rput(0.75,0){\rnode{5}{\Large $\frac{z^2}{x}$}}
% x-arrows
	\ncline[linecolor=blue]{->}{0}{1}
	\ncline[linecolor=blue]{->}{2}{3}
	\ncline[linecolor=blue]{->}{4}{5}
% y-arrows
% z-arrows
	\ncline{->}{0}{3}
	\ncline{->}{2}{5}
	}}
% U_22
\rput(3.5,0.5){
\scalebox{0.85}{
% Vertices of the quiver 
	\rput(0,1.25){\rnode{0}{1}}
	\rput(0.75,2.5){\rnode{1}{$x$}}
	\rput(2.25,2.5){\rnode{2}{$x^2$}}
	\rput(3,1.25){\rnode{3}{$x^3$}}
	\rput(2.25,0){\rnode{4}{$x^4$}}
	\rput(0.75,0){\rnode{5}{$x^5$}}
% x-arrows
	\ncline[linecolor=blue]{->}{0}{1}
	\ncline{->}{1}{2}
	\ncline[linecolor=blue]{->}{2}{3}
	\ncline{->}{3}{4}
	\ncline[linecolor=blue]{->}{4}{5}
% y-arrows
% z-arrows
	}}
% U_23
\rput(7,0.5){
\scalebox{0.85}{
% Vertices of the quiver 
	\rput(0,1.25){\rnode{0}{1}}
	\rput(0.75,2.5){\rnode{1}{$x$}}
	\rput(2.25,2.5){\rnode{2}{$y$}}
	\rput(3,1.25){\rnode{3}{$xy$}}
	\rput(2.25,0){\rnode{4}{$y^2$}}
	\rput(0.75,0){\rnode{5}{$xy^2$}}
% x-arrows
	\ncline[linecolor=blue]{->}{0}{1}
	\ncline[linecolor=blue]{->}{2}{3}
	\ncline[linecolor=blue]{->}{4}{5}
% y-arrows
	\ncarc{->}{0}{2}
	\ncarc{->}{2}{4}
	\ncarc{->}{1}{3}
	\ncarc{->}{3}{5}
	}}
\rput(1.5,3.5){$U_{11}$}
\rput(5,3.5){$U_{12}$}
\rput(8.5,3.5){$U_{13}$}
\rput(1.5,-0.25){$U_{21}$}
\rput(5,-0.25){$U_{22}$}
\rput(8.5,-0.25){$U_{23}$}
	}
\end{pspicture}
\end{center}
\caption{Skeletons for $\Hilb{\Z/3}{\Hilb{\Z/2}{\C^3}}$ with the corresponding $G$-constellations.}
\label{GconstZ3Z2}
\end{figure}
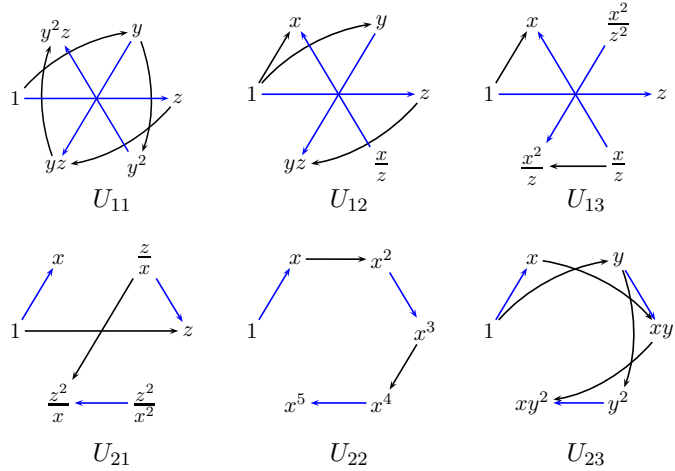

\end{exa}

\newpage
\section{Non-Abelian examples}\label{Examples}

\subsection{Dihedral groups $D_{2n}\subset\SO(3)$.}\label{exaDn}

These groups are generated by
\[D_{2n}:=\Span{\alpha=\frac{1}{n}(1,-1,0),\beta=\left(\begin{smallmatrix}0&1&0\\1&0&0\\0&0&-1\end{smallmatrix}\right)}
\] 
and they have order $2n$. The normal subgroup $N:=\Span{\alpha}$ has $n$ 1-dimensional irreducible representations $\sigma_i(\alpha)=\varepsilon^i$ where $\varepsilon$ is an $n$-th root of unity and $i=0,\ldots,n-1$. The action of $G/N=\Span{\overline{\beta}}\cong\Z/2\Z$ give the irreducible representations of $G$ as in Table \ref{IrrD2n} depending on the parity of $n$. 

{\renewcommand{\arraystretch}{1.25}
\begin{table}[h]
\begin{center}
\begin{small}
\begin{tabular}{|c|c|c|c|c|}
\multicolumn{1}{c}{$\sigma_0$}
 &  \multicolumn{1}{c}{$\{\sigma_1,\sigma_{2m-1}\}$}
 & \multicolumn{1}{c}{} 
 & \multicolumn{1}{c}{$\{\sigma_{m-1},\sigma_{m+1}\}$} 
 & \multicolumn{1}{c}{$\sigma_m$} \\
\cline{1-5}
$\rho_0^0$ & \multirow{2}{*}{$\rho_1^0$} & \multirow{2}{*}{$\cdots$} & \multirow{2}{*}{$\rho_{m-1}^0$} & \multirow{1}{*}{$\rho_m^0$} \\
\cline{1-1}\cline{5-5} 
$\rho_0^1$ & &  &  & $\rho_n^1$ \\
\cline{1-5}
\multicolumn{1}{c}{}
 &  \multicolumn{1}{c}{}
 & \multicolumn{1}{c}{$(a)$} 
 & \multicolumn{1}{c}{} 
 & \multicolumn{1}{c}{}
\end{tabular}
\begin{tabular}{cc|c|c|c|c|}
 & \multicolumn{1}{c}{}
 & \multicolumn{1}{c}{$\sigma_0$}
 &  \multicolumn{1}{c}{$\{\sigma_1,\sigma_{2m}\}$}
 & \multicolumn{1}{c}{} 
 & \multicolumn{1}{c}{$\{\sigma_{m},\sigma_{m+1}\}$} \\
\cline{3-6}
&& $\rho_0^0$ & \multirow{2}{*}{$\rho_1^0$} & \multirow{2}{*}{$\cdots$} & \multirow{2}{*}{$\rho_{m}^0$}  \\
\cline{3-3}
&& $\rho_0^1$ & &  & \\
\cline{3-6}
 & \multicolumn{1}{c}{}
 &  \multicolumn{1}{c}{}
 &  \multicolumn{1}{c}{}
 & \multicolumn{1}{c}{$(b)$} 
 & \multicolumn{1}{c}{} 
\end{tabular}
\end{small}
\end{center}
\caption{$\Irr(D_{2n})$ for (a) $n=2m$ even, and (b) $n=2m+1$ odd.}
\label{IrrD2n}
\end{table}}

The dimension vectors are ${\Large{\begin{smallmatrix}1\\1\end{smallmatrix}}}2\ldots2{\Large{\begin{smallmatrix}1\\1\end{smallmatrix}}}$ and ${\Large{\begin{smallmatrix}1\\1\end{smallmatrix}}}2\ldots2$ respectively. In this case, as well as for any subgroup $G\subset\SO(3)$, the fibre over the origin $f^{-1}(0)$ of any crepant resolution $f:Y\to\C^3/G$ has dimension 1. The description of $f^{-1}(0)$ in the case of $Y=\Hilb{G}{\C^3}$ was first given by \cite{GNS2}, see \cite{NdCS} for the rest of the crepant resolutions of $\C^3/G$.

The open cover of $\Hilb{N}{\C^3}$ is given by $n$ open sets $\bigcup_{j=1}^{n}U_j$, covering $n-1$ rational curves $E_i$ for $i=1,\ldots,n-1$. The action $G/N$ on $\Hilb{N}{\C^3}$ identifies $U_i$ and $U_{n-i+1}$ for $i=1,\ldots,n$. 

If $n=2m$ is even then $E_m$ is fixed by $G/N$ having two fixed lines $L_+$ and $L_-$ crossing transversally the $\PP^1$ covered by $U_m$ and $U_{m+1}$. They give rise to $E_+$ and $E_-$ respectively. If $n=2m+1$ is odd then the open set $U_{m+1}$ is fixed by $G/N$ and there is just one fixed line $L$, producing the new rational curve $E$.

Diagonalizing the action of $G/N$ we see that in both cases these singularities are of type $\frac{1}{2}(1,1,0)$. By blowing up these singular lines it follows that the dual graph of the fibre over the origin of the singularity is shown in Figure \ref{ExclDnHoH}.

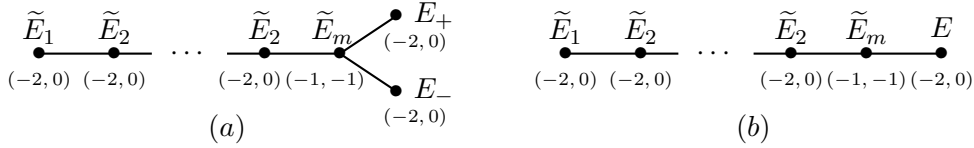
\begin{figure}[h]
\begin{center}
\begin{pspicture}(0,0.5)(12,2.25)
	\psset{arcangle=15,nodesep=2pt}
\rput(0,0.5){
	\rput(0,1){$\bullet$}\rput(0,1.35){$\widetilde{E}_1$}\rput(0,0.65){\tiny $(-2,0)$}
	\rput(1,1){$\bullet$}\rput(1,1.35){$\widetilde{E}_2$}\rput(1,0.65){\tiny $(-2,0)$}
	\rput(3,1){$\bullet$}\rput(3,1.35){$\widetilde{E}_2$}\rput(2.8,0.65){\tiny $(-2,0)$}
	\rput(4,1){$\bullet$}\rput(3.9,1.35){$\widetilde{E}_{m}$}\rput(3.8,0.65){\tiny $(-1,-1)$}	
	\rput(4.75,1.5){$\bullet$}\rput(5.25,1.5){$E_+$}\rput(5,1.15){\tiny $(-2,0)$}
	\rput(4.75,0.5){$\bullet$}\rput(5.25,0.5){$E_-$}\rput(5,0.15){\tiny $(-2,0)$}
	\psline(0,1)(1.5,1)
	\rput(2,1){$\cdots$}
	\psline(2.5,1)(3,1)
	\psline(3,1)(4,1)
	\psline(4,1)(4.75,1.5)
	\psline(4,1)(4.75,0.5)
	\rput(2.5,0){$(a)$}
	}
\rput(7,0.5){
	\rput(0,1){$\bullet$}\rput(0,1.35){$\widetilde{E}_1$}\rput(0,0.65){\tiny $(-2,0)$}
	\rput(1,1){$\bullet$}\rput(1,1.35){$\widetilde{E}_2$}\rput(1,0.65){\tiny $(-2,0)$}
	\rput(3,1){$\bullet$}\rput(3,1.35){$\widetilde{E}_2$}\rput(3,0.65){\tiny $(-2,0)$}
	\rput(4,1){$\bullet$}\rput(4,1.35){$\widetilde{E}_{m}$}\rput(4,0.65){\tiny $(-1,-1)$}	
	\rput(5,1){$\bullet$}\rput(5,1.35){$E$}\rput(5,0.65){\tiny $(-2,0)$}
	\psline(0,1)(1.5,1)
	\rput(2,1){$\cdots$}
	\psline(2.5,1)(3,1)
	\psline(3,1)(4,1)
	\psline(4,1)(5,1)
	\rput(2.5,0){$(b)$}
	}
\end{pspicture}
\end{center}
\caption{Dual graph of $f^{-1}(0)$ for $\Hilb{G/N}{\Hilb{N}{\C^3}}$ for (a) $n=2m$ even, and (b) $n=2m+1$ odd. The curve $\widetilde{E}_i$ denote the strict transform of $E_i$ and the numbers denote the degree of the normal bundle at every curve.}
\label{ExclDnHoH}
\end{figure}

 The fibre over the origin in $\Hilb{G}{\C^3}$ with the degrees of the normal bundles in each of the rational curves is shown in Figure \ref{ExclDnHilb} (see \cite{GNS2} and \cite{NdCS} for details). We can therefore see the difference between $\Hilb{G/N}{\Hilb{N}{\C^3}}$ and $\Hilb{G}{\C^3}$: in the case $n$ even the graph is different, whereas in the case $n$ odd the difference resides in the degrees of the normal bundles. This concludes the proof of the $D_{2n}$ case in Theorem \ref{thm:isoDim3} (iii).

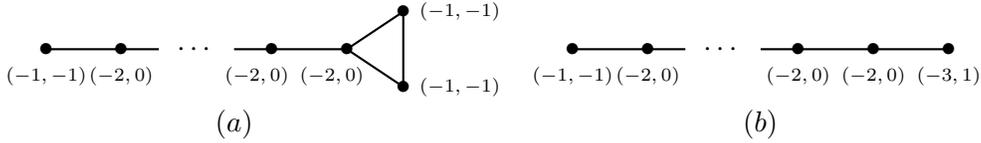
\begin{figure}[h]
\begin{center}
\begin{pspicture}(0,0.5)(12,2.25)
	\psset{arcangle=15,nodesep=2pt}
\rput(0,0.5){
	\rput(0,1){$\bullet$}\rput(0,1.35){}\rput(0,0.65){\tiny $(-1,-1)$}
	\rput(1,1){$\bullet$}\rput(1,1.35){}\rput(1,0.65){\tiny $(-2,0)$}
	\rput(3,1){$\bullet$}\rput(3,1.35){}\rput(2.8,0.65){\tiny $(-2,0)$}
	\rput(4,1){$\bullet$}\rput(3.9,1.35){}\rput(3.8,0.65){\tiny $(-2,0)$}	
	\rput(4.75,1.5){$\bullet$}\rput(5.25,1.5){}\rput(5.5,1.5){\tiny $(-1,-1)$}
	\rput(4.75,0.5){$\bullet$}\rput(5.25,0.5){}\rput(5.5,0.5){\tiny $(-1,-1)$}
	\psline(0,1)(1.5,1)
	\rput(2,1){$\cdots$}
	\psline(2.5,1)(3,1)
	\psline(3,1)(4,1)
	\psline(4,1)(4.75,1.5)
	\psline(4,1)(4.75,0.5)
	\psline(4.75,1.5)(4.75,0.5)
	\rput(2.5,0){$(a)$}
	}
\rput(7,0.5){
	\rput(0,1){$\bullet$}\rput(0,1.35){}\rput(0,0.65){\tiny $(-1,-1)$}
	\rput(1,1){$\bullet$}\rput(1,1.35){}\rput(1,0.65){\tiny $(-2,0)$}
	\rput(3,1){$\bullet$}\rput(3,1.35){}\rput(3,0.65){\tiny $(-2,0)$}
	\rput(4,1){$\bullet$}\rput(4,1.35){}\rput(4,0.65){\tiny $(-2,0)$}	
	\rput(5,1){$\bullet$}\rput(5,1.35){}\rput(5,0.65){\tiny $(-3,1)$}
	\psline(0,1)(1.5,1)
	\rput(2,1){$\cdots$}
	\psline(2.5,1)(3,1)
	\psline(3,1)(4,1)
	\psline(4,1)(5,1)
	\rput(2.5,0){$(b)$}
	}
\end{pspicture}
\end{center}
\caption{Dual graph of $f^{-1}(0)$ for $\Hilb{G}{\C^3}$ for (a) $n=2m$ even, and (b) $n=2m+1$ odd.}
\label{ExclDnHilb}
\end{figure}

We illustrate the construction of $\Hilb{G/N}{\Hilb{N}{\C^3}}$ in this case by an example. The general case is analogous.

\begin{exa} Let $G=D_{12}$. In this case the McKay quiver with relations $(Q,R)$ is given in Figure \ref{McKayD12} where the relations $R$ provide the Morita equivalence between $\mathbb CQ/\langle R \rangle$ and $S \ast G$, and are obtained following \cite{BSW}.

\begin{figure}[h]
\begin{center}
\begin{pspicture}(0,-0.75)(12,1.25)
	\psset{arrowsize=4pt,arrowlength=2,arcangle=15,nodesep=1pt}
%%%%%%%%%%%%%%%%%%%%%%%% McKay quiver for D_12
\rput(-0.5,0){
\scalebox{0.85}{
% Vertices of the quiver 
	\rput(-0,1.2){\rnode{0}{$\rho_0^0$}}
	\rput(-0,-1.2){\rnode{0'}{$\rho_0^1$}}
	\rput(2,0){\rnode{1}{$\rho_1^0$}}
	\rput(4.5,0){\rnode{2}{$\rho_2^0$}}
	\rput(6.5,1.2){\rnode{3}{$\rho_3^0$}}
	\rput(6.5,-1.2){\rnode{3'}{$\rho_3^1$}}	
% points for loops
	\rput(2,0.3){\rnode{u1}{}}	
	\rput(4.5,0.3){\rnode{u2}{}}	
% arrows
	\ncarc{->}{0}{0'}\mput*[0.1]{$a$}	% a
	\ncarc{->}{0'}{0}\Aput*[0.1]{$A$}	% A
	\ncarc{->}{0}{1}\Aput*[0.1]{$c$}	% c
	\ncarc{->}{1}{0}\mput*[0.1]{$C$}	% C
	\ncarc{->}{0'}{1}\mput*[0.1]{$b$}	% b
	\ncarc{->}{1}{0'}\Aput*[0.1]{$B$}	% B
	\ncarc{->}{1}{2}	\Aput*[0.1]{$d$}	% d
	\ncarc{->}{2}{1}	\Aput*[0.1]{$D$}	% D
	\ncarc{->}{2}{3}	\Aput*[0.1]{$C'$}	% C'
	\ncarc{->}{3}{2}	\mput*[0.1]{$c'$}	% c'
	\ncarc{->}{2}{3'}\mput*[0.1]{$B'$}	% B'
	\ncarc{->}{3'}{2}\Aput*[0.1]{$b'$}	% b'
	\ncarc[arcangle=-15]{->}{3}{3'}\mput*[0.1]{$a'$}	% a'
	\ncarc[arcangle=-15]{->}{3'}{3}\Bput*[0.1]{$A'$}	% A'
% loops
	 \nccircle[angleA=-25,nodesep=3pt]{->}{u1}{.3cm}\Bput[0.05]{$u$}
	 \nccircle[angleA=25,nodesep=3pt]{->}{u2}{.3cm}\Bput[0.05]{$v$}	
	 }}	

%%%%%%%%%%%%%%%%%%%% Relations	
\rput(9.5,1){\scriptsize $bC=0, b'C'=0, cB=0, c'B'=0, $}
\rput(9.5,0.5){\scriptsize $Ca=uB, C'a'=vB', Ac=bu, A'c'=b'v,$}
\rput(9.5,0){\scriptsize $BA=uC, B'B'=vC', ab=cu, a'b'=c'v$}
\rput(9.5,-0.5){\scriptsize $Bb+Cc=dD, B'b'+C'c'=Dd$}
\rput(9.5,-1){\scriptsize $Du=vD, ud=dv$}

\end{pspicture}
\end{center}
\caption{The McKay of $D_{12}$ with relations.}
\label{McKayD12}
\end{figure}
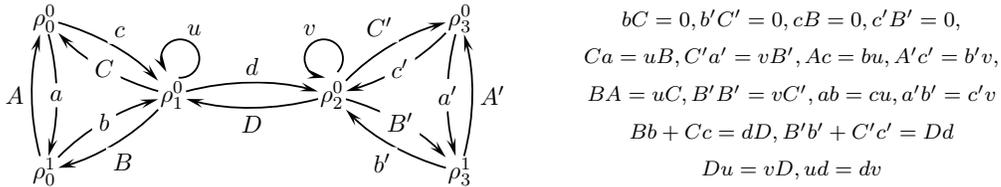

In this case the stability condition given in Definition \ref{theta} is shown in Table \ref{ThetaD12}, where $a_i,b>0$ for $i=1,\ldots,5$ and $0<\varepsilon<<1$.

{\renewcommand{\arraystretch}{1.4}
\begin{table}[h]
\begin{center}
\begin{small}
$\theta:=$
\begin{tabular}{|c|c|c|c|}
\multicolumn{1}{c}{$\sigma_0$}
 &  \multicolumn{1}{c}{$\{\sigma_1,\sigma_{5}\}$}
 & \multicolumn{1}{c}{$\{\sigma_{2},\sigma_{4}\}$} 
 & \multicolumn{1}{c}{$\sigma_3$} \\
\cline{1-4}
$-\sum_{i=1}^5a_i-\varepsilon b$ & \multirow{2}{*}{$a_1+a_5$} & \multirow{2}{*}{$a_2+a_4$} & \multirow{1}{*}{$a_3$} \\
\cline{1-1}\cline{4-4} 
$-\sum_{i=1}^5a_i+\varepsilon b$ & &  &  $a_3$ \\
\cline{1-4}
\multicolumn{1}{c}{}
 &  \multicolumn{1}{c}{}
 & \multicolumn{1}{c}{} 
 & \multicolumn{1}{c}{}
\end{tabular}
\end{small}
\end{center}
\caption{Stability condition for $\Hilb{D_{12}/N}{\Hilb{N}{\C^3}}$ with $N=\frac{1}{6}(1,-1,0)$.}
\label{ThetaD12}
\end{table}}

We start by considering the gluing of $U_1$ and $U_2$ with $U_5$ and $U_6$ by the action of $G/N$. Since $U_3$ and $U_4$ contain the fixed part we will treat them separately. We need to consider the $N$-constellations at the origin of each open chart, so after choosing a basis for every $H^0(gZ)$ for $g\in G/N$ consisting of $N$-graphs, we obtain the $G$-constellations shown in Figure \ref{OpOrbits}.

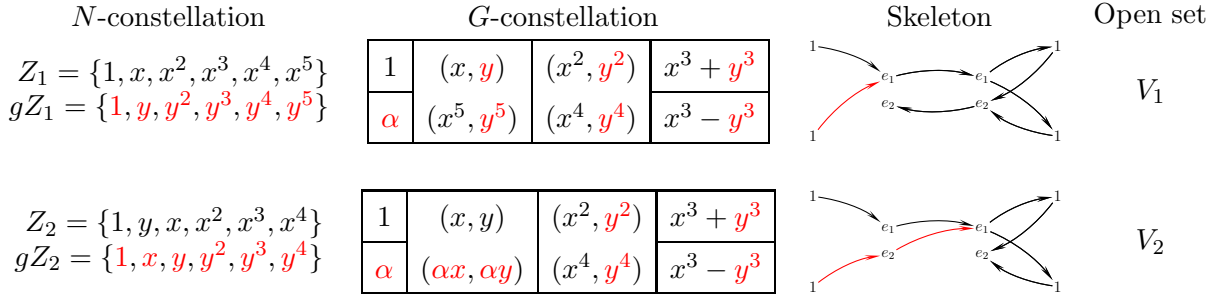
\begin{figure}[h]
\begin{center}
\begin{pspicture}(0,-0.75)(15,3.5)
	\psset{arrowsize=4pt,arrowlength=3,arcangle=15,nodesep=1pt}

\rput(1.75,3){$N$-constellation}
\rput(7,3){$G$-constellation}
\rput(12,3){Skeleton}
\rput(14.8,3){Open set}
\rput(14.8,2){$V_1$}
\rput(14.8,0){$V_2$}

%%%%%%%%%%%%%%%%%%%%%%%% V_1
\rput(1.75,2){$
\begin{array}{r}
Z_1 = \{1,x,x^2,x^3,x^4,x^5\} \\
gZ_1 = \{ {\red{1}},{\red{y}},{\red{y^2}},{\red{y^3}},{\red{y^4}},{\red{y^5}}\}  \\
\end{array}
$}
\rput(7,2){
{\renewcommand{\arraystretch}{1.5}
\begin{tabular}{|c|c|c|c|}
\cline{1-4}
$1$ & $(x,{\red{y}})$ & $(x^2,{\red{y^2}})$ & $x^3+{\red{y^3}}$ \\
\cline{1-1}\cline{4-4} 
$\red{\alpha}$ & $(x^5,{\red{y^5}})$ & $(x^4,{\red{y^4}})$ & $x^3-{\red{y^3}}$ \\
\cline{1-4}
\end{tabular}}
}
\rput(10.25,2){
\scalebox{0.5}{
% Vertices of the quiver 
	\rput(-0,1.2){\rnode{0}{\large $1$}}
	\rput(-0,-1.2){\rnode{0'}{\large $1$}}
	\rput(2,0.35){\rnode{1a}{\large $e_1$}}
	\rput(2,-0.35){\rnode{1b}{\large $e_2$}}
	\rput(4.5,0.35){\rnode{2a}{\large $e_1$}}
	\rput(4.5,-0.35){\rnode{2b}{\large $e_2$}}
	\rput(6.5,1.2){\rnode{3}{\large $1$}}
	\rput(6.5,-1.2){\rnode{3'}{\large $1$}}	
% arrows
	\ncarc{->}{0}{1a}	
	\ncarc{->}{1a}{2a}
	\ncarc{->}{1a}{2a}
	\ncarc{->}{2a}{3}
	\ncarc{->}{2a}{3}
	\ncarc{->}{2a}{3'}
	\ncarc{->}{2a}{3'}
	\ncarc{->}{3}{2b}
	\ncarc{->}{3}{2b}
	\ncarc{->}{3'}{2b}
	\ncarc{->}{3'}{2b}
	\ncarc{->}{2b}{1b}
	\ncarc{->}{2b}{1b}	
	\ncarc[linecolor=red]{->}{0'}{1a}
	}}

%%%%%%%%%%%%%%%%%%%%%%%% V_2
\rput(1.75,0){$
\begin{array}{r}
Z_2 = \{1,y,x,x^2,x^3,x^4\} \\
gZ_2 = \{{\red{1}},{\red{x}},{\red{y}},{\red{y^2}},{\red{y^3}},{\red{y^4}}\} \\
\end{array}
$}
\rput(7,0){
{\renewcommand{\arraystretch}{1.5}
\begin{tabular}{|c|c|c|c|}
\cline{1-4}
$1$ & $(x,y)$ & $(x^2,{\red{y^2}})$ & $x^3+{\red{y^3}}$ \\
\cline{1-1}\cline{4-4} 
$\red{\alpha}$ & $({\red{\alpha x}},{\red{\alpha y}})$ & $(x^4,{\red{y^4}})$ & $x^3-{\red{y^3}}$ \\
\cline{1-4}
\end{tabular}}
}
\rput(10.25,0){
\scalebox{0.5}{
% Vertices of the quiver 
	\rput(-0,1.2){\rnode{0}{\large $1$}}
	\rput(-0,-1.2){\rnode{0'}{\large $1$}}
	\rput(2,0.35){\rnode{1a}{\large $e_1$}}
	\rput(2,-0.35){\rnode{1b}{\large $e_2$}}
	\rput(4.5,0.35){\rnode{2a}{\large $e_1$}}
	\rput(4.5,-0.35){\rnode{2b}{\large $e_2$}}
	\rput(6.5,1.2){\rnode{3}{\large $1$}}
	\rput(6.5,-1.2){\rnode{3'}{\large $1$}}		
% arrows
	\ncarc{->}{0}{1a}	
	\ncarc{->}{1a}{2a}
	\ncarc{->}{2a}{3}
	\ncarc{->}{2a}{3}
	\ncarc{->}{2a}{3'}
	\ncarc{->}{2a}{3'}
	\ncarc{->}{3}{2b}
	\ncarc{->}{3}{2b}
	\ncarc{->}{3'}{2b}
	\ncarc{->}{3'}{2b}
	\ncarc[linecolor=red]{->}{1b}{2a}	
	\ncarc[linecolor=red]{->}{0'}{1b}
	}}

\end{pspicture}
\caption{$G$-constellations at the open sets $V_1$ and $V_2$.}
\label{OpOrbits}
\end{center}
\end{figure}

In the skeleton of the open sets the dots stacked vertically in the two middle vertices of the McKay quiver denote the two linearly independent vectors $e_1$ and $e_2$ in the corresponding vector space at that vertex. 

With a similar calculation as in Section \ref{LocalCoord} we can calculate the local coordinates to obtain 
\begin{align*}
V_1&\cong\Spec\C\bigg[\frac{z(x^3+y^3)}{x^3-y^3},\frac{xy}{(x^3+y^3)^2},(x^3+y^3)^2\bigg], \\
V_2&\cong\Spec\C\bigg[\frac{z(x^3+y^3)}{x^3-y^3},\frac{x^2y^2}{(x^3+y^3)^2},\frac{(x^3+y^3)^2}{xy}\bigg]
\end{align*}
and $\alpha=\frac{x^3-y^3}{x^3+y^3}$.

Let us now consider the $G$-constellations arising from the blowup of the fixed lines $L_+$ and $L_-$. In the open set $U_3\cong\C^3_{a,b,c}$ with $a=x^4/y^2$, $b=y^3/x^3$ and $c=z$, every $N$-cluster is given by the ideal $I_{a,b,c}=(x^4-ay^2,y^3-bx^3,xy-ab,z-c)$. Then the lines $L_\pm$ are defined by $b=\pm1$, $c=0$, which means that the ideals defining the lines are $I_{L_\pm}=(x^4-ay^2,y^3\mp x^3,xy\mp a,z)$. Therefore we can choose as basis for the $N$-constellations at these lines the $N$-graphs $\Gamma_\pm=\{ 1,x,y,x^2,y^2,x^3\pm y^3 \}$, which are invariant under the action of $G/N$. 

For the line $L_+$ we have $b=1$. By the change of coordinate $b^+=1-b=\frac{x^3-y^3}{x^3}$, we have that the action of $G/N$ as $\frac{1}{2}(1,1,0)$ is defined on $\C^3_{A,B,C}$ where $A=c$, $B=\frac{2-b^+}{b^+}$, $C=a(b^+)^2$. This implies that the rational curve $E_+$ is covered by the open sets $V_3\cong\C^3_{B^2,A/B,C}$, $V_6\cong\C^3_{A^2,B/A,C}$, and $E^+$ is given by the ratio $(A:B)=(z(x^3-y^3):x^3+y^3)$. In terms of $G$-constellations, the calculation is the same as in the Abelian case, where $V_3$ and $V_6$ are the open cover of $\Hilb{\frac{1}{2}(1,1,0)}{\C^3_{A,B,C}}$. See Figure \ref{OpFixLines}. In the case of $V_3$ for instance, we have the nonzero maps coming from the $N$-constellations $\Gamma_+$ and $\Gamma_+\cdot B$ (generated from $\rho_0^0$ and $\rho_0^1$ respectively) and the extra arrow comes from the fact that $(x^3+y^3)B=(x^3-y^3)+\frac{4x^3y^3}{x^3-y^3}$, which is induced by the $G/N$-equivariant map $\Gamma_+\to B\Gamma_+$.

\begin{figure}[h]
\begin{center}
\begin{pspicture}(0,-0.75)(15,6)
	\psset{arrowsize=4pt,arrowlength=3,arcangle=15,nodesep=1pt}

\scalebox{0.95}{
\rput(6,6.25){$G$-constellation}
\rput(12.25,6.25){Skeleton}
\rput(15.5,6.25){Open set}
\rput(15.5,5.25){$V_3$}
\rput(15.5,3.5){$V_6$}
\rput(15.5,1.75){$V_4$}
\rput(15.5,0){$V_5$}

%%%%%%%%%%%%%%%%%%%%%%%%	V_3
\rput(1,5.25){$
\begin{array}{r}
\Gamma_+\cdot\{1,{\red{B}}\}: \\
\end{array}
$}
\rput(5.9,5.25){
{\renewcommand{\arraystretch}{1.5}
\begin{tabular}{|c|c|c|c|}
\cline{1-4}
$1$ & $(x,y)$ & $(x^2,y^2)$ & $x^3+y^3$ \\
\cline{1-1}\cline{4-4} 
$\red{B}$ & $({\red{xB}},{\red{yB}})$ & $({\red{x^2B}},{\red{y^2B}})$ & ${\red{(x^3+y^3)B}}$ \\
\cline{1-4}
\end{tabular}}
}
\rput(10.5,5.25){
\scalebox{0.5}{
% Vertices of the quiver 
	\rput(-0,1.2){\rnode{0}{\large $1$}}
	\rput(-0,-1.2){\rnode{0'}{\large $1$}}
	\rput(2,0.35){\rnode{1a}{\large $e_1$}}
	\rput(2,-0.35){\rnode{1b}{\large $e_2$}}
	\rput(4.5,0.35){\rnode{2a}{\large $e_1$}}
	\rput(4.5,-0.35){\rnode{2b}{\large $e_2$}}
	\rput(6.5,1.2){\rnode{3}{\large $1$}}
	\rput(6.5,-1.2){\rnode{3'}{\large $1$}}		
% arrows
	\ncarc{->}{0}{1a}	
	\ncarc{->}{1a}{2a}
	\ncarc{->}{2a}{3}
	\ncarc[linecolor=blue]{->}{2a}{3'}
	\ncarc[linecolor=red]{->}{2b}{3'}
	\ncarc[linecolor=red]{->}{1b}{2b}	
	\ncarc[linecolor=red]{->}{0'}{1b}
	}}

%%%%%%%%%%%%%%%%%%%%%%%% V_6
\rput(1,3.5){$
\begin{array}{r}
\Gamma_+\cdot\{1,{\red{z}}\}: \\\end{array}
$}
\rput(5.9,3.5){
{\renewcommand{\arraystretch}{1.5}
\begin{tabular}{|c|c|c|c|}
\cline{1-4}
$1$ & $(x,y)$ & $(x^2,y^2)$ & $x^3+y^3$ \\
\cline{1-1}\cline{4-4} 
$\red{z}$ & $({\red{xz}},{\red{yz}})$ & $({\red{x^2z}},{\red{y^2z}})$ & ${\red{(x^3+y^3)z}}$ \\
\cline{1-4}
\end{tabular}}
}
\rput(10.5,3.5){
\scalebox{0.5}{
% Vertices of the quiver 
	\rput(-0,1.2){\rnode{0}{\large $1$}}
	\rput(-0,-1.2){\rnode{0'}{\large $1$}}
	\rput(2,0.35){\rnode{1a}{\large $e_1$}}
	\rput(2,-0.35){\rnode{1b}{\large $e_2$}}
	\rput(4.5,0.35){\rnode{2a}{\large $e_1$}}
	\rput(4.5,-0.35){\rnode{2b}{\large $e_2$}}
	\rput(6.5,1.2){\rnode{3}{\large $1$}}
	\rput(6.5,-1.2){\rnode{3'}{\large $1$}}	
% arrows
	\ncarc{->}{0}{1a}	
	\ncarc{->}{1a}{2a}
	\ncarc{->}{2a}{3}
	\ncarc[linecolor=red]{->}{0'}{1b}
	\ncarc[linecolor=red]{->}{1b}{2b}	
	\ncarc[linecolor=red]{->}{2b}{3'}
	\ncarc[linecolor=blue]{->}{0}{0'}
	\ncarc[arrowlength=2,arcangle=0,linecolor=blue,nodesep=0pt]{->}{1a}{1b}
	\ncarc[arrowlength=2,arcangle=0,linecolor=blue,nodesep=0pt]{->}{2a}{2b}
	\ncarc[arcangle=-15,linecolor=blue]{->}{3}{3'}
	}}

%%%%%%%%%%%%%%%%%%%%%%%%	V_4
\rput(1,1.75){$
\begin{array}{r}
\Gamma_-\cdot\{1,{\red{B'}}\}: \\
\end{array}
$}
\rput(5.9,1.75){
{\renewcommand{\arraystretch}{1.5}
\begin{tabular}{|c|c|c|c|}
\cline{1-4}
$1$ & $(x,y)$ & $(x^2,y^2)$ & $\red{{(x^3-y^3)B'}}$ \\
\cline{1-1}\cline{4-4} 
$\red{B'}$ & $({\red{xB'}},{\red{yB'}})$ & $({\red{x^2B'}},{\red{y^2B'}})$ & $x^3-y^3$ \\
\cline{1-4}
\end{tabular}}
}
\rput(10.5,1.75){
\scalebox{0.5}{
% Vertices of the quiver 
	\rput(-0,1.2){\rnode{0}{\large $1$}}
	\rput(-0,-1.2){\rnode{0'}{\large $1$}}
	\rput(2,0.35){\rnode{1a}{\large $e_1$}}
	\rput(2,-0.35){\rnode{1b}{\large $e_2$}}
	\rput(4.5,0.35){\rnode{2a}{\large $e_1$}}
	\rput(4.5,-0.35){\rnode{2b}{\large $e_2$}}
	\rput(6.5,1.2){\rnode{3}{\large $1$}}
	\rput(6.5,-1.2){\rnode{3'}{\large $1$}}			
% arrows
	\ncarc{->}{0}{1a}	
	\ncarc{->}{1a}{2a}
	\ncarc{->}{2a}{3}
	\ncarc{->}{2a}{3'}
	\ncarc[linecolor=blue]{->}{2a}{3}	
	\ncarc[linecolor=red]{->}{2b}{3}
	\ncarc[linecolor=red]{->}{1b}{2b}	
	\ncarc[linecolor=red]{->}{0'}{1b}
	}}

%%%%%%%%%%%%%%%%%%%%%%%% V_5
\rput(1,0){$
\begin{array}{r}
\Gamma_+\cdot\{1,{\red{z}}\}: \\\end{array}
$}
\rput(5.9,0){
{\renewcommand{\arraystretch}{1.5}
\begin{tabular}{|c|c|c|c|}
\cline{1-4}
$1$ & $(x,y)$ & $(x^2,y^2)$ & ${\red{z(x^3-y^3)}}$ \\
\cline{1-1}\cline{4-4} 
$\red{z}$ & $({\red{xz}},{\red{yz}})$ & $({\red{x^2z}},{\red{y^2z}})$ & $x^3-y^3$ \\
\cline{1-4}
\end{tabular}}
}
\rput(10.5,0){
\scalebox{0.5}{
% Vertices of the quiver 
	\rput(-0,1.2){\rnode{0}{\large $1$}}
	\rput(-0,-1.2){\rnode{0'}{\large $1$}}
	\rput(2,0.35){\rnode{1a}{\large $e_1$}}
	\rput(2,-0.35){\rnode{1b}{\large $e_2$}}
	\rput(4.5,0.35){\rnode{2a}{\large $e_1$}}
	\rput(4.5,-0.35){\rnode{2b}{\large $e_2$}}
	\rput(6.5,1.2){\rnode{3}{\large $1$}}
	\rput(6.5,-1.2){\rnode{3'}{\large $1$}}	
% arrows
	\ncarc{->}{0}{1a}	
	\ncarc{->}{1a}{2a}
	\ncarc{->}{2a}{3'}
	\ncarc[linecolor=red]{->}{0'}{1b}
	\ncarc[linecolor=red]{->}{1b}{2b}	
	\ncarc[linecolor=red]{->}{2b}{3}
	\ncarc[linecolor=blue]{->}{0}{0'}
	\ncarc[arrowlength=2,arcangle=0,linecolor=blue,nodesep=0pt]{->}{1a}{1b}
	\ncarc[arrowlength=2,arcangle=0,linecolor=blue,nodesep=0pt]{->}{2a}{2b}
	\ncarc[arcangle=-15,linecolor=blue]{->}{3'}{3}
	}}}
\end{pspicture}
\caption{$G$-constellations at the open sets covering the exceptional curves $E^+$ and $E^-$, where $B=\frac{x^3+y^3}{x^3-y^3}$ and $B'=\frac{x^3-y^3}{x^3+y^3}$.}
\label{OpFixLines}
\end{center}
\end{figure}
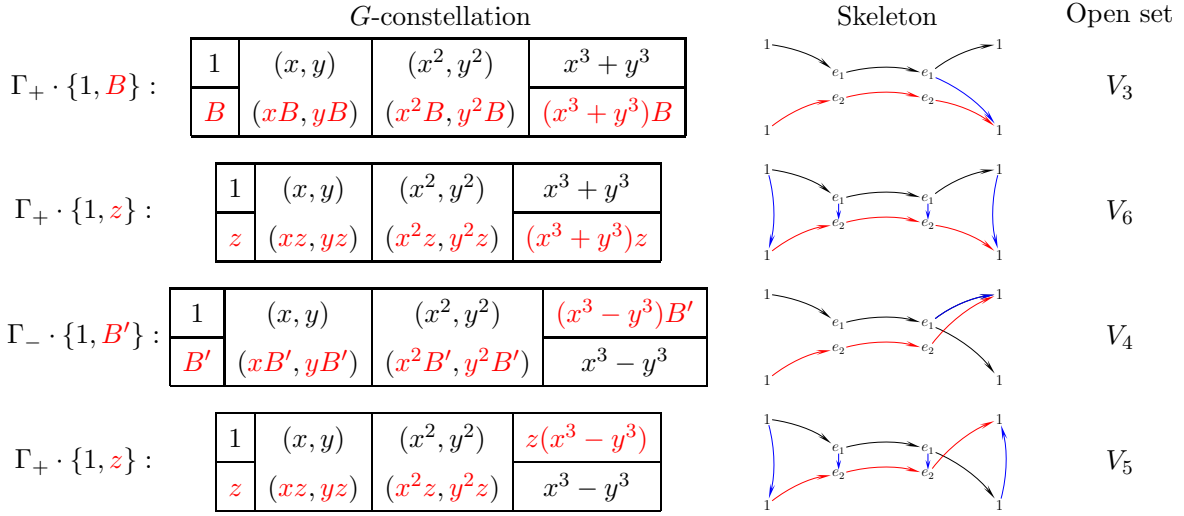

Similarly, for the line $L_-$ we have $b=-1$. By the change of coordinate $b^-=1-b=\frac{x^3+y^3}{x^3}$, the action of $G/N$ as $\frac{1}{2}(1,1,0)$ is defined on $\C^3_{A',B',C'}$ where $A'=c$, $B'=\frac{b^--2}{b^-}$, $C'=a(b^-)^2$. This implies that the rational curve $E_-$ is covered by the open sets $V_4\cong\C^3_{B'^2,A'/B',C'}$, $V_5\cong\C^3_{A'^2,B'/A',C'}$, and $E^-$ is given by the ratio $(A':B')=(z(x^3+y^3):x^3-y^3)$. In terms of $G$-constellations, the calculation is similar but now taking $\Gamma_-$ instead of $\Gamma_+$. See Figure \ref{OpFixLines}.

The skeletons provide the choices of nonzero variables in the representation space $\rep(Q)$. For example, the case $V_3$ can we choose $c=(1,0)$, $b=(0,1)$, $d=\Id$, $C_1=1$, $B=\left(\begin{smallmatrix}1\\1\end{smallmatrix}\right)$ and use the relations we obtain the representation space of $V_3\subset\mathcal{M}_\theta$ shown in Figure \ref{OpenV3}, where $K=c'(C'-1)$. Thus $V_3\cong\C^3_{a,c'C'}$ where 
\[
a=\frac{z(x^3+y^3)}{x^3-y^3}, c'=\frac{(x^3+y^3)^2}{x^2y^2} \text{ and } C'=\frac{(x^3-y^3)^2}{(x^3+y^3)^2}.
\]

\begin{figure}[h]
\begin{center}
\begin{pspicture}(0,-0.75)(7,1.55)
	\psset{arrowsize=4pt,arrowlength=2,arcangle=15,nodesep=1pt}
%%%%%%%%%%%%%%%%%%%%%%%% V_1
\rput(0.5,0.25){
\scalebox{0.9}{
% Vertices of the quiver 
	\rput(-0,1.2){\rnode{0}{$\rho_0^0$}}
	\rput(-0,-1.2){\rnode{0'}{$\rho_0^1$}}
	\rput(2,0){\rnode{1}{$\rho_1^0$}}
	\rput(4.5,0){\rnode{2}{$\rho_2^0$}}
	\rput(6.5,1.2){\rnode{3}{$\rho_3^0$}}
	\rput(6.5,-1.2){\rnode{3'}{$\rho_3^1$}}	
% points for loops
	\rput(2,0.3){\rnode{u1}{}}	
	\rput(4.5,0.3){\rnode{u2}{}}	
% arrows
	\ncarc{->}{0}{0'}\mput*[0.1]{$z$}	% a
	\ncarc{->}{0'}{0}\Aput*[0.1]{$aC'$}	% A
	\ncarc{->}{0}{1}\Aput*[0.1]{\footnotesize $(1,0)$}\ncarc{->}{0}{1}		% c
	\ncarc{->}{1}{0}\aput*[-0.25](0.45){$\left(\begin{smallmatrix}K\\0\end{smallmatrix}\right)$}		% C
	\ncarc{->}{0'}{1}\aput*[-0.3](0.5){\footnotesize $(0,1)$}	% b
	\ncarc{->}{1}{0'}\aput*[0](0.4){$\left(\begin{smallmatrix}0\\K\end{smallmatrix}\right)$}\ncarc{->}{1}{0'}	% B
	\ncarc{->}{1}{2}	\Aput*[0.1]{$\left(\begin{smallmatrix}1&0\\0&1\end{smallmatrix}\right)$}	% d
	\ncarc{->}{2}{1}	\Aput*[0.1]{$\left(\begin{smallmatrix}K&0\\0&K\end{smallmatrix}\right)$}	% D
	\ncarc{->}{2}{3}	\aput*[0](0.6){$\left(\begin{smallmatrix}1\\C'\end{smallmatrix}\right)$}\ncarc{->}{2}{3}	% C'
	\ncarc{->}{3}{2}	\Aput*[-0.25]{\footnotesize $(\text{-}c',c')$}	% c'
	\ncarc{->}{2}{3'}\mput*[0.05]{$\left(\begin{smallmatrix}1\\1\end{smallmatrix}\right)$}	% B'
	\ncarc{->}{3'}{2}\Aput*[0]{\footnotesize $(c'C',\text{-}c')$}\ncarc{->}{3'}{2}	% b'
	\ncarc[arcangle=-15]{->}{3}{3'}\mput*[0.1]{$a$}	% a'
	\ncarc[arcangle=-15]{->}{3'}{3}\Bput*[0.1]{$aC'$}	% A'
% loops
	 \nccircle[angleA=-25,nodesep=3pt]{->}{u1}{.3cm}\Bput[-0.05]{\footnotesize $\left(\begin{smallmatrix}0&a\\aC'&0\end{smallmatrix}\right)$}
	 \nccircle[angleA=25,nodesep=3pt]{->}{u2}{.3cm}\Bput[-0.05]{\footnotesize $\left(\begin{smallmatrix}0&a\\aC'&0\end{smallmatrix}\right)$}	
	 }}	
\end{pspicture}
\caption{Open set $V_3\subset\mathcal{M}_\theta$.}
\label{OpenV3}
\end{center}
\end{figure}
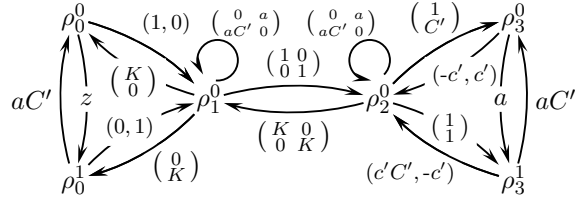

\end{exa}

\subsection{$\Hilb{D_8}{\Hilb{\frac{1}{3}(1,2,0)}{\C^3}}$.}

Let $G$ be the group $N\rtimes G/N\subset\SL(3,\C)$ where $N$ is generated by $g=\frac{1}{3}(1,2,0)$ and $G/N\cong D_8\subset\SO(3)$ is generated by $\alpha=\frac{1}{4}(1,3,0)$ and $\beta=\left(\begin{smallmatrix}0&1&0\\1&0&0\\0&0&-1\end{smallmatrix}\right)$. Let $\Irr(N)=\{\sigma_0,\sigma_1,\sigma_2\}$ where $\sigma_i(g)=\omega^i$ for $i=0,1,2$ and $\omega$ a primitive cube root of unity. In this case we have $\Orb(\sigma_0)=\sigma_0$ and $\Orb(\sigma_1)=\{\sigma_1,\sigma_2\}$ with stabilizers $G_0=G/N\cong D_8$ and $G_1=\Span{\alpha}\cong\Z/4\Z$. Therefore $\Irr G_0=\{\tau_0,\tau_1,\tau_2,\tau_3,\tau_4\}$ where $\dim(\tau_i)=1$ for $i=0,\ldots,3$ and $\dim(\tau_4)=2$. The irreducible representations of $G$ are shown in Table \ref{IrrD24}.

{\renewcommand{\arraystretch}{0.3}
\begin{table}[h]
\begin{center}
\begin{tabular}{|c|c|}
\multicolumn{1}{c}{$\sigma_0$} &  \multicolumn{1}{c}{$\{\sigma_1,\sigma_2\}$} \\
\multicolumn{1}{c}{} &  \multicolumn{1}{c}{} \\
\cline{1-2} 
\multirow{4}{*}{$\rho_0^0$} & \multirow{5}{*}{$\rho_1^0$}   \\
& \\
& \\
& \\
\cline{1-1}
\multirow{4}{*}{$\rho_0^1$} & \multirow{7}{*}{$\rho_1^1$}   \\
\cline{2-2}
& \\
& \\
& \\
\cline{1-1}
\multirow{4}{*}{$\rho_0^2$} & \multirow{9}{*}{$\rho_1^2$}   \\
& \\
\cline{2-2}
& \\
& \\
\cline{1-1}
\multirow{4}{*}{$\rho_0^3$} & \multirow{11}{*}{$\rho_1^3$}   \\
& \\
& \\
\cline{2-2}
& \\
\cline{1-1}
\multirow{4}{*}{$\rho_0^4$} &  \\
& \\
& \\
& \\
\cline{1-2}
\end{tabular}
, {\bf{d}} = 
\begin{tabular}{|c|c|}
\multicolumn{1}{c}{$\sigma_0$} &  \multicolumn{1}{c}{$\{\sigma_1,\sigma_2\}$} \\
\multicolumn{1}{c}{} &  \multicolumn{1}{c}{} \\
\cline{1-2} 
\multirow{4}{*}{$1$} & \multirow{5}{*}{$2$}   \\
& \\
& \\
& \\
\cline{1-1}
\multirow{4}{*}{$1$} & \multirow{7}{*}{$2$}   \\
\cline{2-2}
& \\
& \\
& \\
\cline{1-1}
\multirow{4}{*}{$1$} & \multirow{9}{*}{$2$}   \\
& \\
\cline{2-2}
& \\
& \\
\cline{1-1}
\multirow{4}{*}{$1$} & \multirow{11}{*}{$2$}   \\
& \\
& \\
\cline{2-2}
& \\
\cline{1-1}
\multirow{4}{*}{$2$} &  \\
& \\
& \\
& \\
\cline{1-2}
\end{tabular}
, $\theta$ = 
\begin{tabular}{|c|c|}
\multicolumn{1}{c}{$\sigma_0$} &  \multicolumn{1}{c}{$\{\sigma_1,\sigma_2\}$} \\
\multicolumn{1}{c}{} &  \multicolumn{1}{c}{} \\
\cline{1-2} 
\multirow{4}{*}{$-a-b-5\varepsilon c$} & \multirow{5}{*}{$a+b$}   \\
& \\
& \\
& \\
\cline{1-1}
\multirow{4}{*}{$-a-b+\varepsilon c$} & \multirow{7}{*}{$a+b$}   \\
\cline{2-2}
& \\
& \\
& \\
\cline{1-1}
\multirow{4}{*}{$-a-b+\varepsilon c$} & \multirow{9}{*}{$a+b$}   \\
& \\
\cline{2-2}
& \\
& \\
\cline{1-1}
\multirow{4}{*}{$-a-b+\varepsilon c$} & \multirow{11}{*}{$a+b$}   \\
& \\
& \\
\cline{2-2}
& \\
\cline{1-1}
\multirow{4}{*}{$-2a-2b+\varepsilon c$} &  \\
& \\
& \\
& \\
\cline{1-2}
\end{tabular}
\end{center}
\caption{Irreducible representations of $G$ with their dimensions and the stability condition for $\Hilb{G/N}{\Hilb{N}{\C^3}}$ with $a,b,c>0$ and $0<\varepsilon<<1$.}
\label{IrrD24}
\end{table}}

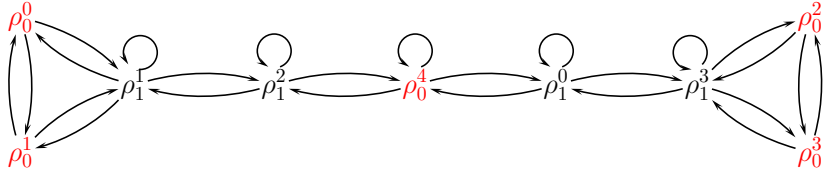
\begin{figure}[h]
\begin{center}
\begin{pspicture}(0,-0.75)(11,1)
	\psset{arrowlength=2,arcangle=15,nodesep=1pt}
%%%%%%%%%%%%%%%%%%%%%%%% McKay quiver for D_12
\rput(0,0){
\scalebox{0.75}{
% Vertices of the quiver 
	\rput(-0,1.2){\rnode{0}{\Large ${\red{\rho_0^0}}$}}
	\rput(-0,-1.2){\rnode{0'}{\Large ${\red{\rho_0^1}}$}}
	\rput(2,0){\rnode{1}{\Large $\rho_1^1$}}
	\rput(4.5,0){\rnode{2}{\Large $\rho_1^2$}}
	\rput(7,0){\rnode{3}{\Large ${\red{\rho_0^4}}$}}
	\rput(9.5,0){\rnode{4}{\Large $\rho_1^0$}}
	\rput(12,0){\rnode{5}{\Large $\rho_1^3$}}
	\rput(14,1.2){\rnode{6}{\Large ${\red{\rho_0^2}}$}}
	\rput(14,-1.2){\rnode{6'}{\Large ${\red{\rho_0^3}}$}}	
% points for loops
	\rput(2,0.3){\rnode{u1}{}}	
	\rput(4.5,0.3){\rnode{u2}{}}	
	\rput(7,0.3){\rnode{u3}{}}	
	\rput(9.5,0.3){\rnode{u4}{}}	
	\rput(12,0.3){\rnode{u5}{}}	
% arrows
	\ncarc{->}{0}{0'}%\mput*[0.1]{$a$}	% a
	\ncarc{->}{0'}{0}%\Aput*[0.1]{$A$}	% A
	\ncarc{->}{0}{1}%\Aput*[0.1]{$c$}	% c
	\ncarc{->}{1}{0}%\mput*[0.1]{$C$}	% C
	\ncarc{->}{0'}{1}%\mput*[0.1]{$b$}	% b
	\ncarc{->}{1}{0'}%\Aput*[0.1]{$B$}	% B
	\ncarc{->}{1}{2}	%\Aput*[0.1]{$d_1$}	% d1
	\ncarc{->}{2}{1}	%\Aput*[0.1]{$D_1$}	% D1
	\ncarc{->}{2}{3}	%\Aput*[0.1]{$d_2$}	% d2
	\ncarc{->}{3}{2}	%\Aput*[0.1]{$D_2$}	% D2
	\ncarc{->}{3}{4}	%\Aput*[0.1]{$d_3$}	% d3
	\ncarc{->}{4}{3}	%\Aput*[0.1]{$D_3$}	% D3
	\ncarc{->}{4}{5}	%\Aput*[0.1]{$d_4$}	% d4
	\ncarc{->}{5}{4}	%\Aput*[0.1]{$D_4$}	% D4
	\ncarc{->}{5}{6}	%\Aput*[0.1]{$C'$}	% C'
	\ncarc{->}{6}{5}	%\mput*[0.1]{$c'$}	% c'
	\ncarc{->}{5}{6'}%\mput*[0.1]{$B'$}	% B'
	\ncarc{->}{6'}{5}%\Aput*[0.1]{$b'$}	% b'
	\ncarc[arcangle=-15]{->}{6}{6'}%\mput*[0.1]{$a'$}	% a'
	\ncarc[arcangle=-15]{->}{6'}{6}%\Bput*[0.1]{$A'$}	% A'
	 \nccircle[angleA=-25,nodesep=3pt]{->}{u1}{.3cm}%\Bput[0.05]{\footnotesize $u_1$}
	 \nccircle[nodesep=3pt]{->}{u2}{.3cm}%\Bput[0.05]{\footnotesize $u_2$}
	 \nccircle[nodesep=3pt]{->}{u3}{.3cm}%\Bput[0.05]{\footnotesize $u_3$}
	 \nccircle[nodesep=3pt]{->}{u4}{.3cm}%\Bput[0.05]{\footnotesize $u_4$}
	 \nccircle[angleA=25,nodesep=3pt]{->}{u5}{.3cm}%\Bput[0.05]{\footnotesize $u_5$}
	}}	
\end{pspicture}
\end{center}
\caption{The McKay quiver of $G$.}
\label{McKayD24}
\end{figure}

The McKay quiver is shown in Figure \ref{McKayD24} which coincides with the Mckay quiver of $D_{24}\subset\SO(3)$ as in \ref{exaDn} since $G\cong D_{24}=\Span{\frac{1}{12}(1,11,0),\beta}$. 

The scheme $\Hilb{N}{\C^3}$ can be covered by 3 open sets $U_i$ for $i=1,2,3$ with the corresponding distinguished $N$-constellations $Z_1=\{1,x,x^2\}$, $Z_2=\{1,x,y\}$ and $Z_3=\{1,y,y^2\}$. By the action of $G/N\cong D_8$ the open set $U_2$ is fixed while $U_1$ and $U_3$ are identified. 

The fixed open set $U_2$ has coordinates $a=\frac{x^2}{y}$, $b=\frac{y^2}{x}$ and $c=z$. If we denote by $(+):=a^2+b^2$ and $(-):=a^2-b^2$, we have that $\Hilb{G/N}{\C^3_{a,b,c}}$ is covered by 5 open sets given by the following $G/N$ constellations:
\begin{align*}
\Gamma_1 &= \{ 1,a,b,(+),(-),a(+),b(+),-(+)(-) \} \\
\Gamma_2 &= \{ 1,a,b,(+),(-),a(+),b(+),c \} \\
\Gamma_3 &= \{ 1,a,b,(+),(-),ac,-bc,c \} \\
\Gamma_4 &= \{ 1,a,b,c(-),(-),ac,-bc,c \} \\
\Gamma_5 &= \{ 1,a,b,(+),c(+),ac,-bc,c \} 
\end{align*}

We obtain in this way the open sets $V_i$ given by the $G$-constellations $\Gamma_i\cdot Z_2$, for $i=1,\ldots,5$ shown in Figure \ref{OpFixLinesD8}. 

In the case of the orbit $\{U_1,U_3\}$ we have $U_1\cong\C^3_{d,e,f}$ with coordinates $d=x^3$, $e=y/x^2$, $f=z$, and $U_3\cong\C^3_{d',e',f'}$ with coordinates $d'=y^3$, $e'=x/y^2$, $f'=z$. In each of the open sets there exist a fixed line with stabilizer subgroup $G_1=\Span{\alpha}\cong\Z/4\Z$. This implies that we have to consider the $G_1$-graphs
\[
\Omega_1=\{1,c,c^2,c^3\}, \Omega_2=\{1,c,c^2,d\}, \Omega_3=\{ 1,c,d,d^2\}, \Omega_4=\{1,d,d^2,d^3\} 
\]
in $U_1$, and 
\[
\Omega'_1=\{1,c',c'^2,c'^3\}, \Omega'_2=\{1,c',c'^2,d'\}, \Omega_3=\{ 1,c',d',d'^2\}, \Omega_4=\{1,d',d'^2,d'^3\}
\]
in $U_3$. The identification of $U_1$ and $U_3$ by $\beta$ produces the open sets $U_i$ given by the $G$-constellations $\Omega_i\cdot Z_1\cup\Omega'_i\cdot Z_3$ for $i=1,\ldots,4$. See Figure \ref{OpFixLinesD8}.

\begin{figure}[h]
\begin{center}
\begin{pspicture}(0,-0.25)(20,8)
	\psset{arrowsize=4pt,arrowlength=2,arcangle=15,nodesep=1pt}

\scalebox{0.95}{

\rput(0.75,8){$G$-const.}
\rput(3.75,8){Skeleton}
\rput(6.75,8){Open set}
\rput(6.75,7){$V_1$}
\rput(6.76,5.25){$V_2$}
\rput(6.75,3.5){$V_3$}
\rput(6.75,1.75){$V_4$}
\rput(6.75,0){$V_5$}

\rput(9.5,8){$G$-const.}
\rput(13,8){Skeleton}
\rput(16,8){Open set}
\rput(16,7){$U_1$}
\rput(16,5.25){$U_2$}
\rput(16,3.5){$U_3$}
\rput(16,1.75){$U_4$}

\psline(8,-0.5)(8,8.25)

%%%%%%%%%%%%%%%%%%%%%%%%	V_1
\rput(0.75,7){$
\begin{array}{r}
\Gamma_1\cdot Z_2 \\
\end{array}
$}

\rput(1.75,7){
\scalebox{0.5}{
% Vertices of the quiver 
	\rput(-0,1.2){\rnode{0}{\large $1$}}
	\rput(-0,-1.2){\rnode{0'}{\large $1$}}
	\rput(1,0.35){\rnode{1a}{\large $e_1$}}
	\rput(1,-0.35){\rnode{1b}{\large $e_2$}}
	\rput(2.5,0.35){\rnode{2a}{\large $e_1$}}
	\rput(2.5,-0.35){\rnode{2b}{\large $e_2$}}
	\rput(4,0.35){\rnode{3a}{\large $e_1$}}
	\rput(4,-0.35){\rnode{3b}{\large $e_2$}}
	\rput(5.5,0.35){\rnode{4a}{\large $e_1$}}
	\rput(5.5,-0.35){\rnode{4b}{\large $e_2$}}
	\rput(7,0.35){\rnode{5a}{\large $e_1$}}
	\rput(7,-0.35){\rnode{5b}{\large $e_2$}}
	\rput(8,1.2){\rnode{6}{\large $1$}}
	\rput(8,-1.2){\rnode{6'}{\large $1$}}			
% arrows
	\ncarc[linecolor=red]{->}{0}{1a}	
	\ncarc[linecolor=red]{->}{0'}{1b}	
	\ncarc[linecolor=blue]{->}{1a}{2a}
	\ncarc[linecolor=blue]{->}{2b}{1b}	
	\ncarc[linecolor=red]{->}{3a}{2a}
	\ncarc[linecolor=red]{->}{3b}{2b}	
	\ncarc[linecolor=red]{->}{3a}{4a}
	\ncarc[linecolor=red]{->}{3b}{4b}	
	\ncarc[linecolor=blue]{->}{4a}{5a}
	\ncarc[linecolor=blue]{->}{4a}{5b}
	\ncarc[linecolor=blue]{->}{5a}{4b}
	\ncarc[linecolor=blue]{->}{5b}{4b}
	\ncarc[linecolor=red]{->}{6}{5a}	
	\ncarc[linecolor=red]{->}{6'}{5b}
	}}

%%%%%%%%%%%%%%%%%%%%%%%%	U_1
\rput(9.5,7){$
\begin{array}{r}
\Omega_1\!\cdot\!Z_1\cup \Omega'_1\!\cdot\!Z_3 \\
\end{array}
$}

\rput(11,7){
\scalebox{0.5}{
% Vertices of the quiver 
	\rput(-0,1.2){\rnode{0}{\large $1$}}
	\rput(-0,-1.2){\rnode{0'}{\large $1$}}
	\rput(1,0.35){\rnode{1a}{\large $e_1$}}
	\rput(1,-0.35){\rnode{1b}{\large $e_2$}}
	\rput(2.5,0.35){\rnode{2a}{\large $e_1$}}
	\rput(2.5,-0.35){\rnode{2b}{\large $e_2$}}
	\rput(4,0.35){\rnode{3a}{\large $e_1$}}
	\rput(4,-0.35){\rnode{3b}{\large $e_2$}}
	\rput(5.5,0.35){\rnode{4a}{\large $e_1$}}
	\rput(5.5,-0.35){\rnode{4b}{\large $e_2$}}
	\rput(7,0.35){\rnode{5a}{\large $e_1$}}
	\rput(7,-0.35){\rnode{5b}{\large $e_2$}}
	\rput(8,1.2){\rnode{6}{\large $1$}}
	\rput(8,-1.2){\rnode{6'}{\large $1$}}			
% arrows
	\ncarc[linecolor=red]{->}{0}{1a}	
	\ncarc[linecolor=red]{->}{0'}{1a}	
	\ncarc[linecolor=red]{->}{1a}{2a}
	\ncarc[linecolor=blue]{->}{2a}{3a}
	\ncarc[linecolor=red]{->}{3a}{4a}
	\ncarc[linecolor=red]{->}{4a}{5a}
	\ncarc[linecolor=blue]{->}{5a}{6}
	\ncarc[linecolor=blue]{->}{5a}{6'}
	
	\ncarc[linecolor=red]{->}{6}{5b}
	\ncarc[linecolor=red]{->}{6'}{5b}
	\ncarc[linecolor=red]{->}{5b}{4b}
	\ncarc[linecolor=blue]{->}{4b}{3b}
	\ncarc[linecolor=red]{->}{3b}{2b}
	\ncarc[linecolor=red]{->}{2b}{1b}
	}}

%%%%%%%%%%%%%%%%%%%%%%%%	V_2
\rput(0.75,5.25){$
\begin{array}{r}
\Gamma_2\cdot Z_2 \\
\end{array}
$}
\rput(1.75,5.25){
\scalebox{0.5}{
% Vertices of the quiver 
	\rput(-0,1.2){\rnode{0}{\large $1$}}
	\rput(-0,-1.2){\rnode{0'}{\large $1$}}
	\rput(1,0.35){\rnode{1a}{\large $e_1$}}
	\rput(1,-0.35){\rnode{1b}{\large $e_2$}}
	\rput(2.5,0.35){\rnode{2a}{\large $e_1$}}
	\rput(2.5,-0.35){\rnode{2b}{\large $e_2$}}
	\rput(4,0.35){\rnode{3a}{\large $e_1$}}
	\rput(4,-0.35){\rnode{3b}{\large $e_2$}}
	\rput(5.5,0.35){\rnode{4a}{\large $e_1$}}
	\rput(5.5,-0.35){\rnode{4b}{\large $e_2$}}
	\rput(7,0.35){\rnode{5a}{\large $e_1$}}
	\rput(7,-0.35){\rnode{5b}{\large $e_2$}}
	\rput(8,1.2){\rnode{6}{\large $1$}}
	\rput(8,-1.2){\rnode{6'}{\large $1$}}		
% arrows
	\ncarc[linecolor=red]{->}{0}{1a}	
	\ncarc[linecolor=red]{->}{0'}{1b}	
	\ncarc[linecolor=blue]{->}{0}{0'}	
	\ncarc[linecolor=blue]{->}{1a}{2a}
	\ncline[linecolor=blue,arrowsize=4pt,arrowlength=1,nodesep=0pt]{->}{1a}{1b}	
	\ncarc[linecolor=red]{->}{3a}{2a}
	\ncarc[linecolor=red]{->}{3b}{2b}	
	\ncarc[linecolor=red]{->}{3a}{4a}
	\ncarc[linecolor=red]{->}{3b}{4b}	
	\ncarc[linecolor=blue]{->}{4a}{5a}
	\ncarc[linecolor=blue]{->}{4a}{5b}
	\ncarc[linecolor=blue]{->}{5a}{4b}
	\ncarc[linecolor=blue]{->}{5b}{4b}
	\ncarc[linecolor=red]{->}{6}{5a}	
	\ncarc[linecolor=red]{->}{6'}{5b}
	}}

%%%%%%%%%%%%%%%%%%%%%%%%	U_2
\rput(9.5,5.25){$
\begin{array}{r}
\Omega_2\!\cdot\!Z_1\cup \Omega'_2\!\cdot\!Z_3 \\
\end{array}
$}

\rput(11,5.25){
\scalebox{0.5}{
% Vertices of the quiver 
	\rput(-0,1.2){\rnode{0}{\large $1$}}
	\rput(-0,-1.2){\rnode{0'}{\large $1$}}
	\rput(1,0.35){\rnode{1a}{\large $e_1$}}
	\rput(1,-0.35){\rnode{1b}{\large $e_2$}}
	\rput(2.5,0.35){\rnode{2a}{\large $e_1$}}
	\rput(2.5,-0.35){\rnode{2b}{\large $e_2$}}
	\rput(4,0.35){\rnode{3a}{\large $e_1$}}
	\rput(4,-0.35){\rnode{3b}{\large $e_2$}}
	\rput(5.5,0.35){\rnode{4a}{\large $e_1$}}
	\rput(5.5,-0.35){\rnode{4b}{\large $e_2$}}
	\rput(7,0.35){\rnode{5a}{\large $e_1$}}
	\rput(7,-0.35){\rnode{5b}{\large $e_2$}}
	\rput(8,1.2){\rnode{6}{\large $1$}}
	\rput(8,-1.2){\rnode{6'}{\large $1$}}			
% arrows
	\ncarc[linecolor=red]{->}{0}{1a}	
	\ncarc[linecolor=red]{->}{0'}{1a}	
	\ncarc[linecolor=blue]{->}{0'}{1b}
	\ncarc[linecolor=red]{->}{1a}{2a}
	\ncarc[linecolor=blue]{->}{2a}{3a}
	\ncarc[linecolor=red]{->}{3a}{4a}
	\ncarc[linecolor=red]{->}{4a}{5a}
	\ncarc[linecolor=blue]{->}{5a}{6}
	\ncarc[linecolor=blue]{->}{5a}{6'}
	
	\ncarc[linecolor=red]{->}{6}{5b}
	\ncarc[linecolor=red]{->}{6'}{5b}
	\ncarc[linecolor=red]{->}{5b}{4b}
	\ncarc[linecolor=red]{->}{3b}{2b}
	\ncarc[linecolor=red]{->}{2b}{1b}
	}}

%%%%%%%%%%%%%%%%%%%%%%%% V_3
\rput(0.75,3.5){$
\begin{array}{r}
\Gamma_3\cdot Z_2 \\
\end{array}
$}

\rput(1.75,3.5){
\scalebox{0.5}{
% Vertices of the quiver 
	\rput(-0,1.2){\rnode{0}{\large $1$}}
	\rput(-0,-1.2){\rnode{0'}{\large $1$}}
	\rput(1,0.35){\rnode{1a}{\large $e_1$}}
	\rput(1,-0.35){\rnode{1b}{\large $e_2$}}
	\rput(2.5,0.35){\rnode{2a}{\large $e_1$}}
	\rput(2.5,-0.35){\rnode{2b}{\large $e_2$}}
	\rput(4,0.35){\rnode{3a}{\large $e_1$}}
	\rput(4,-0.35){\rnode{3b}{\large $e_2$}}
	\rput(5.5,0.35){\rnode{4a}{\large $e_1$}}
	\rput(5.5,-0.35){\rnode{4b}{\large $e_2$}}
	\rput(7,0.35){\rnode{5a}{\large $e_1$}}
	\rput(7,-0.35){\rnode{5b}{\large $e_2$}}
	\rput(8,1.2){\rnode{6}{\large $1$}}
	\rput(8,-1.2){\rnode{6'}{\large $1$}}		
% arrows
	\ncarc[linecolor=red]{->}{0}{1a}	
	\ncarc[linecolor=red]{->}{0'}{1b}	
	\ncarc[linecolor=blue]{->}{0}{0'}	
	\ncarc[linecolor=blue]{->}{1a}{2a}
	\ncarc[linecolor=blue]{->}{1b}{2b}
	\ncline[linecolor=blue,arrowsize=4pt,arrowlength=1,nodesep=0pt]{->}{1a}{1b}	
	\ncline[linecolor=blue,arrowsize=4pt,arrowlength=1,nodesep=0pt]{->}{2a}{2b}	
	\ncline[linecolor=blue,arrowsize=4pt,arrowlength=1,nodesep=0pt]{->}{3a}{3b}
	\ncline[linecolor=blue,arrowsize=4pt,arrowlength=1,nodesep=0pt]{->}{4a}{4b}		
	\ncarc[linecolor=red]{->}{3a}{2a}
	\ncarc[linecolor=red]{->}{3b}{2b}	
	\ncarc[linecolor=red]{->}{3a}{4a}
	\ncarc[linecolor=red]{->}{3b}{4b}	
	\ncarc[linecolor=blue]{->}{4a}{5a}
	\ncarc[linecolor=blue]{->}{4a}{5b}
	\ncarc[linecolor=red]{->}{6}{5a}	
	\ncarc[linecolor=red]{->}{6'}{5b}
	}}

%%%%%%%%%%%%%%%%%%%%%%%%	U_3
\rput(9.5,3.5){$
\begin{array}{r}
\Omega_3\!\cdot\!Z_1\cup \Omega'_3\!\cdot\!Z_3 \\
\end{array}
$}

\rput(11,3.5){
\scalebox{0.5}{
% Vertices of the quiver 
	\rput(-0,1.2){\rnode{0}{\large $1$}}
	\rput(-0,-1.2){\rnode{0'}{\large $1$}}
	\rput(1,0.35){\rnode{1a}{\large $e_1$}}
	\rput(1,-0.35){\rnode{1b}{\large $e_2$}}
	\rput(2.5,0.35){\rnode{2a}{\large $e_1$}}
	\rput(2.5,-0.35){\rnode{2b}{\large $e_2$}}
	\rput(4,0.35){\rnode{3a}{\large $e_1$}}
	\rput(4,-0.35){\rnode{3b}{\large $e_2$}}
	\rput(5.5,0.35){\rnode{4a}{\large $e_1$}}
	\rput(5.5,-0.35){\rnode{4b}{\large $e_2$}}
	\rput(7,0.35){\rnode{5a}{\large $e_1$}}
	\rput(7,-0.35){\rnode{5b}{\large $e_2$}}
	\rput(8,1.2){\rnode{6}{\large $1$}}
	\rput(8,-1.2){\rnode{6'}{\large $1$}}		
% arrows
	\ncarc[linecolor=red]{->}{0}{1a}	
	\ncarc[linecolor=red]{->}{0'}{1a}	
	\ncarc[linecolor=blue]{->}{0'}{1b}
	\ncarc[linecolor=red]{->}{1a}{2a}
	\ncarc[linecolor=blue]{->}{2a}{3a}
	\ncarc[linecolor=red]{->}{3a}{4a}
	\ncarc[linecolor=blue]{->}{3b}{4b}
	\ncarc[linecolor=red]{->}{4a}{5a}
	
	\ncarc[linecolor=red]{->}{6}{5b}
	\ncarc[linecolor=red]{->}{6'}{5b}
	\ncarc[linecolor=red]{->}{5b}{4b}
	\ncarc[linecolor=red]{->}{3b}{2b}
	\ncarc[linecolor=red]{->}{2b}{1b}
	}}

%%%%%%%%%%%%%%%%%%%%%%%%	V_4
\rput(0.75,1.75){$
\begin{array}{r}
\Gamma_4\cdot Z_2 \\
\end{array}
$}

\rput(1.75,1.75){
\scalebox{0.5}{
% Vertices of the quiver 
	\rput(-0,1.2){\rnode{0}{\large $1$}}
	\rput(-0,-1.2){\rnode{0'}{\large $1$}}
	\rput(1,0.35){\rnode{1a}{\large $e_1$}}
	\rput(1,-0.35){\rnode{1b}{\large $e_2$}}
	\rput(2.5,0.35){\rnode{2a}{\large $e_1$}}
	\rput(2.5,-0.35){\rnode{2b}{\large $e_2$}}
	\rput(4,0.35){\rnode{3a}{\large $e_1$}}
	\rput(4,-0.35){\rnode{3b}{\large $e_2$}}
	\rput(5.5,0.35){\rnode{4a}{\large $e_1$}}
	\rput(5.5,-0.35){\rnode{4b}{\large $e_2$}}
	\rput(7,0.35){\rnode{5a}{\large $e_1$}}
	\rput(7,-0.35){\rnode{5b}{\large $e_2$}}
	\rput(8,1.2){\rnode{6}{\large $1$}}
	\rput(8,-1.2){\rnode{6'}{\large $1$}}		
% arrows
	\ncarc[linecolor=red]{->}{0}{1a}	
	\ncarc[linecolor=red]{->}{0'}{1b}	
	\ncarc[linecolor=blue]{->}{0}{0'}	
	\ncarc[linecolor=blue]{->}{1a}{2a}
	\ncarc[linecolor=blue]{->}{1b}{2b}
	\ncline[linecolor=blue,arrowsize=4pt,arrowlength=1,nodesep=0pt]{->}{1a}{1b}	
	\ncline[linecolor=blue,arrowsize=4pt,arrowlength=1,nodesep=0pt]{->}{2a}{2b}	
	\ncline[linecolor=blue,arrowsize=4pt,arrowlength=1,nodesep=0pt]{->}{3a}{3b}
	\ncline[linecolor=blue,arrowsize=4pt,arrowlength=1,nodesep=0pt]{->}{4a}{4b}		
	\ncarc[linecolor=red]{->}{3a}{2a}
	\ncarc[linecolor=red]{->}{3b}{2b}	
	\ncarc[linecolor=red]{->}{3a}{4a}
	\ncarc[linecolor=red]{->}{3b}{4b}	
	\ncarc[linecolor=blue]{->}{4a}{5b}
	\ncarc[linecolor=blue]{->}{4b}{5a}
	\ncline[linecolor=blue,arrowsize=4pt,arrowlength=1,nodesep=0pt]{->}{5b}{5a}
	\ncarc[linecolor=red]{->}{6}{5a}	
	\ncarc[linecolor=red]{->}{6'}{5b}		
	\ncarc[linecolor=blue,arcangle=-15]{->}{6'}{6}
	}}

%%%%%%%%%%%%%%%%%%%%%%%%	U_4
\rput(9.5,1.75){$
\begin{array}{r}
\Omega_3\!\cdot\!Z_1\cup \Omega'_3\!\cdot\!Z_3 \\
\end{array}
$}

\rput(11,1.75){
\scalebox{0.5}{
% Vertices of the quiver 
	\rput(-0,1.2){\rnode{0}{\large $1$}}
	\rput(-0,-1.2){\rnode{0'}{\large $1$}}
	\rput(1,0.35){\rnode{1a}{\large $e_1$}}
	\rput(1,-0.35){\rnode{1b}{\large $e_2$}}
	\rput(2.5,0.35){\rnode{2a}{\large $e_1$}}
	\rput(2.5,-0.35){\rnode{2b}{\large $e_2$}}
	\rput(4,0.35){\rnode{3a}{\large $e_1$}}
	\rput(4,-0.35){\rnode{3b}{\large $e_2$}}
	\rput(5.5,0.35){\rnode{4a}{\large $e_1$}}
	\rput(5.5,-0.35){\rnode{4b}{\large $e_2$}}
	\rput(7,0.35){\rnode{5a}{\large $e_1$}}
	\rput(7,-0.35){\rnode{5b}{\large $e_2$}}
	\rput(8,1.2){\rnode{6}{\large $1$}}
	\rput(8,-1.2){\rnode{6'}{\large $1$}}		
% arrows
	\ncarc[linecolor=red]{->}{0}{1a}	
	\ncarc[linecolor=red]{->}{0'}{1a}	
	\ncarc[linecolor=blue]{->}{0'}{1b}
	\ncarc[linecolor=red]{->}{1a}{2a}
	\ncarc[linecolor=red]{->}{3a}{4a}
	\ncarc[linecolor=blue]{->}{3b}{4b}
	\ncarc[linecolor=red]{->}{4a}{5a}
	
	\ncarc[linecolor=red,nodesep=0]{->}{6}{5b}
	\ncarc[linecolor=red]{->}{6'}{5b}
	\ncarc[linecolor=blue,arcangle=-15,nodesep=0]{->}{6'}{5a}
	\ncarc[linecolor=red]{->}{5b}{4b}
	\ncarc[linecolor=red]{->}{3b}{2b}
	\ncarc[linecolor=red]{->}{2b}{1b}
	}}

%%%%%%%%%%%%%%%%%%%%%%%% V_5
\rput(0.75,0){$
\begin{array}{r}
\Gamma_5\cdot Z_2 \\
\end{array}
$}

\rput(1.75,0){
\scalebox{0.5}{
% Vertices of the quiver 
	\rput(-0,1.2){\rnode{0}{\large $1$}}
	\rput(-0,-1.2){\rnode{0'}{\large $1$}}
	\rput(1,0.35){\rnode{1a}{\large $e_1$}}
	\rput(1,-0.35){\rnode{1b}{\large $e_2$}}
	\rput(2.5,0.35){\rnode{2a}{\large $e_1$}}
	\rput(2.5,-0.35){\rnode{2b}{\large $e_2$}}
	\rput(4,0.35){\rnode{3a}{\large $e_1$}}
	\rput(4,-0.35){\rnode{3b}{\large $e_2$}}
	\rput(5.5,0.35){\rnode{4a}{\large $e_1$}}
	\rput(5.5,-0.35){\rnode{4b}{\large $e_2$}}
	\rput(7,0.35){\rnode{5a}{\large $e_1$}}
	\rput(7,-0.35){\rnode{5b}{\large $e_2$}}
	\rput(8,1.2){\rnode{6}{\large $1$}}
	\rput(8,-1.2){\rnode{6'}{\large $1$}}		
% arrows
	\ncarc[linecolor=red]{->}{0}{1a}	
	\ncarc[linecolor=red]{->}{0'}{1b}	
	\ncarc[linecolor=blue]{->}{0}{0'}	
	\ncarc[linecolor=blue]{->}{1a}{2a}
	\ncarc[linecolor=blue]{->}{1b}{2b}
	\ncline[linecolor=blue,arrowsize=4pt,arrowlength=1,nodesep=0pt]{->}{1a}{1b}	
	\ncline[linecolor=blue,arrowsize=4pt,arrowlength=1,nodesep=0pt]{->}{2a}{2b}	
	\ncline[linecolor=blue,arrowsize=4pt,arrowlength=1,nodesep=0pt]{->}{3a}{3b}
	\ncline[linecolor=blue,arrowsize=4pt,arrowlength=1,nodesep=0pt]{->}{4a}{4b}		
	\ncarc[linecolor=red]{->}{3a}{2a}
	\ncarc[linecolor=red]{->}{3b}{2b}	
	\ncarc[linecolor=red]{->}{3a}{4a}
	\ncarc[linecolor=red]{->}{3b}{4b}	
	\ncarc[linecolor=blue]{->}{4a}{5a}
	\ncarc[linecolor=blue]{->}{4b}{5b}
	\ncline[linecolor=blue,arrowsize=4pt,arrowlength=1,nodesep=0pt]{->}{5a}{5b}
	\ncarc[linecolor=red]{->}{6}{5a}	
	\ncarc[linecolor=red]{->}{6'}{5b}		
	\ncarc[linecolor=blue,arcangle=-15]{->}{6}{6'}
	}}
	}
\end{pspicture}
\end{center}
\caption{$G$-constellations for $\Hilb{D_8}{\Hilb{\frac{1}{3}(1,2,0)}{\C^3}}$.}
\label{OpFixLinesD8}
\end{figure}

Let $Y:=\Hilb{\frac{1}{3}(1,2,0)}{\C^3}$, Then the exceptional fibre over the origin $\pi^{-1}(0)$ of the crepant resolution $\pi:Y\to\C^3/N$ consist of two $(-2,0)$-rational curves intersecting in one point. The action of $G/N$ interchange these two curves, producing in $Y/(G/N)$ a single rational curve $E$ with singularities of types $\frac{1}{4}(1,3,0)$ and $D_8$ at 0 and $\infty$ respectively. The fibre $\phi^{-1}(0)$ of the crepant resolution $\phi:\Hilb{G/N}{Y}\to\C^3/G$ is therefore given by the following dual graph:

%\begin{figure}[h]
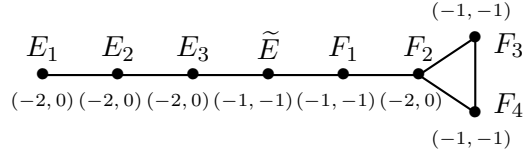
\begin{center}
\begin{pspicture}(0,0)(6,2.25)
	\psset{arcangle=15,nodesep=2pt}
	
	\rput(0,1){$\bullet$}\rput(0,1.35){$E_1$}\rput(0,0.65){\tiny $(-2,0)$}
	\rput(1,1){$\bullet$}\rput(1,1.35){$E_2$}\rput(0.9,0.65){\tiny $(-2,0)$}
	\rput(2,1){$\bullet$}\rput(2,1.35){$E_3$}\rput(1.8,0.65){\tiny $(-2,0)$}	
	\rput(3,1){$\bullet$}\rput(3,1.4){$\widetilde{E}$}\rput(2.8,0.65){\tiny $(-1,-1)$}	
	\rput(4,1){$\bullet$}\rput(4,1.35){$F_1$}\rput(3.9,0.65){\tiny $(-1,-1)$}
	\rput(5,1){$\bullet$}\rput(5,1.35){$F_2$}\rput(4.9,0.65){\tiny $(-2,0)$}
	\rput(5.75,1.5){$\bullet$}\rput(6.2,1.4){$F_3$}\rput(5.7,1.85){\tiny $(-1,-1)$}
	\rput(5.75,0.5){$\bullet$}\rput(6.2,0.6){$F_4$}\rput(5.7,0.15){\tiny $(-1,-1)$}
	\psline(0,1)(5,1)(5.75,1.5)(5.75,0.5)(5,1)

\end{pspicture}
%\caption{McKay quiver of $G$ with relations.}
%\label{McKQT12}
\end{center}
%\end{figure}

where $E_i$ are covered by $U_i$ for $i=1,\ldots,3$, and $F_j$ are covered by $V_j$ for $j=1,\ldots,4$, and $\widetilde{E}$ is the strict transform of $E$.

\subsection{Trihedral group of order 12}\label{exaE6}

Let $G:=G_{12}\subset\SO(3)$ be the trihedral group of order 12 generated by $N:=\Span{\frac{1}{2}(1,1,0), \frac{1}{2}(1,0,1)}$ and $T:=\left(\begin{smallmatrix}0&1&0\\0&0&1\\1&0&0\end{smallmatrix}\right)$. The Abelian normal subgroup $N\cong\Z/2\times\Z/2\Z$ with $\Irr(N)=\{\sigma_0,\sigma_1,\sigma_2,\sigma_3\}$ induce the irreducible representations of $G$ as in Table \ref{IrrG12}. The McKay quiver with relations $(Q,R)$ is given in Figure \ref{McKQT12}.

{\renewcommand{\arraystretch}{1.25}
\begin{table}[h]
\begin{center}
\begin{tabular}{|c|c|}
\multicolumn{1}{r}{$\sigma_0$}
 &  \multicolumn{1}{c}{$\{\sigma_1,\sigma_2,\sigma_3\}$} \\
\cline{1-2}
$\rho_0^0$ & \multirow{3}{*}{$\rho_1$}   \\
\cline{1-1}
$\rho_0^1$&  \\
\cline{1-1}
$\rho_0^2$ &  \\
\cline{1-2}
\end{tabular}
, {\bf{d}}$=$
\begin{tabular}{|c|c|}
\multicolumn{1}{r}{$\sigma_0$}
 &  \multicolumn{1}{c}{$\{\sigma_1,\sigma_2,\sigma_3\}$} \\
\cline{1-2}
1 &  \multirow{3}{*}{3}   \\
\cline{1-1}
1 &  \\
\cline{1-1}
1 &  \\
\cline{1-2}
\end{tabular}
, $\theta=$
\begin{tabular}{|c|c|}
\multicolumn{1}{r}{$\sigma_0$}
 &  \multicolumn{1}{c}{$\{\sigma_1,\sigma_2,\sigma_3\}$} \\
\cline{1-2}
$-\sum a_i-2\varepsilon b$ &  \multirow{3}{*}{$a_1+a_2+a_3$}   \\
\cline{1-1}
$-\sum a_i+\varepsilon b$ &  \\
\cline{1-1}
$-\sum a_i+\varepsilon b$ &  \\
\cline{1-2}
\end{tabular}
\end{center}
\caption{Irreducible representations of $G$ with their dimensions and the stability condition for $\Hilb{G/N}{\Hilb{N}{\C^3}}$ with $a_i,b>0$ for $i=1,2,3$ and $0<\varepsilon<<1$.}
\label{IrrG12}
\end{table}}

By acting first on $\C^3_{x,y,z}$ with $N$ we have that $\Hilb{N}{\C^3}$ is given by 4 affine open sets $U_i$, $i=1,\ldots,4$. The $N$-constellations at each of these open sets are
\[
Z_1=\{ 1,x,y,xy \}, Z_2=\{ 1,y,z,yz \}, Z_3=\{ 1,x,z,xz \}, Z_4=\{ 1,x,y,z \}.
\]

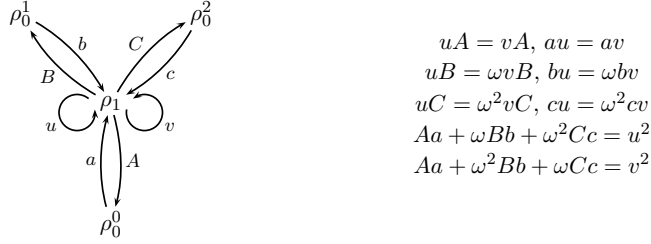
\begin{figure}[h]
\begin{center}
\begin{pspicture}(0,-1.5)(6,1.5)
	\psset{arcangle=15,nodesep=2pt}
\rput(1,0){
\scalebox{0.8}{
% Vertices of the quiver 
	\rput(0,-2){\rnode{0}{$\rho_0^0$}}
	\rput(0,0){\rnode{3}{$\rho_1$}}
	\rput(-1.5,1.5){\rnode{1}{$\rho_0^1$}}
	\rput(1.5,1.5){\rnode{2}{$\rho_0^2$}}
% points for loops
	\rput(-0.3,0){\rnode{u}{}}	
	\rput(0.3,0){\rnode{v}{}}	
% arrows
	\ncarc{->}{0}{3}\Aput[0.05]{\footnotesize $a$}	% a
	\ncarc{->}{3}{0}\Aput[0.05]{\footnotesize $A$}	% A
	\ncarc{->}{1}{3}\Aput[0.05]{\footnotesize $b$}	% b	
	\ncarc{->}{3}{1}\Aput[0.05]{\footnotesize $B$}	% B
	\ncarc{->}{2}{3}\Aput[0.05]{\footnotesize $c$}	% c
	\ncarc{->}{3}{2}\Aput[0.05]{\footnotesize $C$}	% C
	\nccircle[angleA=120,nodesep=3pt]{->}{u}{.3cm}\Bput[0.05]{\footnotesize $u$}
	 \nccircle[angleA=240,nodesep=3pt]{->}{v}{.3cm}\Bput[0.05]{\footnotesize $v$}
\rput(7,1){$uA=vA$, $au=av$}
\rput(7,0.5){$uB=\omega vB$,  $bu=\omega bv$}
\rput(7,0){$uC=\omega^2 vC$, $cu=\omega^2 cv$}
\rput(7,-0.5){$Aa+\omega Bb+\omega^2 Cc=u^2$}
\rput(7,-1){$Aa+\omega^2 Bb+\omega Cc=v^2$}
	}}
\end{pspicture}
\caption{McKay quiver of $G$ with relations.}
\label{McKQT12}
\end{center}
\end{figure}

The action of $G/N=\Span{\bar{T}}\cong\Z/3\Z$ identifies the open sets $U_1$, $U_2$ and $U_3$, and fixes $U_4$, inducing the corresponding action on the $N$-constellations. The orbit $\{U_1,U_2,U_3\}$ give rise to the open set $V_1\subset\Hilb{G/N}{\Hilb{N}{\C^3}}$ shown in Figure \ref{OpOrbitUi}.

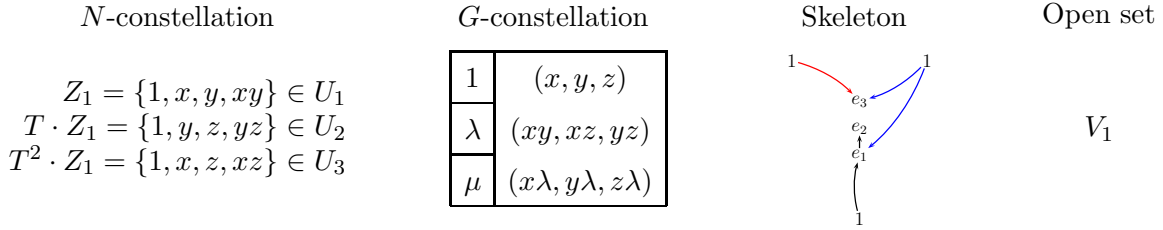
\begin{figure}[h]
\begin{center}
\begin{pspicture}(0,-1)(15,1.5)
	\psset{arcangle=15,nodesep=1pt}

\rput(2,1.5){$N$-constellation}
\rput(7,1.5){$G$-constellation}
\rput(11,1.5){Skeleton}
\rput(14.25,1.5){Open set}
\rput(14.25,0){$V_1$}

%%%%%%%%%%%%%%%%%%%%%%%% V_1
% N-constellations
\rput(2,0){$
\begin{array}{r}
Z_1 = \{1,x,y,xy\}\in U_1 \\
T\cdot Z_1 = \{ 1,y,z,yz\} \in U_2 \\
T^2\cdot Z_1 = \{ 1,x,z,xz\} \in U_3 \\
\end{array}
$}
% Splitting into IrrG
\rput(7,0){
{\renewcommand{\arraystretch}{1.5}
\begin{tabular}{|c|c|}
\cline{1-2}
$1$ & $(x,y,z)$  \\
\cline{1-1}
$\lambda$ & $(xy,xz,yz)$ \\
\cline{1-1}
$\mu$ & $(x\lambda,y\lambda,z\lambda)$ \\
\cline{1-2}
\end{tabular}}
}
\rput(11,0){
\scalebox{0.6}{
% Vertices of the quiver 
	\rput(0,-2){\rnode{0}{\large $1$}}
	\rput(0,0.6){\rnode{33}{\large $e_3$}}	
	\rput(0,0.0){\rnode{32}{\large $e_2$}} 	
	\rput(0,-0.6){\rnode{31}{\large $e_1$}}
	\rput(-1.5,1.5){\rnode{1}{\large $1$}}
	\rput(1.5,1.5){\rnode{2}{\large $1$}}
% arrows
	\ncarc{->}{0}{31}
	\ncarc[linecolor=red]{->}{1}{33}
	\ncarc[linecolor=blue]{->}{2}{33}
	\ncarc[linecolor=blue]{->}{2}{31}
	\ncline{->}{31}{32}
	}}

\end{pspicture}
\end{center}
\caption{$G$-constellation arising from the orbit $\{U_1,U_2,U_3\}$.}
\label{OpOrbitUi}
\end{figure}

It follows from the same method as Section \ref{LocalCoord} that $\lambda=\frac{R_2}{R_0}$ and $\mu=\frac{R_1}{R_0}$, where $R_0:=y^2z^2+x^2z^2+x^2y^2$, $R_1:=y^2z^2+\omega x^2z^2+\omega^2x^2y^2$ and $R_2:=y^2z^2+\omega^2x^2z^2+\omega x^2y^2$. The local coordinates of this open set are written at the end of the section.

The remaining case is the fixed $N$-constellation $Z_4$ in $U_4$. The open set $U_4\cong\C^3_{a,b,c}$ has coordinates $a=\frac{yz}{x}$, $b=\frac{xz}{y}$ and $c=\frac{xy}{z}$ and we have $T(a)=b$, $T(b)=c$ and $T(c)=a$. On the other hand, diagonalizing the action of $T\cong\Z/3\Z$, we can consider it to be of type $\frac{1}{3}(1,2,0)$ on $\C^3_{\alpha,\beta,\gamma}$ with the new coordinates $\alpha=a+\omega^2b+\omega c$, $\beta=a+\omega b+\omega^2c$ and $\gamma=a+b+c$. That is, 
\[
\alpha=\frac{f_1}{f_3}, \beta=\frac{f_2}{f_3}, \gamma=\frac{f_0}{f_3}
\]
where $f_0:=x^2+y^2+z^2$, $f_1:=x^2+\omega^2 y^2+\omega z^2$, $f_2:=x^2+\omega y^2+\omega^2 z^2$, $f_3:=xyz$ and $\omega$ is a primitive cube root of unity. 

The situation is therefore identical to the Abelian case, so we need to consider the distinguished $G/N$-constellations in $\Hilb{\frac{1}{3}(1,2,0)}{\C_{\alpha,\beta,\gamma}}$, namely $\Gamma_1=\{ 1,\alpha,\alpha^2 \}$, $\Gamma_2=\{ 1,\alpha,\beta \}$, $\Gamma_3=\{ 1,\beta,\beta^2 \}$. They give rise to the $G$-constellations $Z_4\cdot\Gamma_i$ shown in Figure \ref{OpFixU4}.

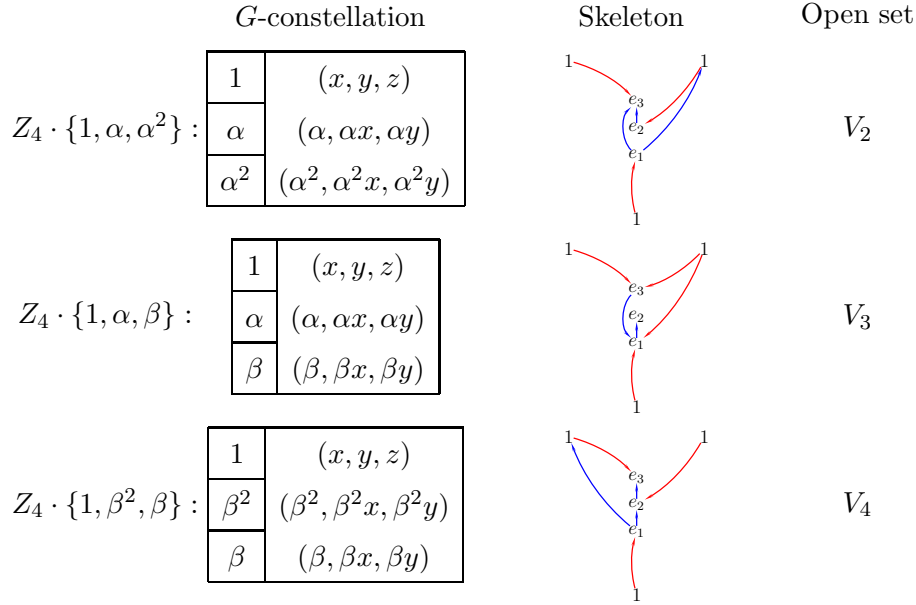
\begin{figure}[h]
\begin{center}
\begin{pspicture}(0,-1)(14,6.75)
	\psset{arrowsize=2pt,arrowlength=3,arcangle=15,nodesep=0pt}

\rput(5,6.5){$G$-constellation}
\rput(9,6.5){Skeleton}
\rput(12,6.5){Open set}
\rput(12,5){$V_2$}
\rput(12,2.5){$V_3$}
\rput(12,0){$V_4$}

%%%%%%%%%%%%%%%%%%%%%%%% V_2
% N-constellations
\rput(2,5){$Z_4\cdot\{1,\alpha,\alpha^2\}:$}
% Splitting into IrrG
\rput(5,5){
{\renewcommand{\arraystretch}{1.5}
\begin{tabular}{|c|c|}
\cline{1-2}
$1$ & $(x,y,z)$  \\
\cline{1-1}
$\alpha$ & $(\alpha,\alpha x,\alpha y)$ \\
\cline{1-1}
$\alpha^2$ & $(\alpha^2,\alpha^2x,\alpha^2y)$ \\
\cline{1-2}
\end{tabular}}
}
\rput(9,5){
\scalebox{0.6}{
% Vertices of the quiver 
	\rput(0,-2){\rnode{0}{\large $1$}}
	\rput(0,0.6){\rnode{33}{\large $e_3$}}	
	\rput(0,0.0){\rnode{32}{\large $e_2$}} 	
	\rput(0,-0.6){\rnode{31}{\large $e_1$}}
	\rput(-1.5,1.5){\rnode{1}{\large $1$}}
	\rput(1.5,1.5){\rnode{2}{\large $1$}}
% arrows
	\ncarc[linecolor=red]{->}{0}{31}
	\ncarc[linecolor=red]{->}{1}{33}
	\ncarc[linecolor=red]{->}{2}{32}
	\ncarc[linecolor=blue,arcangle=-15]{->}{31}{2}
	\ncline[linecolor=blue]{<-}{33}{32}
	\ncarc[linecolor=blue,arcangle=45]{->}{31}{33}
	}}

%%%%%%%%%%%%%%%%%%%%%%%% V_3
% N-constellations
\rput(2,2.5){$Z_4\cdot\{1,\alpha,\beta\}:$}
% Splitting into IrrG
\rput(5,2.5){
{\renewcommand{\arraystretch}{1.5}
\begin{tabular}{|c|c|}
\cline{1-2}
$1$ & $(x,y,z)$  \\
\cline{1-1}
$\alpha$ & $(\alpha,\alpha x,\alpha y)$ \\
\cline{1-1}
$\beta$ & $(\beta,\beta x,\beta y)$ \\
\cline{1-2}
\end{tabular}}
}
\rput(9,2.5){
\scalebox{0.6}{
% Vertices of the quiver 
	\rput(0,-2){\rnode{0}{\large $1$}}
	\rput(0,0.6){\rnode{33}{\large $e_3$}}	
	\rput(0,0.0){\rnode{32}{\large $e_2$}} 	
	\rput(0,-0.6){\rnode{31}{\large $e_1$}}
	\rput(-1.5,1.5){\rnode{1}{\large $1$}}
	\rput(1.5,1.5){\rnode{2}{\large $1$}}
% arrows
	\ncarc[linecolor=red]{->}{0}{31}
	\ncarc[linecolor=red]{->}{1}{33}
	\ncarc[linecolor=red]{->}{2}{33}
	\ncarc[linecolor=red]{->}{2}{31}
	\ncline[linecolor=blue]{->}{31}{32}
	\ncarc[linecolor=blue,arcangle=45]{<-}{31}{33}
	}}

%%%%%%%%%%%%%%%%%%%%%%%% V_4
% N-constellations
\rput(2,0){$Z_4\cdot\{1,\beta^2,\beta\}:$}
% Splitting into IrrG
\rput(5,0){
{\renewcommand{\arraystretch}{1.5}
\begin{tabular}{|c|c|}
\cline{1-2}
$1$ & $(x,y,z)$  \\
\cline{1-1}
$\beta^2$ & $(\beta^2,\beta^2x,\beta^2y)$ \\
\cline{1-1}
$\beta$ & $(\beta,\beta x,\beta y)$ \\
\cline{1-2}
\end{tabular}}
}
\rput(9,0){
\scalebox{0.6}{
% Vertices of the quiver 
	\rput(0,-2){\rnode{0}{\large $1$}}
	\rput(0,0.6){\rnode{33}{\large $e_3$}}	
	\rput(0,0.0){\rnode{32}{\large $e_2$}} 	
	\rput(0,-0.6){\rnode{31}{\large $e_1$}}
	\rput(-1.5,1.5){\rnode{1}{\large $1$}}
	\rput(1.5,1.5){\rnode{2}{\large $1$}}
% arrows
	\ncarc[linecolor=red]{->}{0}{31}
	\ncarc[linecolor=red]{->}{1}{33}
	\ncarc[linecolor=red]{->}{2}{32}
	\ncarc[linecolor=blue]{->}{31}{1}
	\ncline[linecolor=blue]{->}{31}{32}
	\ncline[linecolor=blue]{->}{32}{33}
	}}

\end{pspicture}
\end{center}
\caption{$G$-constellations arising from the non-isolated $\frac{1}{3}(1,2,0)$ line.}
\label{OpFixU4}
\end{figure}

It can be checked that the matrices giving the open sets $V_i\subset\Mtheta$ are the following:
\[
\begin{array}{cl}
V_1: & a=(1,0,0), b=(0,0,1), c=(1,c_2,1), A=\left(\begin{smallmatrix}C_1+C_3 \\-c_2C_1A_1+c_2B_1C_3 \\ C_1(c_2^2A_1-1) \end{smallmatrix}\right), B=\left(\begin{smallmatrix}B_1 \\0 \\ A_1\end{smallmatrix}\right), C=\left(\begin{smallmatrix}C_1 \\0 \\ B_1(c_2^2C_1-1)\end{smallmatrix}\right), \\
& u=\left(\begin{smallmatrix}0&1&0 \\ C_3-\omega C_1 & \omega^2 c_2C_1 & \omega^2C_1+\omega B_1 \\ \omega^2c_2C_3 & -\omega^2(c_2^2C_1-1) & -\omega^2 c_2C_1 \end{smallmatrix}\right), v=\left(\begin{smallmatrix}0&1&0 \\ -\omega^2C_1+C_3 & \omega c_2C_1 & \omega C_1+\omega^2B_1 \\ \omega c_2C_3 & -\omega(c_2^2C_1-1) & -\omega c_2C_1 \end{smallmatrix}\right).
\end{array}
\]
\[
\begin{array}{cl}
V_2 :\ a=(1,0,0), b=(B_1-b_3^2B_3,1,b_3), c=(0,0,1), A=\left(\begin{smallmatrix}b_1B_1 + b_3B_3 \\B_3-B_1A_1 \\ (b_1b_3+1)B_1B_3 \end{smallmatrix}\right), B=\left(\begin{smallmatrix}B_1 \\0 \\ B_3\end{smallmatrix}\right), C=\left(\begin{smallmatrix}1 \\0 \\ A_1\end{smallmatrix}\right), \\
u=\left(\begin{smallmatrix}0&1&0 \\ -\omega^2 b_1B_1+b_3B_3& \omega B_1 & \omega^2 +\omega b_3B_1 \\ \omega B_3 & -\omega b_3B_3 & -\omega B_1 \end{smallmatrix}\right), v=\left(\begin{smallmatrix}0&1&0 \\ -\omega b_1B_1+b_3B_3& \omega^2 B_1 & \omega +\omega^2 b_3B_1 \\ \omega^2 B_3 & -\omega^2 b_3B_3 & -\omega^2 B_1 \end{smallmatrix}\right).
\end{array}
\]
\[
\begin{array}{cl}
V_3 :& a=(1,0,0), b=(0,0,1), c=(c_1,1,1), A=\left(\begin{smallmatrix}c_1C_1 + C_3 \\-C_1A_1+B_1C_3 \\ C_1(A_1-c_1^2) \end{smallmatrix}\right), B=\left(\begin{smallmatrix}B_1 \\0 \\ A_1\end{smallmatrix}\right), C=\left(\begin{smallmatrix}C_1 \\0 \\ B_1(C_1-c_1)\end{smallmatrix}\right), \\
& u=\left(\begin{smallmatrix}0&1&0 \\ C_3-\omega C_1c_1 & \omega^2 C_1 & \omega^2C_1+\omega B_1 \\ \omega^2(A_1-C_1c_1) & -\omega^2(C_1-c_1) & -\omega^2 C_1 \end{smallmatrix}\right), v=\left(\begin{smallmatrix}0&1&0 \\ C_3-\omega^2C_1c_1 & \omega C_1 & \omega C_1+\omega^2B_1 \\ \omega(A_1-C_1c_1) & -\omega(C_1-c_1) & -\omega C_1 \end{smallmatrix}\right).
\end{array}
\]
\[
\begin{array}{cl}
V_4 :& a=(1,0,0), b=(0,0,1), c=(C_1-c_3^2C_3,1,c_3), A=\left(\begin{smallmatrix}c_1C_1 + C_3 \\C_3-C_1A_1 \\ (c_1c_3+1)C_1C_3 \end{smallmatrix}\right), B=\left(\begin{smallmatrix}1 \\0 \\ A_1\end{smallmatrix}\right), C=\left(\begin{smallmatrix}C_1 \\0 \\ C_3\end{smallmatrix}\right), \\
& u=\left(\begin{smallmatrix}0&1&0 \\-\omega C_1^2+\omega c_3^2C_1C_3+c_3C_3 & \omega^2 C_1 & \omega +\omega^2 c_3C_1 \\ \omega^2 C_3 & -\omega^2 c_3C_3 & -\omega^2 C_1 \end{smallmatrix}\right), v=\left(\begin{smallmatrix}0&1&0 \\ -\omega^2 C_1^2+\omega^2 c_3^2C_1C_3+c_3C_3 & \omega C_1 & \omega^2 +\omega c_3C_1 \\ \omega C_3 & -\omega c_3C_3 & -\omega C_1 \end{smallmatrix}\right). 
\end{array}
\]

As in Section \ref{LocalCoord}, by using the McKay quiver as the quiver between the Cohen-Macaulay modules $S_\rho$ we can compute the local coordinates at every open set obtaining:
\begin{align*}
V_1 &= \mathbb C^3_{B_1,c_2,C_1} = \mathbb C[-f_1R_0/R_2,-\sqrt{3}f_3/R_0,-f_2R_0/R_1], \\
V_2 &= \mathbb C^3_{b_3,B_1,B_3} = \mathbb C[-R_1/(\sqrt{3}f_2f_3),\sqrt{3}f_1f_3/R_2,\sqrt{3}f_2^2f_3/R_2 ], \\
V_3 &= \mathbb C^3_{B_1,c_1,C_1} = \mathbb C[-\sqrt{3}f_1f_3/R_2,-R_0/(\sqrt{3}f_3),\sqrt{3}f_2f_3/R_1], \\
V_4 & = \mathbb C^3_{c_3,C_1,C_3} =\mathbb C[R_2/(\sqrt{3}f_1f_3),\sqrt{3}f_2f_3/R_1,\sqrt{3}f_1^2f_3/R_1].
\end{align*}

Moreover, the glueing between the different open sets is given as follows:
\begin{align*}
&V_4 \ni (c_3,C_1,C_3) \longleftrightarrow (-c_3^{-1},C_1-c_3^2C_3,C_1) \in V_3 \\
&V_1 \ni (B_1,c_2,C_1) \longleftrightarrow (B_1c_2,c_2^{-1},c_2C_1) \in V_3\\
&V_2 \ni (b_3,B_1,B_3) \longleftrightarrow (-B_1,B_1-b_3^2B_3,-b_3^{-1}) \in V_3
\end{align*}

Hence, the fiber of the origin of the quotient space are three rational curves meeting in a point. The dual graph with the appropriate degrees of the normal bundles is the following:

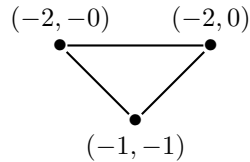
\begin{center}
\begin{pspicture}(0,-1.5)(2.5,0.75)
\rput(0,0){\rnode{1}{$\bullet$}}
\rput(2,0){\rnode{2}{$\bullet$}}
\rput(1,-1){\rnode{3}{$\bullet$}}
\ncline{-}{1}{2}
\ncline{-}{2}{3}
\ncline{-}{3}{1}
\rput(0,0.35){\small $(-2,-0)$}
\rput(2,0.35){\small $(-2,0)$}
\rput(1,-1.35){\small $(-1,-1)$}
\end{pspicture}
\end{center}

In this case the chamber $C\subset\Theta$ for which $\Hilb{G/N}{\Hilb{N}{\C^3}}\cong\mathcal{M}_C$ is given by the inequalities $\theta_1+\theta_3>0$, $\theta_2+\theta_3>0$ and $\theta_1+\theta_2+\theta_3<0$. On the other hand, the fibre over 0 in $\Hilb{G}{\C^3}$ is given by the dual graph 
\begin{center}
\begin{pspicture}(0,-0.5)(3.5,0.75)
\rput(0,0){\rnode{1}{$\bullet$}}
\rput(1.5,0){\rnode{2}{$\bullet$}}
\rput(3,0){\rnode{3}{$\bullet$}}
\ncline{-}{1}{2}
\ncline{-}{2}{3}
\rput(0,0.35){\small $(-1,-1)$}
\rput(1.5,0.35){\small $(-3,1)$}
\rput(3,0.35){\small $(-1,-1)$}
\end{pspicture}
\end{center}
and the chamber $C'$ for $\Hilb{G}{\C^3}\cong\mathcal{M}_{C'}$ is given by $\theta_i>0$ for $i\neq0$. See \cite{NdCS} for details. This concludes the proof of the case $G_{12}$ in Theorem \ref{thm:isoDim3} (iii).

\section{When does Hilb of Hilb coincide with Hilb?}\label{When}

In this section we study the relation between $G$-Hilb and $G/N$-Hilb($N$-Hilb). Since both can be constructed as moduli spaces of representations of the McKay quiver, we may ask when they are isomorphic as moduli spaces (i.e.\ their tautological bundles coincide) and if not, when their underlying algebraic varieties are isomorphic. We answer these questions in many cases.

Considering them as moduli spaces we have that $G$-Hilb $\cong\mathcal{M}_C$ and $G/N$-Hilb($N$-Hilb) $\cong\mathcal{M}_{C'}$ for chambers $C,C'\subset\Theta$, where $C$ is the chamber containing the 0-generated stability and $C'$ is the chamber containing the parameter in Definition \ref{theta}. Then the problem in this case is to determine which groups $G$ admit a normal subgroup $N$ such that $C=C'$. We give a complete answer to this question in the case of $G\subset\GL(2,\C)$ and $G\subset\SL(3,\C)$ in Theorem \ref{IsoModulidim3}.

As algebraic varieties, in dimension 2 and for $G \subset \SL(2, \C)$, there is nothing to prove since both of them are minimal resolutions of $\C^2/G$ thus isomorphic. For non-Abelian subgroups in $\GL(2,\C)$ we treat the case when $N=G\cap\SL(2,\C)$ and conclude that they are non-isomorphic (see Proposition \ref{prop:non-min}). For $G\subset\SL(3,\C)$ we give a complete answer when the group $G$ is Abelian by using the method of Craw and Reid \cite{CR02} to obtain the triangulation of the junior simplex $\Delta$ which corresponds to $\Hilb{G}{\C^3}$. As we saw in Section \ref{AbCase} the triangulation for $\Hilb{G/N}{\Hilb{N}{\C^3}}$ is given by using the Craw-Reid method in two steps, first for $N$ and then for $G/N$. Comparing both triangulations we are able to describe in Theorem \ref{IsoVarAb} all possible configurations for $G$ and $N$ for which there is an isomorphism of varieties over $\C^3/G$.

We finish treating some non-Abelian small subgroups $G \subset \SL(3, \C)$ for which $\Hilb{G}{\C^3}$ is not isomorphic to $\Hilb{G/N}{\Hilb{N}{\C^3}}$. In particular, finite subgroups of $\SO(3)$ of types $D_{2n}$ and $G_{12}$ with $N$ being the maximal normal subgroup, and non-Abelian intransitive subgroups with $N=G \cap \SL(2, \C)$. 

\subsection{As moduli spaces}

Let $G\subset\GL(n,\C)$, assuming either $n=3$ and $G\subset\SL(3,\C)$ or $n=2$ and $G\subset\GL(2,\C)$ small. Let $N$ be a normal subgroup in $G$. With the same notation as in Section \ref{Sect:Moduli} let $Y_1:=\Hilb{N}{\C^n}$ and $Y_2:=\Hilb{G/N}{Y_1}$ with universal families $\mathcal{Z}_1$ and $\mathcal{Z}_2$ respectively, and denote by $\mathcal{U}:=p_{20*}(\OO_{\mathcal{Z}_1\times_{Y_1}\mathcal{Z}_2})$ the flat family over $Y_2$ of $G$-constellations by the projection $p_{20}:Y_2\times Y_1\times\C^n\longrightarrow Y_2\times\C^n$. 
\begin{lem}
Put $X_1=\C^n/N$ and $X_2=Y_1/(G/N)$.
Then, $\mathcal{Z}_1$ is the reduced part of the fibre product $Y_1\times_{X_1}\C^n$
and $\mathcal{Z}_2$ is the reduced part of $Y_2 \times_{X_2} Y_1$.
\begin{proof}
It is sufficient to prove the latter statement.
$\mathcal{Z}_2$ is obviously a closed subscheme of $Y_2 \times_{X_2} Y_1$
and the morphism $\mathcal{Z}_2 \to Y_2 \times_{X_2} Y_1$ is an isomorphism over the
generic point of $Y_2$.
Now $\mathcal{O}_{\mathcal{Z}_2}$, regarded as an $\mathcal{O}_{{Y}_2}$-module via the finite morphism $p_2$, is locally free and therefore is the quotient of
$\mathcal{O}_{Y_2\times_{X_2}Y_1}$ by the torsion part, which must be the nilradical.
\end{proof}
\end{lem}

As we know, $\Hilb{G}{\C^n}\cong\mathcal{M}_C$ and $Y_2\cong\mathcal{M}_{C'}$ are both resolutions of $\C^n/G$ isomorphic to a moduli space of $G$-constellations for some chambers $C,C'\subset\Theta$. Then,

\begin{align*}
C=C'	&\iff \text{$\exists$ a closed subscheme $\mathcal{Z}\subset Y_2\times\C^n$ such that $\mathcal{U}\cong\OO_\mathcal{Z}$ in $\Coh (Y_2 \times \C^n)$} \\
		&\iff \phi:\mathcal{Z}_2\times_{Y_1}\mathcal{Z}_1\longrightarrow Y_2\times\C^n \text{ is a closed immersion} \\
		& ~~\Longrightarrow\phi \text{ is injective on the $\C$-valued points.}
\end{align*}
If we denote by $Y_1^{G/N}$ the fixed locus of the action of $G/N$ into $Y_1$, we obtain the following sufficient condition for $Y_2$ and $\Hilb{G}{\C^n}$ being non-isomorphic as moduli spaces.

\begin{lem}\label{LemDimFibre} If $\tau_1^{-1}(0)\nsubseteq Y_1^{G/N}$ then $C\neq C'$.
\end{lem}

\begin{proof} Assume $y,g(y)\in\tau_1^{-1}(0)\subset Y_1$ are two distinct points for some $g \in G/N$. Then there exists a $G/N$-cluster $W\in Y_2$ such that $y,g(y)\in\Supp(W)$. Then $(W,y),(W,g(y))\in\mathcal{Z}_2$ and $(y,0),(g(y),0)\in\mathcal{Z}_1$, which implies that $(W,y,0),(W,g(y),0)\in\mathcal{Z}_2\times_{Y_1}\mathcal{Z}_1$ are two distinct points. But then $\phi(W,y,0)=\phi(W,g(y),0)=(W,0)$ so $\phi$ is not injective.
\end{proof}

As the following theorem shows, in dimensions 2 and 3 the cases when $G/N$-Hilb($N$-Hilb) and $G$-Hilb coincide as moduli spaces are very few.

\begin{thm}\label{IsoModulidim3} 
(i) Let $G\subset\GL(2,\C)$ be a finite small subgroup and let $N\neq G,\{1\}$ be a normal subgroup in G. Then 
\[
C=C' \iff G\cong\frac{1}{rs}(1,1) \text{ and } N\cong\frac{1}{s}(1,1).
\]
(ii) Let $G\subset\SL(3,\C)$ be a finite subgroup and let $N\neq G,\{1\}$ be a normal subgroup in $G$. Then,
\[ 
C=C' \iff G\cong\frac{1}{2r}(1,1,2r-2) \text{ and } N\cong\frac{1}{2}(1,1,0).
\] 
\end{thm}

\begin{proof} We begin by proving (ii). Recall that if $N \subsetneq G\subset\SL(3,\C)$ then $\dim Y_1^{G/N}\leq1$. Therefore if $C=C'$ then by Lemma \ref{LemDimFibre} we must have $\dim(\tau_1^{-1}(0))\leq1$. Moreover, since the $\tau^{-1}(0)$ is connected it must consist of a single curve. Indeed, if we have more than one curve in $\tau^{-1}(0)$ fixed by $G/N$ then at any intersection point of two curves a 2-dimensional subspace of the tangent space is fixed, thus $\dim Y_1^{G/N}>1$ a contradiction. 
Especially, the Grothendieck group of coherent sheaves on $Y_1$ whose supports
are contained in $\tau^{-1}(0)$ is of rank two.

Therefore $N\cong\Z/2\Z$ as in \cite{IN00} and we can suppose it to be isomorphic to $\frac{1}{2}(1,1,0)$. Then $\Hilb{N}{\C^3}\cong U_1\cup U_2$ where $U_1\cong\C^3_{x^2,z,\frac{y}{x}}$, $U_2\cong\C^3_{y^2,z,\frac{x}{y}}$ and $\tau^{-1}(0)=E\cong\PP^1$ with coordinates $x:y$. After extending the action of $G/N$ on $\C[x,y,z]^N$ naturally into $\C(x,y,z)^N$, we have that $G/N$ fixes $E$ if $g(\frac{x}{y})=\frac{x}{y}$ for all $g\in G/N$. In other words, $g$ as an element of $G$ can be written in the form $\left(\begin{smallmatrix}\varepsilon&0&a\\0&\varepsilon&b\\0&0&\varepsilon^{n-2}\end{smallmatrix}\right)$ with $a,b\in\C$ and $\varepsilon$ a primitive $n$-th root of unity. But the group $N$ is normal in $G$, so $g$ must commute with any element in $N$. This implies that $a=b=0$ and since $G$ contains $\frac{1}{2}(1,1,0)$ as a subgroup $n$ has to be an even number. Thus $G\cong\frac{1}{2r}(1,1,2r-2)$ for some $r>1$.

Conversely, if the group is of the form $G\cong\frac{1}{2r}(1,1,2r-2)$ and $N\cong\frac{1}{2}(1,1,0)$, then by the construction of $G$-constellations in the Abelian case of Section \ref{AbCase}, we see that the elements $\omega_\tau$ for $\tau\in\Irr(G/N)$ are not Laurent monomials. More precisely, for $U_1$ we have $\omega_{\tau_0}=1$, $\omega_{\tau_1}=x^2$ and for $U_2$ we have $\omega_{\tau_0}=1$, $\omega_{\tau_1}=y^2$, so there are no Laurent monomials in the $G$-constellations $\mathcal{Z}$ of $\Hilb{G/N}{\Hilb{N}{\C^3}}$. This means that they are precisely the $G$-graphs of $\Hilb{G}{\C^3}$, thus the chambers are the same.

The proof of (i) follows the same argument. If $C=C'$ then by Lemma \ref{LemDimFibre} we have $\tau^{-1}(0)\subset Y_1^{G/N}$. Since $\tau^{-1}(0)$ is a chain of rational curves then $\dim(Y_1^{G/N})=1$, which in particular implies that $G/N$ is not small. As in the proof of (ii) we have that $\tau^{-1}(0)$ must consists of a single rational curve, thus we may assume $N$ to be isomorphic $\frac{1}{s}(1,1)$ for some $s\geq2$. The exceptional divisor $E\cong\PP^1$ in $Y_1$ has coordinates $(x:y)$ and it is invariant under $G/N$. As before, it follows that any $g\in G/N$ has to be of the form $\frac{1}{n}(1,1)$ for some $n\geq2$, and since $N$ is subgroup we have that $n=rs$ for some $r\geq2$.

Conversely if $G\cong\frac{1}{rs}(1,1)$ and $N\cong\frac{1}{s}(1,1)$, the action of $G/N$ in the two affine pieces of $Y_1$ is of type $1/rs(s,0)$, which is not small as expected (in other words, $Y_1/(G/N) \cong Y_1$ is nonsingular). In terms of $G$-constellations $Y_1$ has two building blocks $\Gamma_1=\{1,\ldots,x^{s-1}\}$ and $\Gamma_2=\{1,\ldots,y^{s-1}\}$, and after the action of $G/N$ we obtain the $G$-constellations $\mathcal{Z}_1=\{1,\ldots,x^{rs-1}\}$ and $\mathcal{Z}_2=\{1,\ldots,y^{rs-1}\}$, so the chambers coincide.
\end{proof}

As an immediate consequence, for any finite subgroup $G\subset\SL(2,\C)$ they are never the same moduli space.

\begin{col}\label{IsoModulidim2}
Let $G\subset\SL(2,\C)$ be a finite subgroup. Then $C\neq C'$.
\end{col}

\begin{proof} If $G$ is in $\SL(2,\C)$ then $rs=2$, thus either $r=1$ or $s=1$, which contradicts $N\neq G,\{1\}$ and we are done.
\end{proof}

\subsection{As varieties}

In this section we treat the problem of when $G$-Hilb and $G/N$-Hilb($N$-Hilb) are isomorphic as algebraic varieties. We start with the dimension 2 case.

\subsubsection{$\Hilb{G/N}{\Hilb{N}{\C^2}}$}
Let $G$ be a finite subgroup of $\SL(2,\C)$. It is well know that the minimal resolution of $\C^2/G$ is unique. Therefore, since both $\Hilb{G}{\C^2}$ and $\Hilb{G/N}{\Hilb{N}{\C^2}}$ are minimal they are isomorphic. 

Now let $G$ be a finite small subgroup of $\GL(2,\C)$ and take $N=G\cap \SL(2, \C)$. In this particular case we have the following result, which proves Theorem \ref{thm:isoDim2} (iii).

\begin{prop}\label{prop:non-min}
Let $G\subset\GL(2,\C)$ be a finite non-Abelian subgroup such that $G \not\subseteq \SL(2, \C)$ and let $N=G\cap \SL(2, \C)$. Then $\Hilb{G/N}{\Hilb{N}{\C^2}}$ is not a minimal resolution of $\C^2/G$.
\end{prop}

\begin{proof}
In this proof, we use the following notation: $Y_1=\Hilb{N}{\C^2}$,
$X_2=Y_1/(G/N)$ and $Y_2=\Hilb{G/N}{Y_1}$.
Then $\tau_1: Y_1 \to \C^2/N$ is a crepant resolution and $Y_2$ is the
minimal resolution of $X_2$.
Let $\{E_i\}=\{E_0, E_1, \dots\}$ be the exceptional curves on $Y_1$.
Since $G$ is non-Abelian, we may assume $E_0$ intersects three other exceptional curves,
$E_1$, $E_2$ and $E_3$.
We denote by $\bar{E_{i}}$ the images of $E_i$ in $X_2$
and by $i'$ the $G/N$-orbit of $i$.
Since $Y_1$ is a crepant resolution, we have $K_{Y_1} \equiv 0$.
Therefore, if $e_i$ denotes the ramification index of $\tau_1$ along $E_i$,
we have $K_{X_2} \equiv -\sum_{i'} (1-\frac{1}{e_{i}})\bar{E_{i}}$.
Now since the action of $G/N$ fixes every point on $E_0$,
we see $e_0=|G/N|$ and $e_1=e_2=e_3=1$.
The canonical bundle of $Y_2$ is written
$$
K_{Y_2} \equiv -\sum_{i'} (1-\frac{1}{e_i})\widetilde{E}_{i} + \sum_j a_j F_j
$$
where $\widetilde{E}_i$ is the proper transform of $\bar{E_i}$ and
$\{F_j\}$ are the exceptional curves of $\tau_2$ with discrepancies $a_j$.
Since $-1<a_j\le 0$, it follows
$$
K_{Y_2}\cdot\widetilde{E}_1=-(1-\frac{1}{e_0})+\sum_{F_j \cap \widetilde{E}_1 \ne \emptyset} a_j < 0
$$
which shows that $K_{Y_2}$ is not nef.
\end{proof}

%%%%%%%%%%%%%%%

\subsubsection{Finite Abelian subgroups in $\SL(3,\C)$}
In this section we assume that $G\subset\SL(3,\C)$ is a finite Abelian subgroup. We use the same notation as in Section \ref{AbCase}. We start by recalling the properties of the the triangulation of the junior simplex $\Delta$ constructed by Craw and Reid in \cite{CR02} corresponding to $\Hilb{G}{\C^3}$ that we need. By abusing the notation, in what follows we identify $\Hilb{G}{\C^3}$ with its corresponding triangulation of $\Delta$ given by \cite{CR02}.

A {\em regular triangle} of side $r$ in $\Delta$ is a lattice triangle with $r+1$ points on each edge. In $\Hilb{G}{\C^3}$ every regular triangle of side $r$ is triangulated with the {\em regular tesselation}, which is done by drawing $r-1$ parallel lines to the sides of the regular triangle, obtaining $r^2$ regular triangles of side 1. There are only two types of regular triangles appearing in $\Hilb{G}{\C^3}$, namely the {\em corner triangle} and the {\em meeting of champions}, both shown in Figure \ref{RegT}.

\begin{figure}[ht]
\begin{center}
\begin{pspicture}(0,-0.5)(10,3.5)
	\psset{arcangle=15}
\rput(0,0){
	\scalebox{0.35}{
	\rput(0,0){\rnode{e3}{\Large $\bullet$}}
	\rput(5,8.65){\rnode{e1}{\Large $\bullet$}}
	\rput(10,0){\rnode{e2}{\Large $\bullet$}}
	\rput(5.67, 5.19){\rnode{1}{\Large $\bullet$}}
	\rput(5.89, 4.03){\rnode{2}{\Large $\bullet$}}
	\rput(6.11, 2.89){\rnode{3}{\Large $\bullet$}}
	\rput(6.33, 1.73){\rnode{4}{\Large $\bullet$}}
	\rput(4.22, 1.15){\rnode{5}{\Large $\bullet$}}
	\rput(2.11, .576){\rnode{6}{\Large $\bullet$}}
	\rput(3.78, 3.46){\rnode{7}{\Large $\bullet$}}
	\rput(1.89, 1.73){\rnode{8}{\Large $\bullet$}}
	\rput(4,2.30){\rnode{9}{\Large $\bullet$}}
	\ncline{e1}{e2}\ncline{e2}{e3}\ncline{e3}{e1}
	\ncline{e3}{4}\ncline{4}{1}\ncline[linestyle=dotted]{1}{e1}
	\ncline{e3}{1}\ncline{6}{2}\ncline{5}{3}
	\ncline{8}{6}\ncline{7}{5}
	\ncline{8}{3}\ncline{7}{2}
	\rput(-0.45,-0.45){\Huge $e_3$}
	\rput(10.45,-0.45){\Huge $e_2$}
	\rput(5,9.1){\Huge $e_1$}
	}}
\rput(6,0){
	\scalebox{0.35}{
	\rput(0,0){\rnode{e3}{\Large $\bullet$}}
	\rput(5,8.65){\rnode{e1}{\Large $\bullet$}}
	\rput(10,0){\rnode{e2}{\Large $\bullet$}}
	\rput(3.57, 1.24){\rnode{1}{\Large $\bullet$}}
	\rput(7.14, 2.47){\rnode{2}{\Large $\bullet$}} 
	\rput(4.29, 4.95){\rnode{3}{\Large $\bullet$}}
	\rput(4.05, 3.70){\rnode{4}{\Large $\bullet$}} 
	\rput(3.81, 2.47){\rnode{5}{\Large $\bullet$}}
	\rput(5.24, 4.12){\rnode{6}{\Large $\bullet$}} 
	\rput(6.19, 3.29){\rnode{7}{\Large $\bullet$}}
	\rput(4.76, 1.65){\rnode{8}{\Large $\bullet$}}
	\rput(5.95, 2.06){\rnode{9}{\Large $\bullet$}}
	\rput(5., 2.89){\rnode{10}{\Large $\bullet$}}
	\ncline{e1}{e2}\ncline{e2}{e3}\ncline{e3}{e1}
	\ncline[linestyle=dotted]{e3}{1}\ncline{1}{2}
	\ncline[linestyle=dotted]{e2}{2}\ncline{2}{3}
	\ncline[linestyle=dotted]{e1}{3}\ncline{3}{1}
	\ncline{5}{7}\ncline{4}{6}
	\ncline{5}{8}\ncline{4}{9}
	\ncline{6}{8}\ncline{7}{9}
	\rput(-0.45,-0.45){\Huge $e_3$}
	\rput(10.45,-0.45){\Huge $e_2$}
	\rput(5,9.1){\Huge $e_1$}
	}}
\rput(1.75,-0.5){$(a)$}
\rput(7.75,-0.5){$(b)$}
\end{pspicture}
\caption{Types of regular triangles: (a) corner triangle and (b) meeting of champions.}
\label{RegT}
\end{center}
\end{figure}
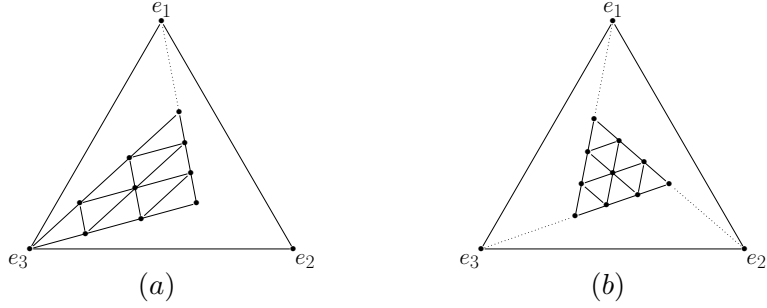

In particular, the sides of a regular triangle always extends to one of the vertices $e_i$. From the construction we can deduce the following properties that we use repeatedly in the rest of the section.

\begin{prop}\label{KeyProp} Consider the triangulation of $\Delta$ corresponding to $\Hilb{G}{\C^3}$. Then,

(i) Any line contained in $\Delta$ either passes through one of the vertices $e_i$ for $i=1,2,3$, or it is contained in a regular triangle. 

(ii) The valency of a vertex $v$ in $\Delta$ is either 3, 4, 5 or 6.

%(iii) Let $v$ be a vertex in $\Delta$ of valency $\geq5$. Then there exists a straight line $L$ through $v$ and for the rest of the lines there are at most 2 coming into $v$ form the same side of $L$.
\end{prop}

\begin{proof} Part (i) follows from the construction of the triangulation of $\Delta$, and part (ii) form \cite{CR02} Corollary 1.4. %For (iii), by the description of vertices of valency 5 and 6 in \cite{Craw05} cases 3 and 4, the straight line $L$ always exists and one of the sides of $L$ corresponds to a regular triangle, so from that side there are precisely 2 lines coming into $v$. Since by $(ii)$ the valency is at most 6 in the other side there are at most 2.
\end{proof}

\begin{thm}\label{IsoVarAb} Let $G$ be a finite non-simple Abelian subgroup of $\SL(3,\C)$ and let $N$ be a normal subgroup of $G$, with $N\ne G,\{1\}$. Then 
\begin{equation}\label{eq:iso}
\Hilb{G/N}{\Hilb{N}{\C^3}}\cong\Hilb{G}{\C^3}
\end{equation}
as algebraic varieties if and only if we are in one of the following situations
\begin{enumerate}
\item $G/N\cong\Z/m\Z\times\Z/m\Z$ for some $m>1$.
\item $G\cong\frac{1}{r}(1,1,r-2)$ or $G\cong\frac{1}{r}(1,r-1,0)$, i.e.\ $\C^3/G$ has a unique crepant resolution.
\item $G\cong\frac{1}{2r}(1,a,-a-1)$ with $(2r,a)=1$, $a^2\equiv1$ (mod $4r$) and $N\cong\frac{1}{2}(1,1,0)$.
\item There is a subgroup $G'\subset G$ containing $N$ such that $(G',N)$ fits into either $(2)$ or $(3)$ and $G/G'\cong\Z/m\Z\times\Z/m\Z$ for some $m>1$.
\end{enumerate}
\end{thm}

The proof of the theorem is deduced from Lemmas \ref{ZmZm} to \ref{CaseZ2}. The first lemma shows the biggest family of Abelian groups for which we have an an isomorphism of varieties, and constitutes the case (1) in Theorem \ref{IsoVarAb}. The rest of the cases are in some sense sporadic modulo case (4).

\begin{lem}\label{ZmZm} If $G/N\cong\Z/m\Z\times\Z/m\Z$ then $\Hilb{G/N}{\Hilb{N}{\C^3}}\cong\Hilb{G}{\C^3}$ as varieties.
\end{lem}

\begin{proof} Assume $G/N \cong \Z/m\Z \times \Z/m\Z$. Let $L\supset L' \supset \Z^3$ be the toric lattices for
$\C^3/G$ and $\C^3/N$ respectively. Then the assumption implies $L/L' \cong \Z/m\Z \times \Z/m\Z$. Since $L$ and $L'$ are generated by elements on the junior simplex, we have decompositions $L=L_0 \oplus \Z e_1$ and $L'=L'_0 \oplus \Z e_1$, where $L_0=L \cap \R^2_\Delta$ and $L'_0=L' \cap \R^2_\Delta$.

Then we have $L_0/L'_0 \cong \Z/m\Z \times \Z/m\Z$ for the two-dimensional lattices $L_0 \supset L'_0$, which implies $L_0 =(1/m)L'_0$. So the Newton polygon for $G$ at $e_1$ is $1/m$ times that for $G'$, thus the triangulations are the same by \cite{CR02}.
\end{proof} 

The following lemma justifies the Case $(4)$ in Theorem \ref{IsoVarAb} and allows us to obtain isomorphism as varieties between $Y_2$ and $\Hilb{G}{\C^3}$ by combining the Cases $(1)$, $(2)$ and $(3)$. 

\begin{lem}\label{Lem:ZmZm} Suppose there exists a surjection $\phi:G/N\twoheadrightarrow\Z/m\Z\times\Z/m\Z$ for some $m>1$ and let $G'$ be the pullback of $\Ker(\phi)$ to $G$.
Then \eqref{eq:iso} holds for the pair $(G,N)$ if and only if it holds for $(G',N)$.
\end{lem}

\begin{proof} If there exists a sequence of normal subgroups $G\rhd G'\rhd N$ and $G/G'\cong\Z/m\Z\times\Z/m\Z$ then we can construct $\Hilb{G/N}{\Hilb{N}{\C^3}}$ in three steps
\[ 
\Hilb{G/G'}{\Hilb{G'/N}{\Hilb{N}{\C^3}}}
\]
so by Lemma \ref{ZmZm} we can take $G'$ instead of $G$. 
\end{proof}

Therefore from now on we assume that no such surjection exists,
which means that $G/N$ is cyclic.

\begin{lem} If $\Hilb{G/N}{\Hilb{N}{\C^3}}\cong\Hilb{G}{\C^3}$ then there is no regular triangle of side $\geq2$ in $\Hilb{N}{\C^3}$.
\end{lem}

\begin{proof} Let $T$ be a regular triangle in $\Delta$ of side $\geq2$,
where $\Delta$ is triangulated for $\Hilb{N}{\C^3}$. Then $T$ is triangulated by the regular tesselation and there always exists a triangle $\Delta_i\subset\Delta$ which does not contain any of the vertices of $T$. Now consider the action of $G/N$ on $\Hilb{N}{\C^3}$ and the corresponding triangulation on $\Delta_i$. Since any side of $T$ extends to some vertex $e_i$, by Proposition \ref{KeyProp} (i) any line of the triangulation for
$\Hilb{G}{\C^3}$ inside $\Delta_i$ must be parallel to some side of $\Delta_i$.
This implies that the action of $G/N$ into $\Delta_i$ has to be of the form $\Z/m\Z\times\Z/m\Z$, so that $\Delta_i$ is triangulated again with the regular tesselation. But then $G/N\cong\Z/m\times\Z/m\Z$ which contradicts our assumption.
\end{proof}

\begin{lem} The triangulation of $\Delta$ corresponding to $\Hilb{N}{\C^3}$ contains only regular triangles of side 1 if and only if $N\cong\frac{1}{r}(1,1,r-2)$, $\frac{1}{r}(1,r-1,0)$ or $\frac{1}{7}(1,2,4)$
\end{lem}

\begin{proof} Notice that there are only regular triangles of side 1 if and only if there exists a unique crepant resolution of $\C^3/N$, namely $\Hilb{N}{\C^3}$. Indeed, if
every regular triangle in the triangulation of $\Delta$ is of side 1,
then Proposition \ref{KeyProp} (i) implies that every line goes to one of the $e_i$'s, and there is no parallelogram in this triangulation.
Therefore, there is no flop from $\Hilb{N}{\C^3}$.
% a flop from $\Hilb{N}{\C^3}$ exists then there are 4 points in $\Delta$ not 3 of them align, and any triangulation of them would contradict Proposition \ref{KeyProp} (i). 
Conversely, if there exist a unique crepant resolution there is no parallelogram in the triangulation of $\Delta$, in particular there is no regular triangle of side bigger than one. 

Finally notice that the Abelian groups for which there exist a unique crepant resolution are $\frac{1}{r}(1,1,r-2)$, $\frac{1}{r}(1,r-1,0)$ or $\frac{1}{7}(1,2,4)$. This follows from the fact that either all points are contained in a line or $N\cong\frac{1}{7}(1,2,4)$, otherwise we would have 4 points in $\Delta$ with not 3 of them align, hence a flop.
\end{proof}

From the three possibilities for $N$ in the previous lemma we can exclude the case $N\cong\frac{1}{7}(1,2,4)$. Indeed, let $\Delta=\bigcup_{i=1}^7\Delta_i$ the triangulation corresponding to $\Hilb{N}{\C^3}$. Then, on the regular triangle with vertices $\frac{1}{7}(1,2,4)$, $\frac{1}{7}(2,4,1)$ and $\frac{1}{7}(4,1,2)$ in the middle of $\Delta$, the next triangulation created by the action of $G/N$ has to be a regular triangle again (otherwise it would contradict Proposition \ref{KeyProp} (i)), so we are again in the case of Lemma \ref{ZmZm}.

Now consider the case $N\cong\frac{1}{r}(1,1,r-2)$. Then every lattice point $P_j\in\Delta$ distinct from the vertices $e_i$ for $i=1,2,3$, are on a line $L$ passing through $e_3$. If we consider the triangulation induced by $G/N$, by Proposition \ref{KeyProp} (i) there are no new lines going out any of the points $P_j$ unless $r=2$ or $3$. Therefore, either $G/N\cong\frac{1}{s}(1,1,s-2)$ so that every new point is again on the line $L$ and $G$ is of type $\frac{1}{n}(1,1,n-2)$, or $N\cong\frac{1}{3}(1,1,1)$, or $N\cong\frac{1}{2}(1,1,0)$, or we are in the case of Lemma \ref{Lem:ZmZm}. 

Similarly, if $N\cong\frac{1}{r}(1,r-1,0)$ then either $G/N\cong\frac{1}{s}(1,s-1,0)$ for some $s|r$ or $N\cong\frac{1}{2}(1,1,0)$. In any case, we are either in the case $(2)$ of the Theorem \ref{IsoVarAb} or $N$ has order 2 or 3.

Let $N\cong\frac{1}{3}(1,1,1)$ and let $P$ be the point in the center of the triangulation $\Delta=\bigcup_{i=1}^3\Delta_i$ of $\Hilb{N}{\C^3}$ where the 3 lines $L_i$ from the vertices $e_i$ meet ($i=1,2,3$). Now consider the second triangulation produced by $G/N$. Because of the ``meeting of champions'' only one of the lines $L_i$ can be extended, so that the final valency of $P$ is at most $4$. 

If the valency is 3 then every $\Delta_i$ has a basic triangle around $P$. 
Since the areas of the basic triangles are the same, it follows that the three vectors at $P$ have the same length, and since two of them form a basis of the two-dimensional lattice $\Z^2_{\Delta}$, we can conclude that $G/N\cong\Z/m\Z\times\Z/m\Z$ as in Lemma \ref{ZmZm}. If the valency is 4 then at least 2 of the $\Delta_i$'s must have a regular triangle around $P$, which must be of the same side since they share a generator of the lattice. See Figure below.

%\begin{figure}[ht]
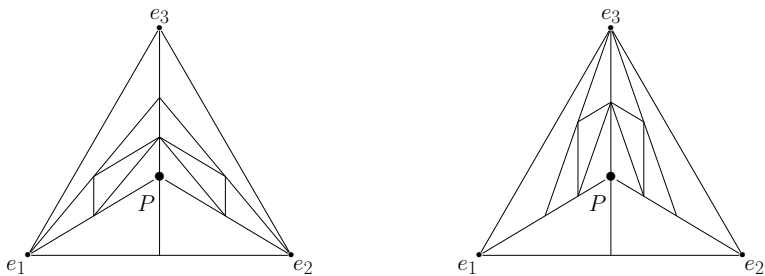
\begin{center}
\begin{pspicture}(0,-0.25)(10,3.5)
	\psset{arcangle=15}
\rput(0,0){
	\scalebox{0.35}{
	\rput(0,0){\rnode{e3}{\Large $\bullet$}}
	\rput(5,8.65){\rnode{e1}{\Large $\bullet$}}
	\rput(10,0){\rnode{e2}{\Large $\bullet$}}
	\rput(5,0){\rnode{P}{}}
	\rput(5,3){\rnode{1}{\Huge $\bullet$}}
	\rput(5,6){\rnode{2}{}}
	\rput(5,4.5){\rnode{3}{}}
	\rput(2.5,3){\rnode{4}{}}
	\rput(2.5,1.5){\rnode{5}{}}
	\rput(7.5,3){\rnode{6}{}}
	\rput(7.5,1.5){\rnode{7}{}}
	\ncline{e1}{e2}\ncline{e2}{e3}\ncline{e3}{e1}
	\ncline{P}{e1}\ncline{1}{e2}\ncline{1}{e3}
	\ncline{e3}{2}\ncline{2}{e2}
	\ncline{7}{3}\ncline{3}{5}
	\ncline{5}{4}\ncline{4}{3}
	\ncline{3}{6}\ncline{6}{7}
	\rput(-0.45,-0.45){\Huge $e_1$}
	\rput(10.45,-0.45){\Huge $e_2$}
	\rput(5,9.1){\Huge $e_3$}
	\rput(4.5,2){\Huge $P$}
	}}
\rput(6,0){
	\scalebox{0.35}{
	\rput(0,0){\rnode{e3}{\Large $\bullet$}}
	\rput(5,8.65){\rnode{e1}{\Large $\bullet$}}
	\rput(10,0){\rnode{e2}{\Large $\bullet$}}
	\rput(5,0){\rnode{P}{}}
	\rput(5,3){\rnode{1}{\Huge $\bullet$}}
	\rput(2.5,1.5){\rnode{5}{}}
	\rput(7.5,1.5){\rnode{7}{}}
	\rput(3.75,2.25){\rnode{2}{}}
	\rput(3.75,5.075){\rnode{3}{}}	
	\rput(6.25,2.25){\rnode{4}{}}
	\rput(6.25,5.075){\rnode{6}{}}
	\rput(5,5.825){\rnode{8}{}}
	\ncline{e1}{e2}\ncline{e2}{e3}\ncline{e3}{e1}
	\ncline{P}{e1}\ncline{1}{e2}\ncline{1}{e3}
	\ncline{e1}{5}\ncline{e1}{7}
	\ncline{2}{8}\ncline{8}{4}
	\ncline{2}{3}\ncline{3}{8}
	\ncline{8}{6}\ncline{6}{4}
	\rput(-0.45,-0.45){\Huge $e_1$}
	\rput(10.45,-0.45){\Huge $e_2$}
	\rput(5,9.1){\Huge $e_3$}
	\rput(4.5,2){\Huge $P$}
	}}
\end{pspicture}
%\caption{}
%\label{}
\end{center}
%\end{figure}

But then there exists a subgroup $G'\cong\frac{1}{3r}(1,1,3r-2)$ such that $G/G' \cong \Z/m\Z \times \Z/m\Z$ with $m>1$. Indeed, let $M$ be the middle point of $e_1$ and $e_2$. Then every regular triangle inside the triangle $e_1PM$ has $e_1$ as a vertex and the other two vertices lie on the segment $PM$. Moreover, these regular triangles have the same area. Then we can see that there is a subgroup
$G' \cong \frac{1}{3r}(1,1,3r-2)$ such that these regular triangles for $G$ are basic triangles for $G'$. If these regular triangles are divided into $m^2$ basic triangles, then $G/G' \cong \Z/m\Z \times \Z/m\Z$.

In the case $N\cong\frac{1}{2}(1,1,0)$ we have the following lemma which gives the case (3) in Theorem \ref{IsoVarAb} and finishes the proof.

\begin{lem}\label{CaseZ2} Let $N\cong\frac{1}{2}(1,1,0)$ and $Y_2:=\Hilb{G/N}{\Hilb{N}{\C^3}}$. Then
\[
Y_2\cong\Hilb{G}{\C^3} \Longleftrightarrow G\cong\frac{1}{2r}(1,a,-a-1) \text{ with $a^2\equiv1$ (mod $4r$)}
\]
\end{lem}

\begin{proof}
Let $\Delta=\Delta_1\cup\Delta_2$ the triangulation for $\Hilb{N}{\C^3}$ and $P:=\frac{1}{2}(1,1,0)\in\Delta$. Suppose that we have the isomorphism as varieties. After the action of $G/N$ on $\Delta$, by Proposition \ref{KeyProp} (i) the triangles $\Delta_i$ must contain regular triangles around $P$. The sides of the regular triangles must pass through some vertex $e_i$ so there are two possible configurations (see Figure \ref{CaseNZ2}). 

Notice that any line passing through the point $P$ must go to one of the vertices $e_i$ for $i=1,2,3$ (otherwise, the sides of two regular triangles would intersect), and in particular there exists the diagonal line $L:=e_3P$. Then the vectors $v_1, v_2$ are sides of a basic triangle, therefore they form a basis of the 2-dimensional lattice $\Z^2_\Delta$.  

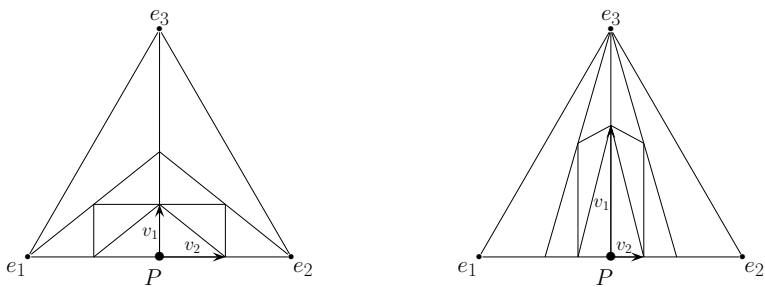
\begin{figure}[ht]
\begin{center}
\begin{pspicture}(0,-0.25)(10,3.5)
	\psset{arcangle=15}
\rput(0,0){
	\scalebox{0.35}{
	\rput(0,0){\rnode{e3}{\Large $\bullet$}}
	\rput(5,8.65){\rnode{e1}{\Large $\bullet$}}
	\rput(10,0){\rnode{e2}{\Large $\bullet$}}
	\rput(2.5,0){\rnode{2}{}}
	\rput(5,0){\rnode{4}{\Huge $\bullet$}}
	\rput(7.5,0){\rnode{6}{}}
	\rput(5,2){\rnode{9}{}}
	\rput(5,4){\rnode{11}{}}
	\rput(2.5,2){\rnode{13}{}}
	\rput(7.5,2){\rnode{16}{}}
	\ncline{e1}{e2}\ncline{e2}{e3}\ncline{e3}{e1}
	\ncline{4}{e1}\ncline{e3}{11}\ncline{11}{e2}
	\ncline{12}{17}\ncline{13}{16}\ncline{14}{15}
	\ncline{1}{10}\ncline{2}{9}
	\ncline{10}{7}\ncline{9}{6}
	\ncline{1}{12}\ncline{2}{13}\ncline{3}{14}
	\ncline{5}{15}\ncline{6}{16}\ncline{7}{17}
	\rput(-0.45,-0.45){\Huge $e_1$}
	\rput(10.45,-0.45){\Huge $e_2$}
	\rput(5,9.1){\Huge $e_3$}
	\rput(4.75,-0.75){\Huge $P$}
	\psline[arrowsize=10pt]{->}(5,0)(7.5,0)\rput(6.25,0.35){\huge $v_2$}
	\psline[arrowsize=10pt]{->}(5,0)(5,2)\rput(4.65,0.9){\huge $v_1$}
	}}
\rput(6,0){
	\scalebox{0.35}{
	\rput(0,0){\rnode{e3}{\Large $\bullet$}}
	\rput(5,8.65){\rnode{e1}{\Large $\bullet$}}
	\rput(10,0){\rnode{e2}{\Large $\bullet$}}
	\rput(2.5,0){\rnode{2}{}}
	\rput(3.75,0){\rnode{3}{}}
	\rput(5,0){\rnode{4}{\Huge $\bullet$}}
	\rput(6.25,0){\rnode{5}{}}
	\rput(7.5,0){\rnode{6}{}}
	\rput(8.75,0){\rnode{7}{}}
	\rput(5,5){\rnode{8}{}}
	\rput(3.75,4.325){\rnode{9}{}}
	\rput(6.25,4.325){\rnode{10}{}}	
	\ncline{e1}{e2}\ncline{e2}{e3}\ncline{e3}{e1}
	\ncline{2}{e1}\ncline{3}{8}\ncline{4}{e1}
	\ncline{5}{8}\ncline{6}{e1}
	\ncline{3}{9}\ncline{8}{9}
	\ncline{8}{10}\ncline{5}{10}
	\rput(-0.45,-0.45){\Huge $e_1$}
	\rput(10.45,-0.45){\Huge $e_2$}
	\rput(5,9.1){\Huge $e_3$}
	\rput(4.75,-0.75){\Huge $P$}
	\psline[arrowsize=10pt]{->}(5,0)(6.25,0)\rput(5.5,0.35){\huge $v_2$}
	\psline[arrowsize=10pt]{->}(5,0)(5,5)\rput(4.65,2){\huge $v_1$}
	}}
	
\end{pspicture}
\caption{The two possible regular tesselations around $P=\frac{1}{2}(1,1,0)$ and the generators $v_1$, $v_2$ of the lattice $\Z^2_\Delta$.}
\label{CaseNZ2}
\end{center}
\end{figure}

This implies that the lattice is symmetric with respect to the diagonal line $L$ and in particular the regular triangles around $P$ have the same side. Then it follows that the continued fraction $\frac{2r}{a}$ at the vertex $e_3$ has to be symmetric with respect to the middle entry, and the boundary of the Newton polygon must contain a lattice point in $L$. This implies that the expansion of $\frac{2r}{2r-a}$ is also symmetric thus if $x^iy^j\in(\Z^2_\Delta)^\vee$ then $x^jy^i\in(\Z^2_\Delta)^\vee$ for any $i$ and $j$. In other words, we have the condition $a^2\equiv 1$ (mod $2r$) in order to have such symmetric Newton polygon. In addition, the vectors $v_1$ and $v_2$ being a basis of $\Z^2_\Delta$ imply that $a^2\equiv1$ (mod $4r$), so the groups $G$ we are looking for are precisely $G\cong\frac{1}{2r}(1,a,-a-1)$ with $a^2\equiv1$ (mod $2n$).

Conversely, if the group is of the form $G\cong\frac{1}{2r}(1,a,-a-1)$ with $a^2\equiv1$ (mod $4r$), then $\{v_1,v_2\}$ and $\{v_1,-v_2\}$ form a basis of the lattice of $\Delta$. It follows that the line $L$ has to be part of the triangulation and again the distribution of points along $\Delta$ is symmetric with respect to $L$, with regular triangles in $\Delta_1$ and $\Delta_2$ around $P$. The continued fractions at the vertices $e_1$ and $e_2$ are the same, and the one for the vertex $e_3$ is symmetric with respect to the middle term, so by the Craw-Reid method the triangulation of $\Hilb{G}{\C^3}$ and $\Hilb{G/N}{\Hilb{N}{\C^3}}$ are the same, and the result follows.
\end{proof}

\begin{exa} In this example we show a group $G$ with two subgroups, with one of them there is an isomorphism as moduli spaces (hence as varieties) between $\Hilb{G}{\C^3}$ and $\Hilb{G/N}{Y_1}$, and with the other an isomorphism as varieties but in different chamber. 

Let $G=\frac{1}{6}(1,1,4)$. By taking $N=\frac{1}{2}(1,1,0)$ and let $Y_1:=\Hilb{N}{\C^3}$. Then we have an isomorphism $\Hilb{G}{\C^3}\cong\Hilb{G/N}{Y_1}$ as moduli spaces of representations of the McKay quiver. 

{\renewcommand{\arraystretch}{1.25}
\begin{figure}[h]
\begin{center}
\begin{pspicture}(0,0)(1,1.25)
	\psset{arcangle=20,nodesep=2pt}
\rput(0,0.5){
$\theta := 
\begin{tabular}{|c|c|}
\cline{1-2}
$\theta_0$	& $\theta_1$	 \\
\cline{1-2}
$\theta_2$	& $\theta_3$	 \\
\cline{1-2}
$\theta_4$	& $\theta_5$	 \\
\cline{1-2}
\end{tabular}
=
\begin{tabular}{|c|c|}
\cline{1-2}
$-a-\varepsilon(b_1+b_2)$	& $a$	 \\
\cline{1-2}
$-a+\varepsilon b_1$	& $a$	 \\
\cline{1-2}
$-a+\varepsilon b_2$	& $a$	 \\
\cline{1-2}
\end{tabular}$}
\end{pspicture}
\end{center}
\caption{Stability condition for $\Hilb{\Z/3}{\Hilb{\Z/2}{\C^3}}$.}
\label{StabZ6}
\end{figure}}

The crepant resolution $Y_1$ is covered by $U_1\cong\C^3_{y^2,z,x/y}$ and $U_2\cong\C^3_{x^2,z,y/x}$, and $\tau_1^{-1}(0)$ consists of the single $\PP^1$ joining the two toric fixed points with coordinates $(x:y)$. In every open set the action of $G/N$ is isomorphic to $\frac{1}{3}(1,2,0)$, and we have that $\tau_1^{-1}(0)=Y_1^{G/N}$. The corresponding $G$-constellations shown in the Figure below. The chamber $C$ for $\Hilb{G}{\C^3}$ is given by the inequalities
\[
\theta_3>0, \theta_5>0, \theta_2+\theta_3>0, \theta_4+\theta_5>0, \theta_1+\theta_3+\theta_5>0, \theta_0<0, \theta_0+\theta_4<0
\]

\begin{figure}[ht]
\begin{center}
\begin{pspicture}(0,0)(10,2.25)
	\psset{arcangle=15}
\rput(0,-0.25){
\scalebox{0.3}{
	% Points of the triangulation
	\rput(5,8.65){\rnode{x}{\Huge $\bullet$}}
	\rput(0,0){\rnode{z}{\Huge $\bullet$}}
	\rput(10,0){\rnode{y}{\Huge $\bullet$}}
	\rput(2.5,1.44){\rnode{1}{}}
	\rput(5,2.89){\rnode{2}{}}
	\rput(7.5,4.32){\rnode{3}{\Huge $\bullet$}}
	% lines
	\ncline{-}{x}{y}\ncline{-}{y}{z}\ncline{-}{z}{x}  % traingle
	\ncline{-}{z}{3}
	}}
\rput(1.25,1){$U_1$}
\rput(1.75,0.25){$U_2$}
\psline{->}(3,1)(4,1)	
\rput(4,-0.25){
\scalebox{0.3}{
	% Points of the triangulation
	\rput(5,8.65){\rnode{x}{\Huge $\bullet$}}
	\rput(0,0){\rnode{z}{\Huge $\bullet$}}
	\rput(10,0){\rnode{y}{\Huge $\bullet$}}
	\rput(2.5,1.44){\rnode{1}{\Huge $\bullet$}}
	\rput(5,2.89){\rnode{2}{\Huge $\bullet$}}
	\rput(7.5,4.32){\rnode{3}{\Huge $\bullet$}}
	% lines
	\ncline{-}{x}{y}\ncline{-}{y}{z}\ncline{-}{z}{x}  % traingle
	\ncline{-}{z}{3}\ncline{-}{x}{1}\ncline{-}{x}{2}
	\ncline{-}{y}{1}\ncline{-}{y}{2}
	}}
{\footnotesize{	
\rput(4.65,0.5){$U_{11}$}
\rput(5.3,0.85){$U_{12}$}
\rput(5.85,1.25){$U_{13}$}
\rput(5,-0.05){$U_{21}$}
\rput(5.65,0.25){$U_{22}$}
\rput(6.25,0.5){$U_{23}$}
}}

\rput(9.5,1){
\scalebox{0.85}{
	${\renewcommand{\arraystretch}{1.25}
	\begin{array}{|c|c|}
		\hline
		U_{11} & \{1,y,y^2,y^3,y^4,y^5\} \\
		U_{12} & \{1,y,y^2,y^3,z,yz\} \\
		U_{13} & \{1,y,z,yz,z^2,y^2\} \\
		U_{21} & \{1,x,x^2,x^3,x^4,x^5\} \\
		U_{22} & \{1,x,x^2,x^3,z,xz\} \\
		U_{23} & \{1,y,z,xz,z^2,xz^2\} \\
		\hline
	\end{array}}$
	}}
\end{pspicture}
\end{center}
\caption{Fan for $\Hilb{\Z/3}{\Hilb{\Z/2}{\C^3}}$ and the corresponding $G$-constellations.}
\end{figure}

For $\Hilb{G/N}{\Hilb{N}{\C^3}}$, according to Definition \ref{theta} the stability is given in Figure \ref{StabZ6} where $a,b\in\Q$ and $0<\varepsilon<<1$, which is contained in $C$. \\

Now take $N=\frac{1}{3}(1,1,1)$. Then we have that $\Hilb{G}{\C^3}\cong\Hilb{G/N}{\Hilb{N}{\C^3}}$ are isomorphic as varieties but in different chambers. In this case $\Hilb{N}{\C^3}$ has 3 open sets 
\[V_1\cong\C^3_{y^3,\frac{z}{y},\frac{x}{y}}, V_2\cong\C^3_{x^3,\frac{z}{x},\frac{y}{x}}, V_3\cong\C^3_{\frac{y}{z},\frac{x}{z},z^3},\]
and $G/N\cong\frac{1}{2}(1,1,0)$ in every open set. Then we obtain the distinguished $G$-constellations shown below. In this case the stability condition is given in Figure \ref{StabZ6-2}, where in particular we have that $\theta_3<0$ so the chamber is different than $C$.

{\renewcommand{\arraystretch}{1.25}
\begin{figure}[h]
\begin{center}
\begin{pspicture}(0,0)(1,1.25)
	\psset{arcangle=20,nodesep=2pt}
\rput(0,0.5){
$\theta := 
\begin{tabular}{|c|c|c|}
\cline{1-3}
$\theta_0$	& $\theta_1$	& $\theta_2$	 \\
\cline{1-3}
$\theta_3$	& $\theta_4$	& $\theta_5$	 \\
\cline{1-3}
\end{tabular}
=
\begin{tabular}{|c|c|c|}
\cline{1-3}
$-a_1-a_2-\varepsilon b$	& $a_1$	& $a_2$	 \\
\cline{1-3}
$-a_1-a_2+\varepsilon b$	& $a_1$	& $a_2$	 \\
\cline{1-3}
\end{tabular}$}
\end{pspicture}
\end{center}
\caption{Stability condition for $\Hilb{\Z/2}{\Hilb{\Z/3}{\C^3}}$.}
\label{StabZ6-2}
\end{figure}}

\begin{figure}[ht]
\begin{center}
\begin{pspicture}(0,0)(10,2.5)
	\psset{arcangle=15}
\rput(0,-0.25){
\scalebox{0.3}{
	% Points of the triangulation
	\rput(5,8.65){\rnode{x}{\Huge $\bullet$}}
	\rput(0,0){\rnode{z}{\Huge $\bullet$}}
	\rput(10,0){\rnode{y}{\Huge $\bullet$}}
	\rput(2.5,1.44){\rnode{1}{}}
	\rput(5,2.89){\rnode{2}{\Huge $\bullet$}}
	\rput(7.5,4.32){\rnode{3}{}}
	% lines
	\ncline{-}{x}{y}\ncline{-}{y}{z}\ncline{-}{z}{x}  % traingle
	\ncline{-}{x}{2}\ncline{-}{y}{2}\ncline{-}{z}{2}
	}}
\rput(1.2,1){$V_1$}
\rput(1.95,1){$V_3$}
\rput(1.65,0.15){$V_2$}
\psline{->}(3,1)(4,1)	
\rput(4,-0.25){
\scalebox{0.3}{
	% Points of the triangulation
	\rput(5,8.65){\rnode{x}{\Huge $\bullet$}}
	\rput(0,0){\rnode{z}{\Huge $\bullet$}}
	\rput(10,0){\rnode{y}{\Huge $\bullet$}}
	\rput(2.5,1.44){\rnode{1}{\Huge $\bullet$}}
	\rput(5,2.89){\rnode{2}{\Huge $\bullet$}}
	\rput(7.5,4.32){\rnode{3}{\Huge $\bullet$}}
	% lines
	\ncline{-}{x}{y}\ncline{-}{y}{z}\ncline{-}{z}{x}  % traingle
	\ncline{-}{z}{3}\ncline{-}{x}{1}\ncline{-}{x}{2}
	\ncline{-}{y}{1}\ncline{-}{y}{2}
	}}
{\footnotesize{	
\rput(4.65,0.5){$V_{11}$}
\rput(5.3,0.85){$V_{12}$}
\rput(5.85,1.25){$V_{31}$}
\rput(5,-0.05){$V_{21}$}
\rput(5.65,0.25){$V_{22}$}
\rput(6.25,0.5){$V_{33}$}
}}

\rput(9.5,1){
\scalebox{0.85}{
	${\renewcommand{\arraystretch}{1.25}
	\begin{array}{|c|c|}
		\hline
		V_{11} & \{1,y,y^2,y^3,y^4,y^5\} \\
		V_{12} & \{1,y,y^2,\frac{z}{y},z,yz\} \\
		V_{21} & \{1,x,x^2,x^3,x^4,x^5\} \\
		V_{22} & \{1,x,x^2,\frac{z}{x},z,xz\} \\
		
		V_{31} & \{1,z,z^2,\frac{x}{z},x,xz\} \\
		V_{32} & \{1,z,z^2,\frac{y}{z},y,yz\} \\
		\hline
	\end{array}}$
	}}
\end{pspicture}
\end{center}
\caption{Fan for $\Hilb{\Z/2}{\Hilb{\Z/3}{\C^3}}$ and the corresponding $G$-constellations.}
\end{figure}
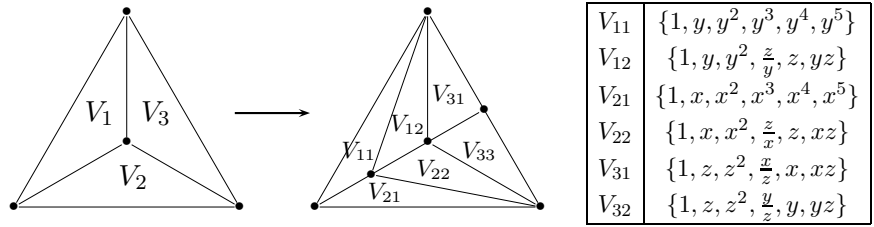

\end{exa}

\subsubsection{Some non-Abelian subgroups of $\SL(3,\C)$}

In this section we present some non-Abelian subgroups $G\subset\SL(3,\C)$ for which $\Hilb{G/N}{\Hilb{N}{\C^3}}$ and $\Hilb{G}{\C^3}$ are non-isomorphic, proving Theorem \ref{thm:isoDim3} (iii).

By the calculations presented in section \ref{Examples} we know that if $G\subset\SO(3)$ is of type $D_{2n}$ or $G_{12}$ and $N$ is the maximal Abelian subgroup of $G$ then the fibre over $0\in\C^3/G$ is different in both spaces. Therefore we conclude that $\Hilb{G/N}{\Hilb{N}{\C^3}}$ is not isomorphic to $\Hilb{G}{\C^3}$.\\

Now consider $G\subset\GL(2,\C)$ to be a small non-Abelian subgroup and let $N=G\cap\SL(2,\C)$. Then we can embed $G$ into $\SL(3, \C)$ to form a non-Abelian intransitive subgroup (c.f.\ type $B$ in \cite{YY}) and we obtain the following result as a consequence of Proposition \ref{prop:non-min}.

\begin{col}\label{col:non-iso}
If $G\subset\SL(3,\C)$ is a non-Abelian small intransitive subgroup and $N=G\cap\SL(2,\C)$ then $\Hilb{G/N}{\Hilb{N}{\C^3}}$ is not isomorphic to $\Hilb{G}{\C^3}$.
\end{col}

\begin{proof}
Both sides are crepant resolutions of $\C^3/G$ and the proper transforms
of $\C^2/G\subset \C^3/G$ to them are $\Hilb{G/N}{\Hilb{N}{\C^2}}$ and
$\Hilb{G}{\C^2}$ respectively.
The former is not a minimal resolution by Proposition \ref{prop:non-min} but the latter is minimal. This implies that the two crepant resolutions are different.
\end{proof}

\bibliographystyle{alpha}

\begin{thebibliography}{NdCS11}

\bibitem[BKR]{BKR}
Tom Bridgeland, Alastair King, and Miles Reid.
\newblock The {M}c{K}ay correspondence as an equivalence of derived categories.
\newblock {\em J. Amer. Math. Soc.}, 14(3):535--554 (electronic), 2001.

%\bibitem[Bri]{Brid02}
%Tom Bridgeland.
%\newblock Flops and derived categories.
%\newblock {\em Invent. Math.}, 147(3):613--632, 2002.

\bibitem[BSW]{BSW}
Raf Bocklandt, Travis Schedler, and Michael Wemyss.
\newblock Superpotentials and higher order derivations.
\newblock {\em J. Pure Appl. Algebra}, 214(9):1501--1522, 2010.

\bibitem[Cra01]{Craw01}
Alastair Craw.
\newblock {\em The McKay correspondence and representations of the McKay
  quiver}.
\newblock PhD thesis, Warwick, 2001.

\bibitem[Cra05]{Craw05}
Alastair Craw.
\newblock An explicit construction of the McKay correspondence for
  $A$-Hilb$\mathbb{C}^3$.
\newblock {\em J. Algebra}, 285(2):682--705, 2005.

\bibitem[CI]{CI}
Alastair Craw and Akira Ishii.
\newblock Flops of {$G$}-{H}ilb and equivalences of derived categories by
  variation of {GIT} quotient.
\newblock {\em Duke Math. J.}, 124(2):259--307, 2004.

\bibitem[CR]{CR02}
Alastair Craw and Miles Reid.
\newblock How to calculate {$A$}-{H}ilb {$\mathbb C\sp 3$}.
\newblock In {\em Geometry of toric varieties}, volume~6 of {\em S\'emin.
  Congr.}, pages 129--154. Soc. Math. France, Paris, 2002.

\bibitem[Dem]{Dem}
Laurent Demonet.
\newblock Skew group algebras of path algebras and preprojective algebras.
\newblock {\em J. Algebra}, 323(4):1052--1059, 2010.

\bibitem[GNS]{GNS2}
Yasushi Gomi, Iku Nakamura, and Ken-ichi Shinoda.
\newblock Coinvariant algebras of finite subgroups of {${\rm SL}(3,{\bf C})$}.
\newblock {\em Canad. J. Math.}, 56(3):495--528, 2004.

\bibitem[Ish]{Ish04}
Akira Ishii.
\newblock Representation moduli of the {M}c{K}ay quiver for finite Abelian
  subgroups of {${\rm SL}(3,\Bbb C)$}.
\newblock In {\em Strings and geometry}, volume~3 of {\em Clay Math. Proc.},
  pages 227--237. Amer. Math. Soc., Providence, RI, 2004.
  
 \bibitem[IU]{IU}
Akira Ishii and Kazushi Ueda.
\newblock The special {M}c{K}ay correspondence and exceptional collection.
\newblock {\em arXiv:1104.2381}, April 2011.

\bibitem[Ito94]{Ito94}
Yukari Ito.
\newblock Crepant resolution of trihedral singularities.
\newblock {\em Proc. Japan Acad. Ser. A Math. Sci.}, 70(5):131--136, 1994.

\bibitem[Ito95]{Ito95b}
Yukari Ito.
\newblock Crepant resolution of trihedral singularities and the orbifold
  {E}uler characteristic.
\newblock {\em Internat. J. Math.}, 6(1):33--43, 1995.
  
\bibitem[IN]{IN00}
Yukari Ito and Hiraku Nakajima.
\newblock Mc{K}ay correspondence and {H}ilbert schemes in dimension three.
\newblock {\em Topology}, 39(6):1155--1191, 2000.

\bibitem[Kin]{King}
Alastair King.
\newblock Moduli of representations of finite-dimensional algebras.
\newblock {\em Quart. J. Math. Oxford Ser. (2)}, 45(180):515--530, 1994.

\bibitem[Nak]{Nak01}
Iku Nakamura.
\newblock Hilbert schemes of Abelian group orbits.
\newblock {\em J. Algebraic Geom.}, 10(4):757--779, 2001.

\bibitem[NdC]{NdC2}
Alvaro Nolla~de Celis.
\newblock Dihedral ${G}$-{H}ilb via representations of the {M}c{K}ay quiver.
\newblock {\em Proc. Japan Acad. Ser. A}, 88(5):78--83, 2012.

\bibitem[NdCS]{NdCS}
Alvaro Nolla~de Celis and Yuhi Sekiya.
\newblock Flops and mutations for crepant resolutions of polyhedral
  singularities.
\newblock  {\em arXiv:1108.2352}, August 2011.

\bibitem[Serre]{Serre}
Jean-Pierre Serre.
\newblock Linear Representations of Finite Groups.
\newblock {\em Springer-Verlag}, 1977

\bibitem[Tha]{Tha}
Michael Thaddeus.
\newblock Geometric invariant theory and flips.
\newblock {\em J. Amer. Math. Soc.}, 9(3):691--723, 1996.

\bibitem[YY]{YY}
Stephen S.-T. Yau and Yung Yu.
\newblock Gorenstein quotient singularities in dimension three.
\newblock {\em Mem. Amer. Math. Soc.}, 105(505):viii+88, 1993.

\end{thebibliography}

\end{document}